\documentclass{amsart}
\usepackage{amsfonts,amscd,amsthm,amsgen,amsmath,amssymb}
\usepackage{float}
\usepackage[all]{xy}
\usepackage{epsfig,color}
\usepackage{tikz}
\usetikzlibrary{calc}
\usepackage[vcentermath]{youngtab}

\newtheorem{theorem}{Theorem}[section]
\newtheorem{lemma}[theorem]{Lemma}
\newtheorem{corollary}[theorem]{Corollary}
\newtheorem{proposition}[theorem]{Proposition}

\theoremstyle{definition}
\newtheorem{definition}[theorem]{Definition}
\newtheorem{example}[theorem]{Example}

\theoremstyle{remark}
\newtheorem{remark}[theorem]{Remark}

\numberwithin{equation}{section}

\begin{document}

\setlength\parskip{0.5em plus 0.1em minus 0.2em}

\title{New invariants of involutions from Seiberg--Witten Floer theory}

\author{David Baraglia}
\address{School of Mathematical Sciences, The University of Adelaide, Adelaide SA 5005, Australia}
\email{david.baraglia@adelaide.edu.au}

\author{Pedram Hekmati}
\address{Department of Mathematics, The University of Auckland, Auckland, 1010, New Zealand}
\email{p.hekmati@auckland.ac.nz}


\date{\today}

\begin{abstract}
We study equivariant Seiberg--Witten Floer theory of rational homology $3$-spheres in the special case where the group action is given by an involution. The case of involutions deserves special attention because we can couple the involution to the charge conjugation symmetry of Seiberg--Witten theory. This leads to new Floer-theoretic invariants which we study and apply in a variety of applications. In particular, we construct a series of delta-invariants $\delta^E_*, \delta^R_*, \delta^S_*$ which are the equivariant equivalents of the Ozsv\'ath--Szab\'o $d$-invariant. The delta-invariants come in three types: equivariant, Real and spin depending on the type of the spin$^c$-structure involved. The delta-invariants satisfy many useful properties, including a Fr{\o}yshov-type inequality for equivariant cobordisms. We compute the delta-invariants in a wide range of examples including: equivariant plumbings, branched double covers of knots and equivariant Dehn surgery. We also consider various applications including obstructions to extending involutions over bounding $4$-manifolds, non-smoothable involutions on $4$-manifolds with boundary, equivariant embeddings of $3$-manifolds in $4$-manifolds and non-orientable surfaces bounding knots.
\end{abstract}

\maketitle


\section{Introduction}

In the paper \cite{bh} we introduced and studied an equivariant version of Seiberg--Witten Floer theory for group actions on rational homology $3$-spheres. In \cite{bh2}, we applied this theory to cyclic group actions on Brieskorn spheres, obtaining obstructions for such actions to extend over a bounding $4$-manifold. In this paper we focus specifically on the case that the group action is given by an involution. Our motivation for considering involutions is that there are special features of the involutive case which makes it possible to define a new set of invariants. We demonstrate the utility of these invariants with a variety of applications.

Let $Y$ be a rational homology $3$-sphere and $\sigma\colon Y \to Y$ an orientation preserving involution. Our Floer-theoretic invariants of $(Y , \sigma)$ will come in three variations: {\em equivariant (E), Real (R)} and {\em odd spin (S)}. Corresponding to these three types, we have three different notions of compatibility between $\sigma$ and a spin$^c$-structure $\mathfrak{s}$:
\begin{itemize}
\item{$\mathfrak{s}$ is {\em equivariant} if $\sigma^*(\mathfrak{s})$ is isomorphic to $\mathfrak{s}$. In this case $\sigma$ can be lifted to a $\mathbb{C}$-linear involution on the spinor bundle corresponding to $\mathfrak{s}$.}
\item{$\mathfrak{s}$ is {\em Real} if $\sigma$ can be lifted to an antilinear involution on the spinor bundle corresponding to $\mathfrak{s}$.}
\item{$\mathfrak{s}$ is {\em odd spin} if $\mathfrak{s}$ is the spin$^c$-structure underlying a spin structure and $\sigma$ admits a lift to the spin bundle which squares to $-1$.}
\end{itemize}

The equivariant case is precisely what was considered in \cite{bh}. The Real and odd spin cases yield new invariants. For each of the three types $T \in \{ E , R , S\}$ we construct a corresponding Seiberg--Witten Floer cohomology theory $HSW_{T}^*( Y , \mathfrak{s} , \sigma)$ which is a graded module over the equivariant cohomology ring $H^*_{G^T_\mathfrak{s}}(pt)$, where $G^T_\mathfrak{s}$ is a certain group depending on the type of spin$^c$-structure. In type $E$, $G^E_{\mathfrak{s}} \cong S^1 \times \mathbb{Z}_2$, in type $R$, $G^R_{\mathfrak{s}} \cong O(2)$ and in type $S$, $G^S_{\mathfrak{s}}$ is a certain extension of $\mathbb{Z}_2$ by $Pin(2)$, see Section \ref{sec:sym} for details.

As one might expect, the Floer cohomology groups are typically difficult to compute. However we use the Floer cohomology to derive a collection of numerical invariants which are much easier to work with. We call these the {\em delta-invariants} of $(Y,\mathfrak{s} , \sigma)$ of type $E,R$, or $S$. These invariants play the role of the Ozsv\'ath--Szab\'o correction term $d(Y,\mathfrak{s})$ in the presence of an involution. For convenience we define the (ordinary) delta-invariant of $(Y,\mathfrak{s})$ to be half the $d$-invariant: $\delta(Y , \mathfrak{s}) = d(Y , \mathfrak{s})/2$.

\begin{itemize}
\item{In types $T = E$ or $R$, the delta-invariants are a sequence $\delta_j^T( Y , \mathfrak{s} , \sigma) \in \mathbb{Q}$ of rational numbers, where $j \ge 0$ is an integer.}
\item{In type $S$, the delta-invariants $\delta^S_{i,j}(Y , \mathfrak{s} , \sigma) \in \mathbb{Q}$ depend on two integers $i,j \ge 0$ subject to the condition that either $i=0$ or $j \le 1$.}
\end{itemize}

Throughout the paper our Floer cohomology groups will be defined with respect to the coefficient ring $\mathbb{F} = \mathbb{Z}/2\mathbb{Z}$. This is not strictly necessary and is chosen for convenience.

The delta-invariants $\delta^T_*(Y , \mathfrak{s} , \sigma)$ are equivariant rational homology cobordism invariants (see Remark \ref{rem:erhci} for the precise statement). They satisfy a long list of properties which we summarise below. For notational convenience we write $\delta^T_*$ to denote either $\delta^T_j$ for some $j\ge 0$ if $T=E,R$, or $\delta^S_{i,j}$ for some $i,j$ if $T=S$.

\begin{theorem}
The $\delta$-invariants satisfy the following properties:
\begin{itemize}
\item[(1)]{$\delta^T_*(Y  , \mathfrak{s} , \sigma) = \delta(Y , \mathfrak{s}) \; ({\rm mod} \; \mathbb{Z})$ for any $T$.}
\item[(2)]{$\delta^S_*(Y  , \mathfrak{s} , \sigma) = \mu(Y , \mathfrak{s}) \; ({\rm mod} \; 2\mathbb{Z})$ where $\mu(Y , \mathfrak{s})$ is the generalised Rokhlin invariant.}
\item[(3)]{For $T=E,R$, the sequence $\delta^T_j(Y , \mathfrak{s} , \sigma)$ is decreasing and is eventually constant.}
\item[(4)]{$\delta^S_{i',j'}(Y , \mathfrak{s} , \sigma ) \le \delta^S_{i,j}(Y , \mathfrak{s} , \sigma )$ whenever $i' \ge i$, $j' \ge j$. The values of $\delta^S_{0,j}(Y , \mathfrak{s} , \sigma )$, $ \delta^S_{j,0}(Y , \mathfrak{s} , \sigma )$ and $\delta^S_{j,1}(Y , \mathfrak{s} , \sigma )$ are independent of $j$ for large enough $j$.}
\item[(5)]{$\delta^T_*(Y , \mathfrak{s} , \sigma ) + \delta^T_*(-Y , \mathfrak{s} , \sigma ) \ge 0$.}
\item[(6)]{If $Y$ is an $L$-space, then $\delta^T_*(Y , \mathfrak{s} , \sigma) = \delta(Y,\mathfrak{s})$ for all $T,*$.}
\item[(7)]{If $T = E,R$, $\delta_0^T(Y , \mathfrak{s} , \sigma) \ge \delta(Y , \mathfrak{s})$.}
\item[(8)]{
\begin{itemize}
\item[(i)]{$\delta^S_{0,0}(Y , \mathfrak{s} , \sigma) \ge \alpha(Y , \mathfrak{s})$,}
\item[(ii)]{$\delta^S_{i,j}(Y , \mathfrak{s} , \sigma) \ge \beta(Y , \mathfrak{s})$ for $i+j=1$,}
\item[(iii)]{$\delta^S_{i,j}(Y,\mathfrak{s},\sigma) \ge \gamma(Y,\mathfrak{s})$ for $i+j=2$,}
\end{itemize}
where $\alpha, \beta, \gamma$ are the invariants defined in \cite{man2}.}
\item[(9)]{$\delta^T_{*_1+*_2}(Y_1 \# Y_2 , \mathfrak{s} , \sigma ) \le \delta^T_{*_1}(Y_1 , \mathfrak{s}_1 , \sigma_1 ) + \delta^T_{*_2}(Y_2 , \mathfrak{s}_2 , \sigma_2 )$.}
\item[(10)]{If $T = E,R$, then $\delta^T_j(Y,\mathfrak{s},\sigma) \ge (l(Y,\mathfrak{s}) - j)/2$ where $l(Y,\mathfrak{s})$ is the lowest degree in which $HF^+_*(Y,\mathfrak{s})$ is non-zero. Similarly $\delta^S_{i,j}(Y,\mathfrak{s},\sigma) \ge (l(Y,\mathfrak{s}) - i-j)/2$.}
\end{itemize}

\end{theorem}

In addition to the above properties, there is a spectral sequence (Theorem \ref{thm:ss}) which relates $HSW^*_T$ to $HSW^*$. This can sometimes be used to compute the delta-invariants.

By far the most important property of the delta-invariants is their behaviour under equivariant cobordism, namely they satisfy an equivariant version of the Fr{\o}yshov inequality.

\begin{theorem}
Let $W$ be a smooth, compact, oriented $4$-manifold with boundary and with $b_1(W) = 0$. Suppose that $\sigma$ is an orientation preserving smooth involution on $W$. Let $\mathfrak{s}$ be a spin$^c$-structure on $W$ of type $T \in \{E,R,S\}$. Suppose each component of $\partial W$ is a rational homology sphere and that $\sigma$ sends each component to itself. 

Suppose that $\partial W = Y_1 \cup -Y_0$. Then:
\begin{itemize}
\item[(1)]{In type $T=E$, suppose that the $\sigma$-invariant subspace of $H^2(W ; \mathbb{R})$ is negative definite. Then
\[
\delta^E_{j+b_+(W)}(Y_0 , \mathfrak{s}|_{Y_0} , \sigma|_{Y_0}) + \frac{ c(\mathfrak{s})^2 - \sigma(W) }{8} \le  \delta^E_j(Y_1, \mathfrak{s}|_{Y_1} , \sigma|_{Y_1})
\]
for all $j \ge 0$.
}
\item[(2)]{In type $T=R$, suppose that the $\sigma$-anti-invariant subspace of $H^2(W ; \mathbb{R})$ is negative definite. Then
\[
\delta^R_{j+b_+(W)}(Y_0 , \mathfrak{s}|_{Y_0} , \sigma|_{Y_0}) + \frac{ c(\mathfrak{s})^2 - \sigma(W) }{8} \le 
\delta^R_j(Y_1, \mathfrak{s}|_{Y_1} , \sigma|_{Y_1})
\]
for all $j \ge 0$.
}
\item[(3)]{In type $T=S$, let $b_+(X)^\sigma, b_+(X)^{-\sigma}$ denote the dimensions of the $\sigma$-invariant/anti-invariant subspaces of $H^+(X)$. Then
\[
\delta^S_{i + b_+(W)^{\sigma} , j + b_+(W)^{-\sigma}}(Y_0 , \mathfrak{s}|_{Y_0} , \sigma|_{Y_0}) - \frac{\sigma(W) }{8} \le 
\delta^S_{i,j}(Y_1, \mathfrak{s}|_{Y_1} , \sigma|_{Y_1})
\]
for all $i,j \ge 0$ such that either $i + b_+(W)^{\sigma} = 0$ or $j + b_+(W)^{-\sigma} \le 1$.
}
\end{itemize}

\end{theorem}

\subsection{Calculations}\label{sec:calc}

One of the strengths of our delta-invariants is that they can be calculated in a wide variety of situations. These calculations are carried out in Section \ref{sec:con}. Here we summarise some of the main results.

{\bf Equivariant plumbing.} Let $\Gamma$ be a plumbing graph and $Y_\Gamma$ the boundary of the plumbing on $\Gamma$ (see \textsection \ref{sec:pg}). We consider two classes of involutions that can be constructed on such plumbings. The first type of involution that we consider is complex conjugation $c_\Gamma$ (\textsection \ref{sec:cc}), obtained by plumbing together the complex conjugation involutions on each disc bundle. Such involutions can be constructed for any plumbing graph. The second kind of involution that we consider will be referred to as {\em $\mathbb{Z}_2$-equivariant plumbing} (\textsection \ref{sec:z2p}) and denoted as $m_\Gamma$. Such involutions can be constructed for plumbings where $\Gamma$ satisfies the following conditions (i) the degrees of all disc bundles in the plumbing are even, (ii) $\Gamma$ is a bipartite graph (vertices can be coloured black or white in an alternating fashion) and (iii) white vertices have at most two edges joining them. If $\Gamma$ satisfies these conditions we will call it a $\mathbb{Z}_2$-equivariant plumbing graph. See Figure \ref{fig:bipartite} for an example.

\begin{figure}
\begin{center}
\begin{tikzpicture}
\draw[thick] (0.1,0) -- (1.9,0) ;
\draw[thick] (2.1,0) -- (3.9,0);
\draw[thick] (4.1,0) -- (5,0);
\draw[thick] (1,-0.9) -- (1,0.9) ;
\draw[thick] (3,0) -- (3,-0.9) ;
\draw[thick] (3,-1.1) -- (3,-2.9) ;
\draw[thick] (3,-2) -- (3.9,-2);
\draw (0,0) circle(0.1);
\draw [fill] (1,0) circle(0.1);
\draw  (2,0) circle(0.1);
\draw [fill] (3,0) circle(0.1);
\draw  (4,0) circle(0.1);
\draw [fill] (5,0) circle(0.1);
\draw (1,1) circle(0.1);
\draw (3,-1) circle(0.1);
\draw [fill] (3,-2) circle(0.1);
\draw (3,-3) circle(0.1);
\draw  (4,-2) circle(0.1);
\draw (1,-1) circle(0.1);
\end{tikzpicture}
\caption{Bipartite graph. White vertices have at most two edges joining them.}\label{fig:bipartite}
\end{center}
\end{figure}
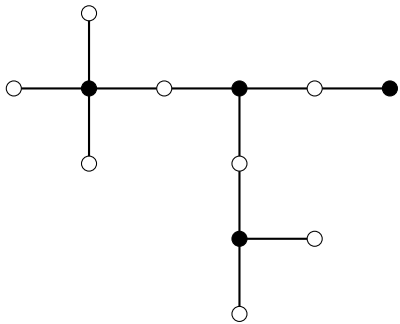

To any plumbing graph $\Gamma$ we define an invariant $j(\Gamma) \in \mathbb{Z}$. The definition is given in Section \ref{sec:pg}. Here we simply note that if $\Gamma$ has $k$ vertices, then $j(\Gamma) \le k$. 

\begin{theorem}\label{thm:plumb0}
Let $\Gamma$ be a connected plumbing graph whose degrees are all even and let $Y_\Gamma$ be the boundary of the plumbing $X_\Gamma$ according to $\Gamma$. Let $\mathfrak{s}$ denote the restriction to $Y_\Gamma$ of the unique spin$^c$-structure on $X_\Gamma$ with $c(\mathfrak{s}) = 0$.
\begin{itemize}
\item[(1)]{$\delta_j^E( Y_\Gamma , \mathfrak{s} , c_\Gamma ) = \delta^S_{0,j}( Y_\Gamma , \mathfrak{s} , c_\Gamma) = - \overline{\mu}(Y,\mathfrak{s}|_Y)$ for all $j \ge j(\Gamma)$.}
\item[(2)]{Suppose that $\Gamma$ is a $\mathbb{Z}_2$-equivariant plumbing graph. Then $\delta_j^R(Y_\Gamma , \mathfrak{s} , m_\Gamma) = \delta^S_{j,k}( Y_\Gamma , \mathfrak{s} , m_\Gamma) = -\overline{\mu}(Y , \mathfrak{s}|_Y)$ for all $j \ge j(\Gamma)$ and $k = 0,1$.}
\end{itemize}
\end{theorem}

In this theorem, $\overline{\mu}(Y , \mathfrak{s}|_Y)$ is the generalised Neumann--Seibenmann invariant \cite[\textsection 4]{neu}.

A particularly interesting class of plumbed $3$-manifolds are the Brieskorn homology spheres $\Sigma(a_1, \dots , a_n)$. With the exception of $S^3$ and $\Sigma(2,3,5)$, the Brieskorn spheres have a geometry modelled on the universal cover of $SL(2,\mathbb{R})$. They have symmetry group $O(2)$ which combines the circle action of the Seifert fibration with complex conjugation (viewing $\Sigma(a_1, \dots , a_n)$ as the link of a complex singularity). By \cite[Theorem 2.1]{mesc}, any smooth, orientation preserving involution on $\Sigma(a_1, \dots , a_n)$ is conjugate to an involution in $O(2)$. There are precisely two conjugacy classes, $m$, the element of order $2$ in the circle action and $c$, complex conjugation. Concerning the delta-invariants of $(Y,m)$, $(Y,c)$, we have the following:

\begin{theorem}\label{thm:brie0}
Let $Y = \Sigma(a_1, \dots , a_n)$ be a Brieskorn sphere. Then
\begin{itemize}
\item[(1)]{$\delta^E_j(Y , c) = \delta^R_j(Y , m) =  -\overline{\mu}(Y)$ for all $j \ge 1$.}
\item[(2)]{$\delta^E_0(Y,c) = \delta^R_0(Y , m) = \delta(Y)$.}
\item[(3)]{$\delta^E_j(-Y,c) = \delta^R_j(-Y,m) = \overline{\mu}(Y)$ for all $j \ge 0$.}
\item[(4)]{$\delta^S_{j,0}(Y,m) = \delta^S_{j,1}(Y,m) = \delta^S_{0,j}(Y,c) = \delta^S_{j,1}(Y,c) = -\overline{\mu}(Y)$ for $j \ge 1$.}
\item[(5)]{$\delta^S_{j,0}(-Y,m) = \delta^S_{j,1}(-Y,m) = \overline{\mu}(Y)$ for all $j \ge 1$.}
\item[(6)]{$\delta^S_{i,j}(-Y,c) = \overline{\mu}(Y)$ for all $i,j$ with $i=0$ or $j \le 1$.}
\item[(7)]{$\delta^R_j(Y,c) \ge -\overline{\mu}(Y)$ and $\delta^R_j(-Y,c) \le \overline{\mu}(Y)$ for all $j \ge 0$.}
\end{itemize}

\end{theorem}

\begin{remark}
This theorem complements the results in \cite{bh,bh2} in which the delta-invariants $\delta^E_j(\pm Y,m)$ are studied. For instance if $p,q$ are odd and coprime, then $\delta^E_j( \Sigma(2,p,q) , m ) = -\lambda( \Sigma(2,p,q) )$ for all $j \ge 0$, where $\lambda$ is the Casson invariant (\cite[Proposition 7.2]{bh}).
\end{remark}

{\bf Branched double covers.} Let $K$ be a knot in $S^3$ and let $Y = \Sigma_2(K)$ be the double cover of $S^3$ branched over $K$. Then $Y$ is a rational homology sphere and there is a unique spin structure on $Y$ (see \textsection \ref{sec:bdc}). The corresponding spin$^c$-structure will be denoted $\mathfrak{s}_0$. Let $\sigma\colon Y \to Y$ be the covering involution of the branched double cover. By uniqueness, $\mathfrak{s}_0$ is preserved by $\sigma$ and is an odd spin involution. The delta-invariants of $(Y , \mathfrak{s}_0 , \sigma)$ define knot invariants of $K$ as follows.

\begin{definition}
Let $K$ be a knot in $S^3$. We define the delta-invariants $\delta^E_j(K), \delta^R_j(K)$ and $\delta^S_{i,j}(K)$ of $K$ to be the corresponding delta-invariants of $( \Sigma_2(K) , \mathfrak{s}_0 , \sigma )$.
\end{definition}

For a knot $K$ in $S^3$ we let $\sigma(K)$ denote the signature and $g_4(K)$ the smooth slice genus.

\begin{theorem}\label{thm:delK}
Let $K$ be a knot in $S^3$.
\begin{itemize}
\item[(1)]{$\delta^E_j(K), \delta^R_j(K), \delta^S_{k,l}(K)$ depend only on the smooth concordance class of $K$.}
\item[(2)]{$\delta^E_j(K), \delta^R_j(K), \delta^S_{k,l}(K) \in \frac{1}{4}\mathbb{Z}$.}
\item[(3)]{$\delta^E_j(K) = \delta^R_j(K) = -\sigma(K)/8 \; ({\rm mod} \; \mathbb{Z})$. $\delta^S_{k,l}(K) = -\sigma(K)/8 \; ({\rm mod} \; 2\mathbb{Z})$.}
\item[(4)]{$\delta^E_j( K ) = \delta^S_{0,j}(K) = -\sigma(K)/8$ for $j \ge g_4(K) - \sigma(K)/2$.}
\item[(5)]{If $K$ is quasi-alternating, then $\delta^E_j(K) = \delta^R_j( K) = \delta^S_{k,l}( K ) = -\sigma(K)/8$ for all $j \ge 0$ and all $(k,l)$ with $k=0$ or $l \le 1$.}
\item[(6)]{If $g_4(K) = -\sigma(K)/2$, then $\delta^R_\infty(K) \ge -\sigma(K)/8$ and $\delta^R_0( -K ) \le \sigma(K)/8$.}
\item[(7)]{If $g_4(K) = 1 - \sigma(K)/2$, then $\delta^S_{i,j}(K) \ge -\sigma(K)/8$ for all $i,j$ with $i=0$ or $j=0$ and $\delta^S_{0,1}(-K) \le \sigma(K)/8$.} 
\end{itemize}

\end{theorem}

Additionally, the knot invariants $\delta^E_*, \delta^R_*, \delta^S_*$ are sub-additive in the sense that $\delta^E_{i+j}(K_1 \# K_2) \le \delta^E_i(K_1) + \delta^E_j(K_2)$ and similarly for $\delta^R_*, \delta^S_*$ (Proposition \ref{prop:subadd}).

Amongst other things Theorem \ref{thm:delK} says that the delta-invariants for a quasi-alternating knot $K$ are all equal to $-\sigma(K)/8$. Another class of knot where the delta-invariant can be computed are the torus knots $T_{p,q}$. This is because the branched double cover $\Sigma_2(T_{p,q})$ is a Seifert fibre space and the covering involution is the element of order $2$ in the circle action. In particular, when $p,q$ are odd, $\Sigma_2(T_{p,q})$ is the Brieskorn sphere $\Sigma(2,p,q)$ and we can apply Theorem \ref{thm:brie0}. If $p$ or $q$ is even, then $\Sigma_2(T_{p,q})$ is not a Brieskorn sphere, but it is still the boundary of a plumbing on a star-shaped graph $\Gamma$ and the covering involution coincides with $m_\Gamma$, so we can apply Theorem \ref{thm:plumb0}. A third class of knots for which the delta-invariants can be readily computed are Montesinos knots. Let $L = M( b ; (a_1,b_1) , \dots , (a_n,b_n))$ denote a Montesinos link (see \textsection \ref{sec:bdc}). Assume that exactly one of the $a_i$ is even. In this case $L$ is a knot. We have the following:

\begin{theorem}
Let $K = M(b ; (a_1, b_1) , \dots , (a_n , b_n))$ be a Montesinos knot where exactly one $a_i$ is even. Let $e = b - \sum_{i=1}^n b_i/a_i$. If $e > 0$, then $\delta^E_j(K) = -\sigma(K)/8$ for all $j \ge 0$. If $e < 0$, then $\delta^E_j(K) = -\sigma(K)/8$ for all $j \ge 1$ and $\delta^E_0(K) = \delta( \Sigma_2(K) , \mathfrak{s}_0 )$.
\end{theorem}

{\bf Equviariant Dehn surgery.} Let $L$ be a link in $S^3$ and suppose that $L$ is sent to itself under some orientation preserving, smooth involution $\sigma\colon S^3 \to S^3$. Let $\mathcal{F}$ denote a framing of $L$ and let $Y = Y(L , \mathcal{F})$ be obtained from $S^3$ by performing Dehn surgery along $L$ with framing $\mathcal{F}$. Suppose that the framing is $\sigma$-invariant in the sense that for any component $K$ of $L$ which is not sent to itself by $\sigma$, the framings of $K$ and $\sigma(K)$ coincide. Then we can carry out Dehn surgery equivariantly with respect to $\sigma$ and the extension is unique up to conjugacy by diffeomorphisms isotopic to the identity (see \cite[\textsection 2]{sak} for details). We denote the induced involution on $Y$ by $\sigma$.
 
We say that $L$ is {\em $2$-periodic} if $\sigma$ sends each component of $L$ to itself and has no fixed points on $L$. We say that $L$ is {\em strongly invertible} if $\sigma$ sends each component of $L$ to itself orientation reversingly. We will say that the framing $\mathcal{F}$ is {\em even} if all the framing coefficients are even integers. In this case $Y(L , \mathcal{F})$ bounds a spin $4$-manifold $X$, the trace of the surgery on $L$. There is a unique spin$^c$-structure on $X$ which comes from a spin structure. By restriction to the boundary, this defines a spin$^c$-structure on $Y$ which we denote by $\mathfrak{s}_0$. This spin$^c$-structure has type $S$ with respect to $\sigma$. Our first main result concerns the delta-invariants of $(Y(L , \mathcal{F}), \mathfrak{s}_0 )$ in the case of even surgery.

\begin{theorem}\label{thm:surglink0}
Let $(L , \sigma)$ be a $2$-periodic or strongly invertible link. Let $Y$ be the $3$-manifold obtained by Dehn surgery on $Y$ with respect to some framing $\mathcal{F}$ and denote by $\sigma$ the induced involution on $Y$. Suppose that $\mathcal{F}$ is even and let $\mathfrak{s}_0$ denote the distinguished spin structure. Let $A$ denote the linking matrix of $(L , \mathcal{F})$. Then
\begin{itemize}
\item[(1)]{If $L$ is $2$-periodic, then $\delta^R_j(Y , \mathfrak{s}_0 , \sigma) = -\sigma(A)/8$ for $j \ge b_-(A)$.}
\item[(2)]{If $L$ is strongly invertible, then $\delta^E_j(Y , \mathfrak{s}_0 , \sigma) = -\sigma(A)/8$ for $j \ge b_-(A)$.}
\end{itemize}

\end{theorem}

In the case of non-integral surgery coefficients, it is still possible to calculate the delta-invariants in some cases. Consider for example a strongly invertible knot $K$ and let $Y = S_{p/q}(K)$ be the Dehn surgery on $K$ with surgery coefficient $p/q$. Assume that $p$ is odd and $q$ is even. Since $p$ is odd, there exists a unique spin$^c$-structure $\mathfrak{s}$ on $Y$ which comes from a spin structure. Further, there exist even integers $a_0, a_1 , \dots, a_m$ such that $p/q = [a_0 , \dots , a_m]$ (Lemma \ref{lem:even}), where $[a_0 , a_2 , \dots , a_m]$ is the negative continued fraction
\[
[a_0,a_1,\ldots,a_m] =a_0-\frac{1}{a_1-\raisebox{-3mm}{$\ddots$
\raisebox{-2mm}{${-\frac{1}{\displaystyle{a_m}}}$}}}.
\]

\begin{theorem}\label{thm:slamd0}
Let $K$ be a strongly invertible knot and let $Y = S_{p/q}(K)$ where $p$ is odd and $q$ is even and non-zero. Let $p/q = [a_0 , a_1, \dots , a_n]$ where $a_0, \dots , a_n$ are even integers and let $A_{ij}$ be the matrix $A_{ii} = a_i$, $A_{ij} = 1$ for $|i-j| = 1$, $A_{ij} = 0$ for $|i - j|>1$. Then $\delta^E_j( S_{p/q}(K) , \mathfrak{s} , \sigma ) = -\sigma(A)/8$ for $j \ge b_-(A)$.
\end{theorem}

\subsection{Applications}

{\bf Obstructions to extending involutions.} Suppose we are given a rational homology $3$-sphere $Y$ with orientation preserving involution $\sigma$ and a spin$^c$-structure $\mathfrak{s}$ of type $E,R$ or $S$. Suppose that $X$ is a compact, oriented smooth $4$-manifold which bounds $Y$ and that $\mathfrak{s}$ extends to a spin$^c$-structure on $X$. We can use Theorem \ref{thm:froy} to obstruct the existence of an extension of $\sigma$ to an involution on $X$, under some assumptions on how $\sigma$ acts on $H^2(X ; \mathbb{Z})$. 

For simplicity we focus on the case that $Y$ is an integral homology sphere so that there is only one spin$^c$-structure on $Y$, which we omit from the notation. We consider four cases of interest: (1) $X$ is negative definite, (2) $b_+(X) = 1$, (3) homologically trivial involutions, (4) $\sigma$ acts as $-1$ on $H^2(X ; \mathbb{Z})$. These are summarised by the following four propositions:

\begin{proposition}\label{prop:def0}
Let $X$ be a compact, oriented, smooth $4$-manifold with boundary an integral homology sphere $Y$. Assume $H_1(X ; \mathbb{Z}_2) = 0$ and that $X$ is negative definite. Let $\sigma$ be an orientation preserving involution on $X$. Then
\begin{itemize}
\item[(1)]{If $c \in H^2(X ; \mathbb{Z})$ is a characteristic element and $\sigma(c) = c$, then 
\[
\frac{ c^2 + b_2(X) }{8} \le \min\{ \delta_\infty^E(Y , \sigma) , -\delta^E_{0}(-Y,\sigma)\}.
\]
}
\item[(2)]{Assume that the fixed point set of $\sigma$ contains non-isolated points. If $c \in H^2(X ; \mathbb{Z})$ is a characteristic element and $\sigma(c) = -c$, then
\[
\frac{c^2 + b_2(X) }{8} \le \min\{ \delta^R_\infty(Y , \sigma) , -\delta^R_{0}(-Y,\sigma) \}.
\]
}
\item[(3)]{If $X$ is spin and $\sigma$ is odd, then
\[
\frac{b_2(X)}{8} \le \min\{ \delta^S_{0,\infty}(Y , \sigma) , \delta^S_{\infty,1}(Y,\sigma) , -\delta^S_{0,0}(-Y,\sigma) \}.
\]
}
\end{itemize}

\end{proposition}

\begin{proposition}\label{prop:b+10}
Let $X$ be a compact, oriented, smooth spin $4$-manifold with boundary an integral homology sphere $Y$. Assume $H_1(X ; \mathbb{Z}_2) = 0$ and that $b_+(X)=1$. Let $\sigma$ be an odd involution on $X$. Let $H^+(X)$ denote a $\sigma$-invariant maximal positive definite subspace of $H^2(X ; \mathbb{R})$. Then:
\begin{itemize}
\item[(1)]{If $\sigma$ acts trivially on $H^+(X)$, then
\[
-\frac{\sigma(X)}{8} \le \min\{ \delta^S_{\infty,0}(Y,\sigma) , -\delta^S_{1,0}(-Y,\sigma) \}.
\]
}
\item[(2)]{If $\sigma$ acts non-trivially on $H^+(X)$, then
\[
-\frac{\sigma(X)}{8} \le \min\{ \delta^S_{0,\infty}(Y,\sigma) , \delta^S_{\infty,1}(Y,\sigma) , -\delta^S_{0,1}(-Y,\sigma) \}.
\]
}
\end{itemize}
\end{proposition}

\begin{proposition}\label{prop:ht0}
Let $X$ be a compact, oriented, smooth spin $4$-manifold with boundary an integral homology sphere $Y$. Assume $H_1(X ; \mathbb{Z}_2) = 0$. Let $\sigma$ be a smooth odd involution on $X$ which acts homologically trivially on $X$. Then:
\[
\sigma(X) = -8 \delta^R_\infty(Y,\sigma).
\]
Furthermore, we have
\[
b_-(X) \ge j^R(Y,\sigma), \quad b_+(X) \ge j^R(-Y,\sigma).
\]
\end{proposition}

\begin{proposition}\label{prop:hm0}
Let $X$ be a compact, oriented, smooth spin $4$-manifold with boundary an integral homology sphere $Y$. Assume $H_1(X ; \mathbb{Z}_2) = 0$. Let $\sigma$ be a smooth odd involution on $X$ which acts as $-1$ on $H^2(X ; \mathbb{Z})$. Then:
\[
\sigma(X) = -8 \delta^E_\infty(Y,\sigma).
\]
Furthermore, we have
\[
b_-(X) \ge j^E(Y,\sigma), \quad b_+(X) \ge j^E(-Y,\sigma).
\]
\end{proposition}

In particular, Propositions \ref{prop:ht0} and \ref{prop:hm0} can be used to bound $b_{\pm}(X)$ for bounding $4$-manifolds over which $\sigma$ extends under the condition that $\sigma$ acts as $+1$ or $-1$ on $H^2(X ; \mathbb{Z})$.

For example, suppose $Y = \Sigma(a_1 , \dots , a_n)$ is a Brieskorn sphere where $a_i$ is even for some $i$ and that $\sigma = m$. Then $\delta^R_\infty(Y , m) = -\overline{\mu}(Y)$. Let $X$ be a spin manifold bounding $Y$ with $H_1(X ; \mathbb{Z}_2)=0$ and with $\sigma(X) \neq 8\overline{\mu}(Y)$. Then $m$ does not extend to a homologically trivial smooth involution on $X$. On the other hand, $m$ does extend to a smooth homologically trivial diffeomorphism on $X$. This is because the involution $m$ belongs to a circle action on $Y$, hence is smoothly isotopic to the identity. 

This non-extension result contrasts with the fact that $m$ does extend to a homologically trivial involution on the plumbing $X_{\Gamma}$ on a star-shaped graph $\Gamma$ whose boundary is $Y$.

{\bf Non-smoothable involutions on $4$-manifolds with boundary.} Our obstruction results can be used to give examples of orientation preserving, locally linear involutions which are non-smoothable in the sense that they are not smooth with respect to any differentiable structure on the manifold. One such construction is as follows. Let $X_0$ be a compact, oriented, spin, smooth $4$-manifold with boundary $Y$, an integral homology sphere. Assume that $H_1(X ; \mathbb{Z}_2) = 0$. Suppose $\sigma_0$ is a smooth, orientation preserving odd involution on $X_0$ and that $\sigma_0$ acts either as $+1$ or $-1$ on $H^2(X_0 ; \mathbb{Z})$. Such $4$-manifolds are easy to construct, for example $X_0$ could be an equivariant plumbing, as in \textsection \ref{sec:calc}. 

Consider two involutions $\sigma_+,\sigma_-$ on $S^2 \times S^2$, where $\sigma_+ = \phi \times \phi$, is the product to two rotations of $S^2$ by $\pi$ and $\sigma_- = r \times r$ is the product of two reflections of $S^2$. Let $X_{\pm}(m)$ be an equivariant connected sum of $(X_0 , \sigma_0)$ with $m$ copies of $(S^2 \times S^2 , \sigma_{\pm})$, where we use $+$ if $\sigma_0$ acts as $+1$ on $H^2(X_0 ; \mathbb{Z})$ and we use $-$ if $\sigma_0$ acts as $-1$ on $H^2(X_0 ; \mathbb{Z})$. Now let $W$ be any closed, simply-connected, topological $4$-manifold whose intersection form is even, negative definite and has non-zero rank. Let $X(m) = X_{\pm}(m) \# 2W$. Then we can define an involution $\sigma$ on $X(m)$ in such a way that the two copies of $W$ are exchanged. Clearly $\sigma$ is locally linear. In Section \ref{sec:nsa} we show the following:
\begin{proposition}
Suppose that $m > 3 b_2(W)/8$. Then $X(m)$ admits a smooth structure. However $\sigma$ is not smooth with respect to any smooth structure on $X(m)$.
\end{proposition}

Taking equivariant connected sums with $(S^2 \times S^2 , \sigma_+)$ or $(S^2 \times S^2 , \sigma_-)$ can be thought of as two kinds of equivariant stabilisation. The above proposition shows that the involution $\sigma$ on $X(m)$ remains non-smoothable upon stabilisation by $(S^2 \times S^2 , \sigma_{\pm})$, where the sign is chosen as previously described. We do not know whether there exist locally linear involutions which remain non-smoothable under both kinds of stabilisation.

{\bf Equivariant embeddings of $3$-manifolds in $4$-manifolds.} Let $Y$ be a rational homology $3$-sphere and $\sigma$ an orientation preserving smooth involution on $Y$. Consider the problem of embedding $Y$ into a closed, oriented, smooth $4$-manifold $X$ in such a way that $\sigma$ extends to an orientation preserving involution on $X$. Theorem \ref{thm:froy} can be used to give constraints on the existence of such embeddings. We will focus on the case of embeddings into $S^4$ or connected sums of $S^2 \times S^2$.

Every orientable $3$-manifold embeds in $\#n(S^2 \times S^2)$ for some sufficiently large $n$ \cite[Theorem 2.1]{agl}. Aceto--Golla--Larson define the embedding number $\varepsilon(Y)$ to be the least such $n$ such that $Y$ embeds in $\#n(S^2 \times S^2)$. We consider three equivariant versions of $\varepsilon$.

\begin{definition}
Let $Y$ be an orientable $3$-manifold and $\sigma$ an orientation preserving smooth involution on $Y$. Define the following invariants of $(Y,\sigma)$:
\begin{itemize}
\item[(1)]{$\varepsilon(Y,\sigma)$ is the least $n$ such that $Y$ embeds in $X = \#n(S^2 \times S^2)$ and $\sigma$ extends to an orientation preserving smooth involution on $X$. If no such $n$ exists, then we set $\varepsilon(Y, \sigma) = \infty$.}
\item[(2)]{$\varepsilon_+(Y,\sigma)$ is the least $n$ such that $Y$ embeds in $X = \#n(S^2 \times S^2)$ and $\sigma$ extends to a homologically trivial, orientation preserving smooth involution on $X$. If no such $n$ exists, then we set $\varepsilon_+(Y, \sigma) = \infty$.}
\item[(1)]{$\varepsilon_-(Y,\sigma)$ is the least $n$ such that $Y$ embeds in $X = \#n(S^2 \times S^2)$ and $\sigma$ extends to an orientation preserving smooth involution on $X$ which acts as $-1$ on $H^2(X ; \mathbb{Z})$. If no such $n$ exists, then we set $\varepsilon_-(Y, \sigma) = \infty$.}
\end{itemize}
\end{definition}

In Section \ref{sec:emb} we obtain upper bounds for $\varepsilon(Y, \sigma), \varepsilon_{\pm}(Y,\sigma)$ by constructing suitable embeddings. On the other hand we also prove the following lower bounds:
\begin{proposition}
Let $Y$ be an integral homology $3$-sphere and $\sigma$ an orientation preserving smooth involution on $Y$. We have:
\begin{itemize}
\item[(1)]{If the delta-invariants $\delta^T_*$ of $(Y,\sigma)$ do not all vanish, then $\varepsilon(Y,\sigma) \ge 2$.}
\item[(2)]{$\varepsilon_+(Y,\sigma) \ge \max\{ j^R(Y , \sigma) , 2j^R(Y,\sigma) - 8 \delta^R_\infty(Y,\sigma) \}$.}
\item[(3)]{$\varepsilon_-(Y,\sigma) \ge \max\{ j^E(Y , \sigma) , 2j^E(Y,\sigma) - 8 \delta^E_\infty(Y,\sigma) \}$.}
\item[(4)]{If $\delta^S_{0,1}(Y,\sigma), \delta^S_{1,0}(Y,\sigma)$ are both non-zero, then $\varepsilon(Y,\sigma) \ge 4$.}
\item[(5)]{If the Rokhlin invariant of $Y$ is non-zero and if $\delta^S_{0,\infty}(Y,\sigma), \delta^S_{\infty,0}(Y,\sigma), \delta^S_{\infty,1}(Y,\sigma)$ do not all have the same sign, then $\varepsilon(Y,\sigma) \ge 10$.}
\end{itemize}
\end{proposition}

For example, if $Y = \Sigma_2( T_{3,13})$ and $\sigma$ is the covering involution. Then $Y$ embeds in $S^4$  \cite[Theorem 2.13]{bb}, so $\varepsilon(Y) = 0$. But we will show that $\varepsilon_-(Y,\sigma) = 24$ and $2 \le \varepsilon(Y,\sigma) \le 6$. In particular, $\varepsilon(Y), \varepsilon(Y,m), \varepsilon_-(Y,m)$ take distinct values.

{\bf Topology of non-orientable surfaces bounding knots.} Let $K$ be a knot in $S^3$ and let $S$ be a connected, properly embedded, non-orientable surface in $D^4$ which bounds $K$. We consider a constraint on the topology of $S$ obtained from the type $R$ delta-invariants of $K$. A similar application was considered in \cite{kmt2}. The main difference is that our invariant can be calculated for a larger class of knots.

Let $e(S)$ denote the relative Euler class of $S$ with respect to the zero framing on $K$. Since $e(S)$ is valued in the orientation local system, we will identify $e(S)$ with an integer. Note that $e(S)$ is always even because its value mod $2$ is the mod $2$ self-intersection number of $S$, which is zero since $H^2(D^4 , S^3 ; \mathbb{Z}_2) = 0$. A natural question to ask is for a given $K$, what possible values of $( e(S) , b_1(S) )$ can be attained? This problem is studied for torus knots in \cite{all}. Here $b_1(S)$ is the first Betti number of $S$. 

Set $x = \sigma(K) - e(S)/2$ and $y = b_1(S)$. Then $x,y \in \mathbb{Z}$, $y \ge 0$, $|x| \le y$ and $x = y \; ({\rm mod} \; 2)$ (see Proposition \ref{prop:xy}). The following result concerns the boundary case $x=-y$ (the case $x=y$ is similar, simply replace $K$ by $-K$ in the proposition).

\begin{proposition}
Suppose that $x=-y$, or equivalently $\sigma(K) - e(S)/2 = -b_1(S)$. Then there exists a spin$^c$-structure $\mathfrak{s}$ on $\Sigma_2(K)$ for which $\delta^R_\infty( \Sigma_2(K) , \mathfrak{s} , \sigma) \ge 0$ and $\delta^R_0(-\Sigma_2(K) , \mathfrak{s} , \sigma) \le 0$, where $\sigma$ is the covering involution on $\Sigma_2(K)$.  Futhermore, if $\delta^R_\infty( \Sigma_2(K) , \mathfrak{s} , \sigma) = 0$ or $\delta^R_0(-\Sigma_2(K) , \mathfrak{s} , \sigma) = 0$, then $\mathfrak{s}$ is the unique spin structure on $\Sigma_2(K)$.

If $K$ is quasi-alternating (or more generally, if $\Sigma_2(K)$ is an $L$-space), then there exists a spin$^c$-structure $\mathfrak{s}$ on $\Sigma_2(K)$ for which $\delta( \Sigma_2(K) , \mathfrak{s} ) \ge 0$, with equality only if $\mathfrak{s}$ is the spin structure.
\end{proposition}

For example, we will show that if $K$ belongs to one of the following classes of knots then there does not exists a non-orientable surface $S$ bounding $K$ with $\sigma(K) - e(S)/2 = -b_1(S)$:
\begin{itemize}
\item[(1)]{Torus knots $T_{p,q}$ with $p,q$ odd and $\overline{\mu}(\Sigma(2,p,q)) > 0$.}
\item[(2)]{Montesinos knots $M( e ; (a_1,b_1) , \dots , (a_n , b_n))$ where $a_1, \dots , a_n$ are coprime, $a_i$ is even for some $i$, $e - \sum_{i=1}^{n} b_i/a_i = 1/(a_1\cdots a_n)$ and satisfying \linebreak $\overline{\mu}( \Sigma(a_1, \dots , a_n) ) < 0$.}
\end{itemize}

\subsection{Relation to other works}

Recently there have been some papers related to Seiberg--Witten Floer theory for Real spin$^c$-structures \cite{kmt1,kmt2,mi}. The main difference between these works and ours is they are based on Seiberg--Witten theory for {\em $\sigma$-invariant} configurations, while ours is based on {\em $\sigma$-equivariant} Seiberg--Witten theory for the full space of configurations. Naturally one might expect a relationship between the $\sigma$-invariant and $\sigma$-equivariant theories. However the relationship is not as straightforward as one might initially expect. The construction of the Seiberg--Witten Floer spectrum $SWF(Y,\mathfrak{s})$ involves desuspension by certain metric dependent quantities. This metric dependence leads to difficulties in comparing the two theories.

The paper \cite{kmt2} introduces a delta-invariant $\delta_R(Y , \mathfrak{s} , \sigma)$ for spin$^c$-structures of type $R$ and a pair of delta-invariants $\overline{\delta}_R(Y , \mathfrak{s} , \sigma), \underline{\delta}_R(Y , \mathfrak{s} , \sigma)$ for spin structures of type $S$. Some of the applications considered in this paper could possibly also be obtained using these invariants. However, the invariants introduced in this paper have some advantages. A main one is the existence of a spectral sequence (Theorem \ref{thm:ss}) which relates $HSW^*_T$ to $HSW^*$. We use this spectral sequence to deduce that our delta-invariants coincide with the ordinary delta-invariant $\delta(Y,\mathfrak{s})$ for $L$-spaces. There is no corresponding result known for $\delta_R, \overline{\delta}_R, \underline{\delta}_R$. Even for $3$-manifolds which are not $L$-spaces, the spectral sequence can be used to compute, or at least constrain the value of the invariants. See the proof of Theorem \ref{thm:briedelt} and \cite[\textsection 3]{bh2} illustrations of this. A second advantage is that because our invariants $\delta^T_*(Y , \mathfrak{s} , \sigma)$ depend on an index $*$, they can provide more refined information than can be obtained from a single delta-invariant. Rather than a single Fr{\o}yshov-type inequality, we obtain a whole series of such inequalities. Furthermore, we can define invariants $j^E(Y , \mathfrak{s} , \sigma), j^R(Y , \mathfrak{s} , \sigma)$ to be the smallest value of $j$ for which the sequences $\delta^E_j(Y , \mathfrak{s} , \sigma), \delta^R_j(Y , \mathfrak{s} , \sigma)$ stabilise (Definition \ref{def:j}). The invariants $j^E,j^R$ have a number of applications as seen in Propositions \ref{prop:ht}, \ref{prop:hm}, \ref{prop:embbound} which would not be possible to obtain with only a single delta-invariant.

Also related to our work are the papers \cite{mon1,mon2}. These papers study equivariant Seiberg--Witten Floer theory for spin structures. However, the invariants in these papers are $K$-theoretical, whereas ours are cohomological.

\subsection{Structure of the paper}

In Section \ref{sec:inv} we establish some preliminary results about involutions and the three types of spin$^c$-structures featured in the paper. In Section \ref{sec:sym} we introduce and study the equivariant Seiberg--Witten Floer cohomology groups for spin$^c$-structures of the three types. We define the corresponding delta-invariants and establish their key properties. The central result is the equivariant Fr{\o}yshov inequality, Theorem \ref{thm:froy}. In Section \ref{sec:con} we consider various constructions of rational homology $3$-spheres with involution and calculate their delta-invariants. The three types of constructions considered are equivariant plumbing (\textsection \ref{sec:pg}), branched double covers of knots (\textsection \ref{sec:bdc}) and equivariant Dehn surgery (\textsection \ref{sec:eds}). In Section \ref{sec:app} we consider various applications of the delta-invariants, namely obstructions to extending involutions over bounding $4$-manifolds (\textsection \ref{sec:obex}), non-smoothable involutions on $4$-manifolds with boundary (\textsection \ref{sec:nsa}), equivariant embeddings of $3$-manifolds in $4$-manifolds (\textsection \ref{sec:emb}) and the topology of non-orientable surfaces bounding knots (\textsection \ref{sec:nos}).

\noindent{\bf Acknowledgments.} The first author was financially supported by an Australian Research Council Future Fellowship, FT230100092. The second author was  supported by the RDF Grant 3727359.

\section{Involutions and spin$^c$-structures}\label{sec:inv}

Let $M$ be an oriented smooth manifold of dimension $n =3$ or $4$, let $\sigma$ be an orientation preserving smooth involution on $M$ and let $g$ be a $\sigma$-invariant metric on $M$. Let $\mathfrak{s}$ be a spin$^c$-structure and $P \to M$ the corresponding principal $Spin^c(n)$-bundle. Let $\tau\colon Spin^c(n) \to Spin^c(n)$ be the automorphism which is the identity on $Spin(n)$ and is complex conjugation on $U(1)$. A {\em Real} or {\em Quaternionic} structure on $\mathfrak{s}$ is a lift $\widetilde{\sigma}\colon P \to P$ of $\sigma$ such that $\widetilde{\sigma}(pg) = \widetilde{p}\tau(g)$ for all $p \in P$, $g \in Spin^c(4)$, $\widetilde{\sigma}$ projects to the derivative of $\sigma$ on the frame bundle of $M$ and $\widetilde{\sigma}^2 = 1$ in the Real case, $\widetilde{\sigma}^2 = -1$ in the Quaternionic case.

If $\mathfrak{s}$ is any spin$^c$-structure on $M$ and $P \to M$ the corresponding $Spin^c(n)$-bundle, then we obtain a new spin$^c$-structure $-\mathfrak{s}$, the {\em charge conjugate} of $\mathfrak{s}$ by declaring the $Spin^c(n)$-bundle of $-\mathfrak{s}$ to be $P^\tau$, where $P^\tau$ is the same underlying space as $P$, but equipped with the right action $(p,g) \mapsto p \tau(g)$. It follows that if $\mathfrak{s}$ admits a Real or Quaternionic structure, then $\sigma^*(\mathfrak{s}) \cong -\mathfrak{s}$.

Let $S \to M$ denote the spinor bundle associated to $\mathfrak{s}$. A lift $\widetilde{\sigma}\colon P \to P^\tau$ of $\sigma$ to $P$ induces a lift of $\sigma$ to $S$ which we will also denote by $\widetilde{\sigma}\colon S \to S$. The map $\widetilde{\sigma}\colon S \to S$ is antilinear, preserves the Hermitian structure and is compatible with Clifford multiplication in the obvious way. Conversely a lift of $\sigma$ to $S$ with these properties corresponds to a lift $\widetilde{\sigma}\colon P \to P^\tau$. We will use these two points of view interchangeably.

\begin{proposition}\label{prop:RQ}
Assume that $b_1(M) = 0$. Let $\mathfrak{s}$ be a spin$^c$-structure on $M$. Then $\mathfrak{s}$ admits a Real or Quaternionic structure if and only if $\sigma^*(\mathfrak{s}) \cong -\mathfrak{s}$. In such a case the Real or Quaternionic structure is unique up to isomorphism (in particular $\mathfrak{s}$ can not admit both a Real and a Quaternionic structure).
\end{proposition}
\begin{proof}
We have already seen that if $\mathfrak{s}$ admits a Real or Quaternionic structure, then $\sigma^*(\mathfrak{s}) \cong -\mathfrak{s}$. We prove the converse. Suppose that $\sigma^*(\mathfrak{s}) \cong -\mathfrak{s}$. Then we can choose an antilinear lift $\widetilde{\sigma}\colon S \to S$ of $\sigma$ which covers the derivative of $\sigma$ on the frame bundle. The derivative of $\sigma$ is an involution, hence $\widetilde{\sigma}^2 = h$ for some $U(1)$-valued function $h$. Since $\widetilde{\sigma} h = \widetilde{\sigma}^3 = h \widetilde{\sigma}$, we see that $\sigma^*(h) = h^{-1}$. Since $b_1(M) = 0$, we can write $h = e^{2\pi i u }$ for some real-valued function $u$. The condition $\sigma^*(h) = h^{-1}$ implies that $\sigma^*(u) = -u + n$ for some integer $n$. Now let $v$ be any real-valued function and set $g = e^{2\pi i v}$. Then $g^{-1} \widetilde{\sigma}$ is another antilinear lift of $\sigma$ which covers the derivative of $\sigma$ on the frame bundle. Furthermore we have $(g^{-1} \widetilde{\sigma})^2 = g^{-1} \sigma^*(g) h = e^{2\pi i( u + \sigma^*(v) - v )}$. We choose $v = u/2$. Then 
\[
u + \sigma^*(v) - v = u + \frac{1}{2}\sigma^*(u) - \frac{1}{2}u = u + \frac{1}{2}( -u + n) - \frac{1}{2}u = n/2.
\]
Hence $(g^{-1} \widetilde{\sigma})^2 = e^{\pi i n} = (-1)^n$ and we get a Real or Quaternionic structure according to whether $n$ is even or odd. This proves existence.

For uniqueness, suppose that $\widetilde{\sigma}$ is a lift of $\sigma$ such that $\widetilde{\sigma}^2 = \epsilon = \pm 1$. Any other antilinear lift of $\sigma$ which agrees with the derivative on the frame bundle is given by $g^{-1} \widetilde{\sigma}$ for some $S^1$-valued function $g$. Furthermore since $b_1(M) = 0$, we can write $g = e^{2\pi i u}$ for some real-valued function $u$. Then $(g^{-1} \widetilde{\sigma})^2 = g^{-1} \sigma^*(g) \epsilon = e^{2\pi i( \sigma^*(u) - u )} \epsilon$. For this to be a Real or Quaternionic structure we need $\sigma^*(u) - u = n/2$ for some integer $n$. So $\sigma^*(u) = u + n/2$. Then $u = \sigma^*( \sigma^*(u) ) = \sigma^*( u+n/2) = u + n$, which is only possible if $n=0$. In particular, $g^{-1}\widetilde{\sigma}$ must have the same type as $\widetilde{\sigma}$ (both Real or both Quaternionic). Set $h = e^{\pi i u}$. Then $\sigma^*(h) = h$ and $h^2 = g$. Therefore
\[
h^{-1} \circ \widetilde{\sigma} \circ h = h^{-1} \sigma^*(h)^{-1} \widetilde{\sigma} = g^{-1} \widetilde{\sigma}.
\]
Hence $h\colon (S , g^{-1} \widetilde{\sigma}) \to (S , \widetilde{\sigma})$ is an isomorphism of Real or Quaternionic structures.

\end{proof}

For the purposes of Seiberg--Witten theory, it is convenient to choose a spin$^c$-connection which is preserved by a Real or Quaternionic structure. The following result shows that this is possible in a strong sense.

\begin{proposition}\label{prop:liftRQ}
Assume that $b_1(M) = 0$. Suppose that $\sigma^*(\mathfrak{s}) \cong -\mathfrak{s}$. Let $A$ be any spin$^c$-connection such that $\sigma^*(F_A) = -F_A$, where $F_A$ is the curvature of $A$. Then there exists an antilinear lift $\widetilde{\sigma}$ of $\sigma$ which equals the derivative of $\sigma$ on the frame bundle, preserves $A$ and squares to $\pm 1$. The lift $\widetilde{\sigma}$ is unique up to multiplication by an element of $S^1$.
\end{proposition}
\begin{proof}
Choose any antilinear lift $\widetilde{\sigma}$ of $\sigma$ which equals the derivative on the frame bundle. Then $\widetilde{\sigma}^*(A) = A + i\mu$ for some real $1$-form $\mu$. Since $\sigma^*(F_A) = -F_A$, it follows that $d\mu = 0$. Since $b_1(M) = 0$, we can write $\mu = du$ for a real-valued function $u$. Then $\sigma_1 = e^{i \sigma^*(u)} \widetilde{\sigma}$ preserves $A$. It follows that $\sigma_1^2 = c$ for some constant $c \in S^1$. But $c \sigma_1 = \sigma_1^3 = \sigma_1 c$, hence $c = c^{-1}$ and so $c = \pm 1$. Hence $\sigma_1$ is a Real or Quaternionic structure preserving $A$. The uniqueness statement follows since the only gauge transformations preserving $A$ are constants.
\end{proof}

Now suppose that $\mathfrak{s}$ is a spin$^c$-structure satisfying $\sigma^*(\mathfrak{s}) \cong \mathfrak{s}$. Then in a similar manner to Proposition \ref{prop:RQ}, it can be shown that $\sigma$ admits a linear involutive lift $\widetilde{\sigma}$. Moreover if $A$ is any spin$^c$-connection such that $\sigma^*(F_A) = F_A$, then we can choose the lift to preserve $A$. If $\widetilde{\sigma}$ is one such lift, then the only other lift is $-\widetilde{\sigma}$. See \cite[\textsection 3.2]{bh} for more details.

Lastly, suppose that $\mathfrak{s}$ is a spin-structure satisfying $\sigma^*(\mathfrak{s}) \cong \mathfrak{s}$. Let $\widetilde{\sigma}$ be a lift of $\sigma$ to the spinor bundles. In this case there are precisely two lifts and if $\widetilde{\sigma}$ is one of the lifts, then the other is $-\widetilde{\sigma}$. These two lifts automatically preserve any spin connection. We then have $\widetilde{\sigma}^2 = \pm 1$ because $\widetilde{\sigma}^2$ is an automorphism of the spin bundle covering the identity on the frame bundle. Recall that $\sigma$ is said to be {\em even} if $\widetilde{\sigma}^2 = 1$ and {\em odd} if $\widetilde{\sigma}^2 = -1$. If $S$ is the complex spinor bundle associated to a spin structure, then we have an antilinear map $j\colon S \to S$ called {\em charge conjugation} which satisfies $j^2 = -1$. Observe that if $\widetilde{\sigma}$ is a linear lift of $\sigma$ then, $j \widetilde{\sigma}$ is an antilinear lift. Moreover, $\widetilde{\sigma}$ commutes with $j$, so $(j \widetilde{\sigma})^2 = - \widetilde{\sigma}^2$. Therefore $\sigma$ is even/odd if and only the underlying spin$^c$-structure of $\mathfrak{s}$ is Quaternionic/Real.

In what follows, we will mainly be interested in the case of Real spin$^c$-structures or odd spin involutions. The reason for this is the following:

\begin{proposition}
Let $M$ be an oriented $3$-manifold. Let $\sigma\colon M \to M$ be an orientation preserving involution with non-empty fixed point set. Then there are no Quaternionic spin$^c$-structures on $M$.
\end{proposition}
\begin{proof}
First recall that if $\sigma$ is an orientation preserving involution which preserves a spin structure, then every component of the fixed point set of $\sigma$ has codimension $0 \; ({\rm mod} \; 4)$ in the even case and codimension $2 \; ({\rm mod} \; 4)$ in the odd case \cite[Proposition 8.46]{ab}. If $M$ is a $3$-manifold and the fixed point set is non-empty, then it necessarily has codimension $2$, so any spin involution must be odd. The argument carries over to the spin$^c$ case, because any spin$^c$-structure can locally be reduced to a spin structure.
\end{proof}

We will treat spin structures (in dimensions $n = 3$ or $4$) as spin$^c$-structures equipped with a charge conjugation symmetry. In other words, a spin structure can be thought of as a spin$^c$-structure $\mathfrak{s}$ together with a map $j\colon P \to P^\tau$, where $P \to M$ is the principal $Spin^c(n)$-bundle corresponding to $\mathfrak{s}$ such that $j$ covers the identity on the frame bundle and $j^2 = -1$.

\begin{definition}
Let $M$ be a smooth, oriented manifold of dimension $n = 3$ or $4$ and assume that $b_1(M) = 0$. Let $\sigma$ be a smooth orientation preserving involution on $M$ and let $g$ be a $\sigma$-invariant metric. Let $\mathfrak{s}$ be a spin$^c$-structure on $M$ and $S$ the associated spinor bundle. We say that $\mathfrak{s}$ is
\begin{itemize}
\item{{\em equivariant type} if $\sigma^*(\mathfrak{s}) \cong \mathfrak{s}$. In this case $\sigma$ can be lifted to a linear involution on $S$.}
\item{{\em Real type} if $\sigma$ can be lifted to an antilinear involution on $S$. If $n=3$ and the fixed point set of $\sigma$ is non-empty, this is equivalent to saying $\sigma^*(\mathfrak{s}) \cong -\mathfrak{s}$.}
\item{{\em odd spin type} if $\mathfrak{s}$ is the spin$^c$-structure underlying a spin structure for which $\sigma$ is an odd involution. In this case $S$ admits a charge conjugation symmetry $j$ and a linear lift $\widetilde{\sigma}$ of $\sigma$ such that $j$ and $\widetilde{\sigma}$ commute and $\widetilde{\sigma}^2 = -1$.}
\end{itemize}
\end{definition}

We will use the letters $E,R,S$ to denote the above three types of spin$^c$-structure, $E$ for equivariant, $R$ for Real and $S$ for odd spin. We will sometimes use a letter $T$ to refer to any one of these three types, thus $T \in \{ E , R , S\}$.

\subsection{Symmetry groups}\label{sec:sym}

Let $M$ be a smooth, oriented manifold of dimension $n = 3$ or $4$ and assume that $b_1(M) = 0$. Let $\sigma$ be a smooth orientation preserving involution on $M$ and let $g$ be a $\sigma$-invariant metric. Let $\mathfrak{s}$ be a spin$^c$-structure on $M$ and $S$ the associated spinor bundle. Suppose that $\mathfrak{s}$ has type $T \in \{E,R,S\}$. Let $A$ be a spin$^c$-connection for $\mathfrak{s}$ and let $F_A$ denote the curvature. Assume that $\sigma^*(F_A) = F_A$ in type $E$, $\sigma^*(F_A) = -F_A$ in type $R$ and assume that $A$ is a spin connection in type $S$ (in particular, $F_A = 0$). We define a group $G^T_{\mathfrak{s}}$ whose elements are certain bundle maps of $S$ covering the identity or $\sigma$:
\begin{itemize}
\item{$G^E_{\mathfrak{s}}$ is the group generated by $S^1$ and all linear lifts of $\sigma$ to $S$ which preserve $A$.}
\item{$G^R_{\mathfrak{s}}$ is the group generated by $S^1$ and all antilinear lifts of $\sigma$ to $S$ which preserve $A$.}
\item{$G^S_{\mathfrak{s}}$ is the group generated by $S^1$, $j$ and all linear lifts of $\sigma$ to $S$ which preserve $A$.}
\end{itemize}

Note that different choices of spin$^c$-connection $A$ yield conjugate subgroups of endomorphisms of $S$.

In types $T = E, R$, $G^T_{\mathfrak{s}}$ is an extension of $\mathbb{Z}_2 = \langle \sigma \rangle$ by $S^1$:
\[
1 \to S^1 \to G^T_{\mathfrak{s}} \to \mathbb{Z}_2 \to 1
\]
In type $E$, $\sigma$ lifts to a linear involution, so $G^E_{\mathfrak{s}} \cong \mathbb{Z}_2 \times S^1$. In type $R$, $\sigma$ lifts to an antilinear involution, so $G^R_{\mathfrak{s}} \cong O(2)$.

In type $S$, $G^S_{\mathfrak{s}}$ is an extension of $\mathbb{Z}_2 = \langle \sigma \rangle$ by the group $Pin(2) = S^1 \cup j S^1$:
\[
1 \to Pin(2) \to G^T_{\mathfrak{s}} \to \mathbb{Z}_2 \to 1.
\]
Moreover, $\sigma$ admits a lift $\widetilde{\sigma}$ which squares to $-1$ and commutes with $Pin(2)$.

Throughout this paper all cohomology groups will use coefficient field $\mathbb{F} = \mathbb{Z}/2\mathbb{Z}$, unless stated otherwise. Of particular interest will be the ring $R_T^* = H^*_{G^T_{\mathfrak{s}}}(pt)$, the $G^T_{\mathfrak{s}}$-equivariant cohomology of a point (with $\mathbb{F}$-coefficients). We have:

\begin{proposition}\label{prop:R*}
We have isomorphisms:
\begin{itemize}
\item[(1)]{$R^*_E \cong \mathbb{F}[s,u]$, $deg(s) = 1$, $deg(u) = 2$.}
\item[(2)]{$R^*_R \cong \mathbb{F}[w_1,w_2]$, $deg(w_1) = 1$, $deg(w_2) = 2$.}
\item[(3)]{$R^*_S \cong \mathbb{F}[s,v,q]/( v(v+s)^2 )$, $deg(s) = 1$, $deg(v) = 1$, $deg(q) = 4$.}
\end{itemize}

In type $E$, the above isomorphism depends on a choice of splitting $\mathbb{Z}_2 \to G^S_{\mathfrak{s}}$. The two isomorphisms $R^*_E \cong \mathbb{F}[s,u]$ corresponding to the two different splittings are related by the automorphism $u \mapsto u + s^2$, $s \mapsto s$. In type $R$ and $S$, the isomorphisms do not depend on a choice of splitting.

\end{proposition}
\begin{proof}
Case (1) follows easily from the K\"unneth theorem since $G^E_\mathfrak{s} \cong \mathbb{Z}_2 \times S^1$. For case (2), we note that since $G^R_{\mathfrak{s}} \cong O(2)$, $R^*_R$ is the ring of characteristic classes for $O(2)$ with coefficients in $\mathbb{F}$, which is well known to be a polynomial ring in the Stiefel--Whitney classes $w_1,w_2$.

Now we consider case (3). Recall \cite{man2} that $H^*_{Pin(2)}(pt) \cong \mathbb{F}[v,q]/(v^3)$ where $deg(v) = 1$, $deg(q) = 4$ and $H^*_{\mathbb{Z}_2}(pt) \cong \mathbb{F}[s]$ where $deg(s) = 1$. We have a short exact sequence 
\[
1 \to Pin(2) \to G^S_{\mathfrak{s}} \to \mathbb{Z}_2 = \langle \sigma \rangle \to 1.
\]
Choose a lift $\widetilde{\sigma}$ of $\sigma$ which squares to $-1$ and commutes with $Pin(2)$. This determines a splitting $\mathbb{Z}_2 \to G^S_{\mathfrak{s}}$ of the above sequence which maps $\sigma$ to $j \widetilde{\sigma}$. Then it follows from the Lyndon--Hochschild--Serre spectral sequence that
\[
H^*_{G^S_{\mathfrak{s}}}(pt) \cong \mathbb{F}[s,v,q]/(v^3+ \cdots)
\]
where $v^3 + \cdots$ denotes a degree $3$ polynomial in $v$ and $s$. The unknown terms in $v^3 + \cdots$ are some linear combination of $v^2s, vs^2$ and $s^3$.

The class $s$ is the pullback of the generator of $H^1_{\mathbb{Z}_2}(pt)$ to $H^*_{G^S_{\mathfrak{s}}}(pt)$ under the homomorphism $\phi_\sigma\colon {G^S_{\mathfrak{s}}} \to \mathbb{Z}_2$ which sends $\widetilde{\sigma}$ to $-1$ and $j$ to $+1$. Similarly, $v$ is the pullback under the homomorphism $\phi_j\colon G^S_{\mathfrak{s}} \to \mathbb{Z}_2$ which sends $\widetilde{\sigma}$ to $+1$ and $j$ to $-1$.

We must have a relation of the form $v^3 + Av^2 s + Bvs^2 + C s^3 = 0$ for some $A,B,C \in \mathbb{F}$. Consider the homomorphism $\psi\colon \mathbb{Z}_2 \to G^S_{\mathfrak{s}}$ which sends $-1$ to $i \widetilde{\sigma}$. Since the composition $\phi_\sigma \circ \psi\colon \mathbb{Z}_2 \to \mathbb{Z}_2$ is the identity, we see that $\psi^*(s)$ is the generator of $H^*_{\mathbb{Z}_2}(pt)$. On the other hand, $\phi_j \circ \psi\colon \mathbb{Z}_2 \to \mathbb{Z}_2$ is the trivial map and hence sends $v$ to $0$. Therefore the relation $v^3 + Av^2 s + Bvs^2 + C s^3 = 0$ when pulled back by $\psi$ gives $C s^3 = 0$, hence $C = 0$.

Next we consider the Lyndon--Hochschild--Serre spectral sequence $\{ E_r^{p,q} , d_r \}$ associated to $O(2) \to G^S_{\mathfrak{s}} \to \mathbb{Z}_2$. Here the map $\phi\colon G^S_{\mathfrak{s}} \to \mathbb{Z}_2$ is the product of $\phi_j$ and $\phi_\sigma$. Letting $\rho \in H^1_{\mathbb{Z}_2}(pt)$ denote a generator of the cohomology of this copy of $\mathbb{Z}_2$, then we have $\phi^*(\rho) = v+s$. Recall that $H^*_{O(2)}(pt) \cong \mathbb{F}[w_1,w_2]$. Since $H^1_{G^S_{\mathfrak{s}}}(pt) \cong \mathbb{F}^2$, there is no differential on $E_2^{0,1}$. Since $H^2_{G^S_{\mathfrak{s}}}(pt) \cong \mathbb{F}^3$, we must have a non-zero differential in the $(p,q) = (0,2)$ position. In fact, either $d_2(w_2) \neq 0$ or $d_2(w_2) = 0$, $d_3(w_2) \neq 0$. In the latter case we must then have $d_3(w_2) = \rho^3 = (v+s)^3$, which would imply that $v^3 + v^2s + vs^2 + s^3 = 0$ and hence $A=B=C=1$. But this contradicts our previous calculation which showed that $C=0$. Hence $d_2(w_2) \neq 0$. This means $d_2(w_2) = \rho^2 w_1$. Now since $w_1 = v \; ({\rm mod} \; (v+s))$ and $d_2(w_2) = \rho^2 w_1 = (v+s)^2 w_1$, we get $(v+s)^2 v = 0 \; ({\rm mod} \; (v+s)^3 )$. Hence either $v(v+s)^2 = 0$ or $v(v+s)^2 = (v+s)^3$. The former gives $v^3 + vs^2 = 0$, the latter gives $v^2 s + s^2 = 0$ which is not possible. So the relation satisfied by $v$ is $v(v+s)^2 = 0$ and hence
\[
H^*_{G^S_{\mathfrak{s}}}(pt) \cong \mathbb{F}[s,v,q]/( v(v+s)^2).
\]
We now show that the above isomorphism does not depend on the choice of lift $\widetilde{\sigma}$. For this it is helpful to interpret the ring $\mathbb{F}[s,v,q]/(v(v+s)^2)$ as follows. Define the ring $H^*_S$ to be the image of $H^*_{\mathbb{Z}_2 \times \mathbb{Z}_2}(pt)$ under the pullback $H^*_{\mathbb{Z}_2 \times \mathbb{Z}_2}(pt) \to H^*_{G^S_{\mathfrak{s}}}(pt)$ induced by the homomorphism $G^S_{\mathfrak{s}} \to \mathbb{Z}_2 \times \mathbb{Z}_2$. Then $H^*_S \cong \mathbb{F}[s,v]/(v(v+s)^2)$ and $H^*_{G^S_{\mathfrak{s}}}(pt) \cong H^*_S[q] \cong \mathbb{F}[s,v,q]/(v(v+s)^2)$, the isomorphism depending only on the choice of a class $q \in H^4_{G^S_{\mathfrak{s}}}(pt)$ which restricts to $q \in H^4_{Pin(2)}(pt)$ under the inclusion $Pin(2) \to G^S_{\mathfrak{s}}$. We show below that the class $q$ can be chosen in a manner that does not depend on the choice of lift. We can take $q$ to be $q = w_4( \mathbb{H}_{i})$, where $\mathbb{H}_i$ is defined as follows. Let $\mathbb{H}$ be the quaternion representation of $Pin(2)$. There are exactly two ways to extend this to a representation of $G$, namely we can take $\widetilde{\sigma}$ to act as $+i$ or as $-i$. We denote these two representations as $\mathbb{H}_i, \mathbb{H}_{-i}$. Notice however that if we change the lift $\widetilde{\sigma}$ to $-\widetilde{\sigma}$ then the roles of $\mathbb{H}_i, \mathbb{H}_{-i}$ are reversed. Now since $\mathbb{H}_i$ is a complex representation its odd Stiefel--Whitney classes are zero and its even ones are the mod $2$ reductions of the Chern classes. In particular $w_2( \mathbb{H}_i ) = c_1( \det(\mathbb{H}_i) )$. Now $e^{i\theta}$ and $j$ act trivially on $\det(\mathbb{H}_i)$, but $\widetilde{\sigma}$ acts as $\det(i,i) = i^2 = -1$. It follows that $w_2( \mathbb{H}_i ) = s^2$. Then, since $\mathbb{H}_{-i} = \mathbb{H}_i \otimes_{\mathbb{R}} \mathbb{R}_{1,-1}$, where $\mathbb{R}_{1,-1}$ denotes the representation where $j=1$, $\widetilde{\sigma} = -1$, it follows that $w_4( \mathbb{H}_{-i}) = q + s^2 w_2(\mathbb{H}_i)s^2 + s^4 = q + s^4 + s^4 = q$. Thus, changing the splitting does not change $q$.
\end{proof}

\section{Seiberg--Witten Floer theory for involutions}\label{sec:swinv}

In this section we construct equivariant versions of Seiberg--Witten Floer cohomology for a rational homology $3$-sphere equipped with a spin$^c$-structure of type $E,R$, or $S$.

Let $Y$ be a rational homology $3$-sphere (a compact, oriented, smooth $3$-manifold with $b_1(Y) = 0$) and let $\mathfrak{s}$ be a spin$^c$-structure on $Y$. Let $g$ be a metric on $Y$. Manolescu constructed an $S^1$-equivariant stable homotopy type $SWF(Y , \mathfrak{s} , g)$ \cite{man} and defined the Seiberg--Witten Floer cohomology $HSW^*(Y,\mathfrak{s})$ of $(Y,\mathfrak{s})$ by setting
\[
HSW^*(Y , \mathfrak{s}) = \widetilde{H}^{* + 2n(Y,\mathfrak{s},g)}_{S^1}( SWF(Y,\mathfrak{s},g))
\]
where
\[
n(Y , \mathfrak{s} , g) = \frac{1}{2} \eta(D_A) - \frac{1}{2}k(D_A) - \frac{1}{8}\eta_{sign}
\]
where $A$ is a flat spin$^c$-connection, $D_A$ is the associated Dirac operator, $\eta(D_A)$ is the eta invariant of $D_A$, $k(D_A)$ is the complex dimension of the kernel of $D_A$ and $\eta_{sign}$ is the eta invariant of the signature operator. The definition of $HSW^*(Y , \mathfrak{s})$ appears to depend on the choice of metric $g$ (and on some other auxiliary choices used to construct $SWF(Y,\mathfrak{s},g)$), but it is shown in \cite{man} that the cohomology groups for different choices are related by canonical isomorphisms. From the definition it is clear that the group $HSW^*(Y,\mathfrak{s})$ is a graded module over the ring $H^*_{S^1}(pt) \cong \mathbb{F}[u]$ where $deg(u) = 2$.

The Seiberg--Witten Floer cohomology groups are known to be isomorphic to other types of Floer cohomology, namely monopole Floer \cite{lima2} and Heegaard Floer homology \cite{klt1,klt2,klt3,klt4,klt5,cgh1,cgh2,cgh3,tau}. To be precise, we have isomorphisms of $\mathbb{F}[u]$-modules:
\[
HSW^*(Y , \mathfrak{s}) \cong HF_+^*(Y , \mathfrak{s})
\]
where $HF_+^*(Y , \mathfrak{s})$ denotes the plus version of Heegaard Floer cohomology with coefficients in $\mathbb{F}$. In the case that $\mathfrak{s}$ is a spin structure, the stable homotopy type $SWF(Y , \mathfrak{s} , g)$ can be constructed $Pin(2)$-equivariantly and following Manolescu \cite{man2} we can define a $Pin(2)$-equivariant version of Seiberg--Witten Floer cohomology
\[
HSW^*_{spin}(Y , \mathfrak{s}) = \widetilde{H}^{* + 2n(Y , \mathfrak{s} , g)}_{Pin(2)}( SWF(Y , \mathfrak{s} , g))
\]
(the notation $HSW^*_{spin}(Y , \mathfrak{s})$ is non-standard, but is chosen to avoid confusion with equivariant Seiberg--Witten Floer cohomology as defined below).

In \cite{bh}, we constructed an equivariant version of Seiberg--Witten Floer cohomology in the presences of a finite group $G$ acting orientation preservingly on $Y$ which preserves the isomorphism class of $\mathfrak{s}$. In more detail, suppose we fix a $G$-invariant flat spin$^c$-connection $A$. Then we define $G_{\mathfrak{s}}$ to be the group of all possible linear lifts of $G$ to the spinor bundle of $\mathfrak{s}$ which preserves $A$. This group is a central extension of $G$ by $S^1$. One can carry out Manolescu's construction of the stable homotopy type $SWF(Y , \mathfrak{s} , g)$ equivariantly with respect to $G_{\mathfrak{s}}$. Then we defined the $G$-equivariant Seiberg--Witten Floer cohomology of $(Y,\mathfrak{s})$ by
\[
HSW_G^*(Y , \mathfrak{s}) = \widetilde{H}^{* + 2n(Y,\mathfrak{s},g)}_{G_{\mathfrak{s}}}( SWF(Y,\mathfrak{s},g)).
\]

Suppose now that $G = \mathbb{Z}_2 = \langle \sigma \rangle$ is a group of order $2$ generated by an involution on $Y$. Suppose that $\mathfrak{s}$ has type $T \in \{E,R,S\}$. In Section \ref{sec:sym} we constructed a group $G^T_{\mathfrak{s}}$ which acts on the spinor bundle of $\mathfrak{s}$. By a straightforward extension of the construction in \cite{bh}, we can construct $SWF(Y , \mathfrak{s} , g)$ equivariantly with respect to the group $G^T_{\mathfrak{s}}$. Then we can use this to define an appropriate notion of Seiberg--Witten Floer cohomology of types $E,R,S$.

\begin{definition}
Let $Y$ be a rational homology $3$-sphere, $\sigma$ a smooth, orientation preserving diffeomorphism and $\mathfrak{s}$ a spin$^c$-structure of type $T \in \{E,R,S\}$. We define the {\em type $T$ Seiberg--Witten Floer cohomology} of $(Y , \mathfrak{s} , \sigma)$ by
\[
HSW^*_{T}(Y , \mathfrak{s} , \sigma) = \widetilde{H}^{* + 2n(Y , \mathfrak{s} , g)}_{G^T_{\mathfrak{s}}}(SWF(Y , \mathfrak{s} , g)).
\]
\end{definition}

By this definition $HSW^*_{T}(Y , \mathfrak{s} , \sigma)$ is an invariant of the triple $(Y , \mathfrak{s} , \sigma)$. Moreover it is a graded module over the ring $R^*_T$.

\begin{remark}
The case $T = E$ is simply the equivariant Seiberg--Witten Floer cohomology of $(Y , \mathfrak{s})$ as defined in \cite{bh} in the special case that $G \cong \mathbb{Z}_2$. The cases $T=R$ and $T=S$ are new.
\end{remark}

In all three types $T \in \{E,R,S\}$, $G^T_{\mathfrak{s}}$ contains a canonically defined $S^1$ subgroup. Let $Q^T_{\mathfrak{s}}$ be the quotient so that we have a short exact sequence
\[
1 \to S^1 \to G^T_{\mathfrak{s}} \to Q^T_{\mathfrak{s}} \to 1.
\]
In types $T = E,R$, we have $Q^T_{\mathfrak{s}} \cong \mathbb{Z}_2$ and in type $T=S$ we have $Q^S_{\mathfrak{s}} \cong \mathbb{Z}_2 \times \mathbb{Z}_2$. Note that $S^1$ acts trivially on $HSW^*(Y , \mathfrak{s})$ and so the natural action of $G^T_{\mathfrak{s}}$ on $HSW^*(Y , \mathfrak{s})$ factors though $Q^T_{\mathfrak{s}}$ and makes $HSW^*(Y , \mathfrak{s})$ into a $Q^T_{\mathfrak{s}}$-module.

\begin{theorem}\label{thm:ss}
Let $Y$ be a rational homology $3$-sphere, $\sigma$ a smooth, orientation preserving diffeomorphism and $\mathfrak{s}$ a spin$^c$-structure of type $T \in \{E,R,S\}$. There is a spectral sequence $E_r^{p,q}$ abutting to $HSW^*_T(Y,\mathfrak{s},\sigma)$ whose second page is
\[
E_2^{p,q} = H^p( BQ^T_{\mathfrak{s}} ; HSW^q(Y , \mathfrak{s}) ).
\]
\end{theorem}
\begin{proof}
This is a straightforward extension of \cite[Theorem 3.2]{bh}.
\end{proof}

We say that $Y$ is an {\em $L$-space} with respect to spin$^c$-structure $\mathfrak{s}$ and coefficient group $\mathbb{F}$ if $HSW^*(Y , \mathfrak{s})$ (taken with $\mathbb{F}$-coefficients) is a free $\mathbb{F}[u]$-module of rank $1$, where $\mathbb{F}[u] = H^*_{S^1}(pt)$ is the $S^1$-equivariant cohomology of a point.

\begin{corollary}
Suppose $Y$ is an $L$-space with respect to $\mathfrak{s}$. Then $HSW^*(Y , \mathfrak{s})$ is a free $R^*_T$-module of rank $1$ with generator in degree $d(Y,\mathfrak{s})$, the $d$-invariant of $Y$ (with respect to $\mathfrak{s}$ and $\mathbb{F}$).
\end{corollary}
\begin{proof}
In the case $T = E$, this follows from \cite[Theorem 3.5]{bh}. The idea of the proof is that if $Y$ is an $L$-space, then the spectral sequence of Theorem \ref{thm:ss} can be shown to degenerate as $E_2$. The cases $T = R, S$ are very similar so we omit the proof.
\end{proof}

For $T \in \{E,R,S\}$, let $H^*_T$ be the graded ring which is the image of the map $H^*_{Q^T_{\mathfrak{s}}}(pt) \to H^*_{G^T_{\mathfrak{s}}}(pt) = R^*_T$ induced by the quotient map $G^T_\mathfrak{s} \to Q^T_\mathfrak{s}$. Then from Proposition \ref{prop:R*} it follows that:
\begin{itemize}
\item[(1)]{$H^*_E \cong \mathbb{F}[s]$ and $R^*_E \cong H^*_E[u]$.}
\item[(2)]{$H^*_R \cong \mathbb{F}[w_1]$ and $R^*_E \cong H^*_R[w_2]$.}
\item[(3)]{$H^*_S \cong \mathbb{F}[s,v]/(v(v+s)^2)$ and $R^*_S \cong H^*_S[q]$.}
\end{itemize}
Let $H^{* \ge j}_T = \bigoplus_{k \ge j} H^k_T$. Let $I^*_T = H^{* \ge 1}_T$. If $M^*$ is a graded $H^*_T$-module (in particular, $M^*$ could be a graded $R^*_T$-module, which by restriction can be regarded as a $H^*_T$-module), then we let $\widehat{M}^*$ denote the $I^*_T$-adic completion of $M^*$.

\begin{lemma}\label{lem:loc}
Let $S$ be the multiplicative subset of $R^*_E = \mathbb{F}[s,u]$ generated by $u$ and $u+s^2$. Let $M^*$ be a graded $R^*_E$-module. Then we have an isomorphism $\widehat{ S^{-1} M^* } \cong \widehat{ u^{-1} M^* }$, where $u^{-1}M^*$ denotes the localisation of $M^*$ with respect to the multiplicative subset generated by $u$. In particular, if $M^* = R^*_E$, then $\widehat{ S^{-1} R^*_E } \cong (\mathbb{F}[u,u^{-1}])[[s]]$.
\end{lemma}
\begin{proof}
Since $S$ is generated by $u$ and $u+s^2$, it suffices to show that $u+s^2$ is invertible in $\widehat{ u^{-1} M^*}$. In fact, we have $1 = (u+s^2)( u^{-1} + u^{-2}s^2 + u^{-3}s^4 + \cdots )$.
\end{proof}

\begin{proposition}\label{prop:loc}
Let $Y$ be a rational homology $3$-sphere, $\sigma$ a smooth, orientation preserving diffeomorphism and $\mathfrak{s}$ a spin$^c$-structure of type $T \in \{E,R,S\}$. Then we have an isomorphism:
\begin{itemize}
\item[(1)]{$\widehat{ u^{-1} HSW}^*_E(Y , \mathfrak{s} , \sigma) \cong (\mathbb{F}[u,u^{-1}])[[s]]\theta$ for some homogeneous element $\theta \neq 0 \; ({\rm mod} \; I^*_E)$.}
\item[(2)]{$\widehat{ w_2^{-1} HSW}^*_R(Y , \mathfrak{s} , \sigma) \cong (\mathbb{F}[w_2,w_2^{-1}])[[w_1]]\theta$ for some homogeneous element $\theta \neq 0 \; ({\rm mod} \; I^*_R)$.}
\item[(3)]{$\widehat{ q^{-1} HSW}^*_S(Y , \mathfrak{s} , \sigma) \cong (\mathbb{F}[q,q^{-1}])[[s,v]]/(v(v+s)^2)\theta$ for some homogeneous element $\theta \neq 0 \; ({\rm mod} \; I^*_S)$.}
\end{itemize}
\end{proposition}
\begin{proof}
We consider case (1), the other cases being similar. Here we have $G^E_{\mathfrak{s}} \cong \mathbb{Z}_2 \times S^1$. Recall \cite[\textsection 3]{bh} that $SWF(Y , \mathfrak{s} , g)$ has the form $SWF(Y , \mathfrak{s} , g) = \Sigma^{-V} C$ where $V$ is some finite dimensional representation of $G^E_{\mathfrak{s}}$ and $C$ is the Conley index of a finite-dimensional approximation of the gradient flow of the Chern--Simons--Dirac functional. 

Let $S$ be the multiplicative subset of $H^*_{G^E_{\mathfrak{s}}}(pt) \cong \mathbb{F}[s,u]$ generated by $u$ and $u+s^2$. We will apply the localisation theorem \cite[III (3.8)]{die} to the inclusion $\iota\colon C^{S^1} \to C$, where $C^{S^1}$ is the fixed point set of the $S^1$ subgroup. In order to apply the localisation theorem we need to check that every orbit type in $C \setminus C^{S^1}$ is of the form $G^E_{\mathfrak{s}}/H$ where $H \in \mathfrak{F}(S) = \{ H \; | \; S \cap Ker( H^*_{G^E_{\mathfrak{s}}}(pt) \to H^*_H(pt)) \neq \emptyset \}$. Since $S^1$ acts freely on $C \setminus C^{S^1}$ \cite[\textsection 2,3]{bh}, we see that the only possible stabilisers are $H = \{1\} , \langle \sigma \rangle, \langle (-1 , \sigma)\rangle$. In the case $H = \langle \sigma \rangle$, $u \in Ker ( H^*_{G^E_{\mathfrak{s}}}(pt) \to H^*_H(pt)) )$. In the case $H = \langle (-1 , \sigma) \rangle$, $u+s^2 \in Ker ( H^*_{G^E_{\mathfrak{s}}}(pt) \to H^*_H(pt)) )$. So the localisation theorem applies and says that $\iota^*\colon S^{-1} \widetilde{H}^*_{G^E_{\mathfrak{s}}}( C^{S^1}) \to S^{-1} \widetilde{H}^*_{G^E_{\mathfrak{s}}}(C)$ is an isomorphism. It can be shown (see \cite[\textsection 2,3]{bh}) that $C^{S^1}$ is the one-point compactification of a finite-dimensional representation of $G^E_{\mathfrak{s}}$, hence $\widetilde{H}^*_{G^E_{\mathfrak{s}}}(C^{S^1}) \cong \mathbb{F}[s,u] \theta$ for some $\theta$, by the Thom isomorphism. So the localisation theorem gives an isomorphism $S^{-1} \widetilde{H}^*_{G^E_{\mathfrak{s}}}(pt) \cong S^{-1} \mathbb{F}[s,u]\theta$. Applying Lemma \ref{lem:loc} gives
\begin{align*}
\widehat{ u^{-1} \widetilde{H}^*_{G^E_{\mathfrak{s}}}( C ) } &\cong \widehat{ S^{-1} \widetilde{H}^*_{G^E_{\mathfrak{s}}}(C)} \\
& \cong \widehat{ S^{-1} \mathbb{F}[s,u] \theta } \\
& \cong (\mathbb{F}[u,u^{-1}])[[s]]\theta.
\end{align*}

The result follows since $HSW^*_E(Y , \mathfrak{s} , \sigma)$ and $\widetilde{H}^*_{G^E_{\mathfrak{s}}}(C)$ are equal, up to a grading shift.
\end{proof}

\begin{remark}
In \cite{bh} the localisation theorem \cite[III (3.8)]{die} was incorrectly applied to the multiplicative subset generated by $u$ alone. However this error does not affect any of the results of \cite{bh} as we now explain. We saw in the proof of Proposition \ref{prop:loc}, in the case $G = \mathbb{Z}_2$, one needs to use the multiplicative subset generated by $u$ and $u + s^2$. More generally, for any finite group $G$, one should localise with respect to the multiplicative set $S$ generated by elements of the form $u + a$, where $a$ ranges over $H^2_G(pt)$. However, the analogue of Lemma \ref{lem:loc} shows that localising with respect to $S$ and localising with respect to $u$ alone become isomorphic upon completion in the $I^*$-adic topology. Now we observe that the definition of the $\delta$-invariants \cite[Definition 3.7]{bh} in that paper only uses truncation to finite order in the $I^*$-adic topology. Hence they are unaffected by taking a completion.
\end{remark}

The element $\theta$ in Proposition \ref{prop:loc} is not uniquely determined, however it is unique modulo $I^*_T$ and powers of $u$ (or $w_2$ or $q$). That is, if $\theta, \theta'$ are two elements for which the isomorphism in Proposition \ref{prop:loc} holds, then $\theta' = \xi^k \theta \; ({\rm mod} \; I^*_T)$ for some $k \in \mathbb{Z}$, where $\xi = u, w_2$, or $q$ depending on $T$.

\begin{definition}\label{def:delta}
Let $Y$ be a rational homology $3$-sphere, $\sigma$ a smooth, orientation preserving diffeomorphism and $\mathfrak{s}$ a spin$^c$-structure of type $T \in \{E,R,S\}$. Set $\xi = u$, $w_2$, or $q$ according to whether $T$ is $E,R$, or $S$. Let 
\[
\varphi\colon \widehat{ \xi^{-1} HSW^*_T(Y , \mathfrak{s} , \sigma)} \to \widehat{H}^*_T [ \xi , \xi^{-1}]\theta
\]
be an isomorphism as in Proposition \ref{prop:loc}. Let $c \in H^*_T$ be a non-zero homogeneous element of degree $|c|$. Define
\[
d^T_c(Y ,\mathfrak{s} , \sigma) = \min\{ i \; | \; \exists x \in HSW^i_T(Y , \mathfrak{s} , \sigma) \;  \xi^n \varphi(x) = c \xi^k \theta \; ({\rm mod} \; (I^*_T)^{|c|+1})  \} - |c|.
\]
It is easily verified that this does not depend on the choice of $\varphi$ or $\theta$.

We set $d^T_0(Y , \mathfrak{s} , \sigma) = 0$. Finally, if $c \in H^*_T$ is any non-zero element, write $c = c_0 + c_1 + \cdots + c_k$, where $c_i$ has degree $i$. Then we set $d^T_c(Y , \mathfrak{s} , \sigma) = \max_i \{ d^T_{c_i}(Y , \mathfrak{s} , \sigma) \}$.

For convenience we also define $\delta^T_{c}(Y , \mathfrak{s} , \sigma)$ by setting $\delta^T_c(Y , \mathfrak{s} , \sigma) = d^T_c(Y , \mathfrak{s} , \sigma)/2$.

\end{definition}

The following properties of the delta-invariants follows by a straightforward extension of \cite[\textsection 3.5-3.7]{bh}.

\begin{proposition}\label{prop:deltaprop}
The $\delta$-invariants satisfy the following properties:
\begin{itemize}
\item[(1)]{$\delta^T_{c_1 + c_2}(Y , \mathfrak{s} , \sigma) \le \max\{ \delta^T_{c_1}(Y , \mathfrak{s} , \sigma) + \delta^T_{c_2}(Y , \mathfrak{s} , \sigma) \}$.}
\item[(2)]{$\delta^T_{c_1 c_2}(Y , \mathfrak{s} , \sigma) \le \min\{ \delta^T_{c_1}(Y , \mathfrak{s} , \sigma) , \delta^T_{c_2}(Y , \mathfrak{s} , \sigma) \}$.}
\item[(3)]{$\delta^T_{1}(Y , \mathfrak{s} , \sigma) \ge \delta(Y , \mathfrak{s})$, where $1 \in H^*_T$ is the identity element and $\delta(Y , \mathfrak{s})$ is the ordinary $\delta$-invariant of $(Y , \mathfrak{s})$.}
\item[(4)]{If $Y$ is an $L$-space with respect to $\mathfrak{s}$ and $\mathbb{F}$, then $\delta^T_c(Y , \mathfrak{s} , \sigma) = \delta(Y , \mathfrak{s})$ for all non-zero $c$.}
\item[(5)]{For all $T$ and $c$, $\delta^T_c(Y  , \mathfrak{s} , \sigma) = \delta(Y , \mathfrak{s}) \; ({\rm mod} \; \mathbb{Z})$. Furthermore $\delta^S_c(Y  , \mathfrak{s} , \sigma) = \mu(Y , \mathfrak{s}) \; ({\rm mod} \; 2\mathbb{Z})$ where $\mu(Y , \mathfrak{s})$ is the generalised Rokhlin invariant, as defined in \cite{man2}.}
\end{itemize}

\end{proposition}

\begin{remark}\label{rem:disjoint}
We can extend the definition of (ordinary and equivariant) Seiberg--Witten Floer cohomology and (ordinary and equivariant) delta-invariants to the case that $Y$ is a (possibly empty) disjoint union of rational homology spheres, provided that $\sigma$ preserves the connected components of $Y$. To do this, suppose that $Y = \cup_{i=1}^{n} Y_i$ where the $Y_i$ are the connected components of $Y$. Then we simply define $SWF(Y , \mathfrak{s} , g)$ to be the smash product $SWF(Y , \mathfrak{s} , g) = \wedge_{i=1}^{n} SWF(Y , \mathfrak{s}|_{Y_i} , g|_{Y_i})$ (in the case $Y$ is empty, $SWF(Y , \mathfrak{s} , g) = S^0$). Then we can define $HSW^*(Y , \mathfrak{s}) = \widetilde{H}^{*+2n(Y , \mathfrak{s} , g)}_{S^1}( SWF(Y , \mathfrak{s} , g))$ and when $\mathfrak{s}$ has type $T$, $HSW^*_T(Y , \mathfrak{s})$ is defined similarly. Then we can define the (ordinary and equivariant) delta-invariants of $Y$ in exactly the same way as in the connected case. It follows easily from the Eilenberg--Moore spectral sequence that $\delta(Y , \mathfrak{s}) = \sum_{i=1}^{n} \delta( Y_i , \mathfrak{s}|_{Y_i})$.
\end{remark}

\begin{definition}
Let $Y$ be a rational homology $3$-sphere, $\sigma$ a smooth, orientation preserving diffeomorphism and $\mathfrak{s}$ a spin$^c$-structure of type $T \in \{E,R,S\}$. We define the following invariants:
\begin{itemize}
\item[(1)]{If $T=E$, we set $\delta^E_j(Y  , \mathfrak{s} , \sigma) = \delta^E_{s^j}(Y  , \mathfrak{s} , \sigma)$ for $j \ge 0$.}
\item[(2)]{If $T=R$, we set $\delta^R_j(Y  , \mathfrak{s} , \sigma) = \delta^R_{w_1^j}(Y  , \mathfrak{s} , \sigma)$ for $j \ge 0$.}
\item[(3)]{If $T=S$, we set $\delta^S_{i,j}(Y  , \mathfrak{s} , \sigma) = \delta^S_{ v^i (v+s)^j }(Y  , \mathfrak{s} , \sigma)$ for $i,j \ge 0$.}
\end{itemize}

\end{definition}

Note that in type $S$, since $v(v+s)^2 = 0$, we have $\delta^S_{i,j}(Y  , \mathfrak{s} , \sigma) = -\infty$ unless $i = 0$ or $j \le 1$. So the interesting invariants in this case are of the form $\delta^S_{0,j}(Y  , \mathfrak{s} , \sigma)$, $\delta^S_{j,0}(Y  , \mathfrak{s} , \sigma)$ or $\delta^S_{j,1}(Y  , \mathfrak{s} , \sigma)$.

\begin{remark}
Suppose that $Y$ is an integral homology $3$-sphere with an orientation preserving involution $\sigma$. Then $Y$ has a unique spin$^c$-structure $\mathfrak{s}$ which has all three types with respect to $\sigma$. In this case we will sometimes omit $\mathfrak{s}$ from the notation.
\end{remark}

\begin{proposition}\label{prop:deltaprop2}
The delta-invariants have the following properties:
\begin{itemize}
\item[(1)]{For $T = E$ or $R$, the sequence $\delta^T_j(Y , \mathfrak{s} , \sigma)$ is decreasing and there is an $N \ge 0$ such that $\delta^T_j(Y , \mathfrak{s} , \sigma)$ is independent of $j$ for $j \ge N$.}
\item[(2)]{For $T = S$, $\delta^S_{i,j}(Y , \mathfrak{s} , \sigma)$ is decreasing in the sense that $\delta^S_{i',j'}(Y , \mathfrak{s} , \sigma ) \le \delta^S_{i,j}(Y , \mathfrak{s} , \sigma )$ whenever $i' \ge i$, $j' \ge j$. There is an $N \ge 0$ such that $\delta^S_{0,j}(Y , \mathfrak{s} , \sigma )$, $\delta^S_{j,0}(Y , \mathfrak{s} , \sigma )$ and $\delta^S_{j,1}(Y , \mathfrak{s} , \sigma )$ are independent of $j$ for $j \ge N$.}
\item[(3)]{For $T = E$ or $R$, $\delta^T_i(Y , \mathfrak{s} , \sigma ) + \delta^T_j(-Y , \mathfrak{s} , \sigma ) \ge 0$ for all $i,j$. Similarly for $T = S$, $\delta^S_{i,j}(Y , \mathfrak{s} , \sigma ) + \delta^S_{k,l}(-Y , \mathfrak{s} , \sigma ) \ge 0$ for all $i,j,k,l$ with $i=k=0$ or $j+l \le 1$.}
\end{itemize}

\end{proposition}
\begin{proof}
Once again, these properties follow from a straightforward extension of the proofs in \cite{bh}.
\end{proof}

In light of Proposition \ref{prop:deltaprop2}, we can make the following definitions:
\begin{definition}\label{def:j}
Let $Y$ be a rational homology $3$-sphere, $\sigma$ a smooth, orientation preserving diffeomorphism and $\mathfrak{s}$ a spin$^c$-structure of type $T \in \{E,R,S\}$. We define the following invariants:
\begin{itemize}
\item[(1)]{If $T = E$ or $R$, we define $\delta^T_\infty(Y , \mathfrak{s} , \sigma) = \lim_{j \to \infty} \delta^T_j(Y , \mathfrak{s} , \sigma)$ and we let $j^T(Y , \mathfrak{s} , \sigma)$ denote the smallest $j \ge 0$ such that $\delta^T_j(Y , \mathfrak{s} , \sigma) = \delta^T_\infty(Y , \mathfrak{s} , \sigma)$.}
\item[(2)]{If $T=S$, we define $\delta^S_{0,\infty}(Y , \mathfrak{s} , \sigma) = \lim_{j\to \infty} \delta^S_{0,j}(Y , \mathfrak{s} , \sigma)$. We define $\delta^S_{\infty,0}(Y,\mathfrak{s},\sigma)$ and $\delta^S_{\infty,1}(Y,\mathfrak{s},\sigma)$ similarly.}
\end{itemize}

\end{definition}

Let $Y_1,Y_2$ be rational homology $3$-spheres equipped with orientation preserving, smooth involutions $\sigma_1, \sigma_2$. Assume that the fixed point sets of $\sigma_1, \sigma_2$ are non-empty. Since $\sigma_1, \sigma_2$ act orientation preservingly, their fixed point sets are $1$-dimensional. This implies that $\sigma_1, \sigma_2$ have the same local form about any fixed point. Let $y_1 \in Y_1$, $y_2 \in Y_2$ be fixed points. Then we can remove $\sigma_i$-invariant balls $B_1, B_2$ around $y_1, y_2$ and identify their boundaries orientation reversingly to form the $\mathbb{Z}_2$-equivariant connected sum $Y = Y_1 \# Y_2$ with involution $\sigma = \sigma_1 \# \sigma_2$. The construction of $\sigma$ depends on the choice of points $y_1, y_2$, but we do not indicate this in the notation.

Let $\mathfrak{s}_1, \mathfrak{s}_2$ be $\sigma_i$-invariant spin$^c$-structures on $Y_1,Y_2$ and set $\mathfrak{s} = \mathfrak{s}_1 \# \mathfrak{s}_2$. If $\mathfrak{s}_1, \mathfrak{s}_2$ both have type $T \in \{E,R,S\}$, then so does $\mathfrak{s}$. The proof of \cite[Proposition 3.1]{bar} can easily be adapted and gives the following.

\begin{proposition}\label{prop:csum}
If $T = E$ or $R$, then
\[
\delta^T_{i+j}(Y , \mathfrak{s} , \sigma ) \le \delta^T_{i}(Y_1 , \mathfrak{s}_1 , \sigma_1 ) + \delta^T_{j}(Y_2 , \mathfrak{s}_2 , \sigma_2 )
\]
for all $i,j \ge 0$. If $T = S$, then
\[
\delta^S_{i+k,j+l}(Y , \mathfrak{s} , \sigma ) \le \delta^S_{i,j}(Y_1 , \mathfrak{s}_1 , \sigma_1 ) + \delta^S_{k,l}(Y_2 , \mathfrak{s}_2 , \sigma_2 )
\]
for all $i,j,k,l \ge 0$ such that $i=k=0$ or $j+l \le 1$.
\end{proposition}

Let $l(Y,\mathfrak{s}) \in \mathbb{Q}$ denote the lowest degree in which $HSW^*(Y , \mathfrak{s})$ is non-zero (or equivalently, the lowest degree in which $HF^+_*(Y,\mathfrak{s})$ is non-zero). From \cite[Corollary 6.3]{os3} we see that $l(Y_1 \# Y_2) \ge l(Y_1) + l(Y_2)$.

\begin{proposition}\label{prop:jbound}
Let $Y$ be a rational homology $3$-sphere, $\sigma$ a smooth, orientation preserving diffeomorphism and $\mathfrak{s}$ a spin$^c$-structure of type $T \in \{E,R\}$. Then $\delta^T_j(Y,\mathfrak{s},\sigma) \ge (l(Y,\mathfrak{s}) - j)/2$ and $j^T(Y,\mathfrak{s},\sigma) \ge l(Y,\mathfrak{s}) - 2\delta^T_\infty(Y,\mathfrak{s},\sigma)$. Similarly $\delta^S_{i,j}(Y,\mathfrak{s},\sigma) \ge (l(Y,\mathfrak{s}) - i-j)/2$.
\end{proposition}
\begin{proof}
Set $\delta = \delta^T_j(Y , \mathfrak{s} , \sigma)$. By the definition of $\delta^T_j(Y , \mathfrak{s} , \sigma)$, there exists some $x \in HSW_T^{2 \delta + j}(Y , \mathfrak{s} , \sigma)$ such that $\xi^n \varphi(x) = c^j \xi^k \theta \; ({\rm mod} \; (I^*_T)^{j+1})$ where $\xi, \varphi, \theta$ are as in Definition \ref{def:delta} and $c = s$ if $T=E$, $c = w_1$ if $T=R$. But from the spectral sequence of Theorem \ref{thm:ss} we see that the lowest degree in which $HSW^*_T(Y , \mathfrak{s} , \sigma)$ is non-zero is at least $l(Y,\mathfrak{s})$. Hence $2\delta + j \ge l(Y , \mathfrak{s})$, or $\delta^T_j(Y,\mathfrak{s},\sigma) \ge (l(Y,\mathfrak{s}) - j)/2$. In the case $j = j^T(Y , \mathfrak{s} , \sigma)$ we have $\delta = \delta^T_j(Y , \mathfrak{s} , \sigma) = \delta^T_\infty(Y , \mathfrak{s} , \sigma)$ and thus $j^T(Y,\mathfrak{s},\sigma) \ge l(Y,\mathfrak{s}) - 2\delta^T_\infty(Y,\mathfrak{s},\sigma)$. The proof that $\delta^S_{i,j}(Y,\mathfrak{s},\sigma) \ge (l(Y,\mathfrak{s}) - i-j)/2$ follows by essentially the same argument.
\end{proof}

\subsection{Behaviour under cobordisms}\label{sec:cobord}

The most important property of the delta-invariants is their behaviour under cobordisms. Suppose that $W$ is smooth, compact, oriented $4$-manifold with boundary a disjoint union of rational homology spheres and with $b_1(W) = 0$. Suppose that $\sigma$ is an orientation preserving smooth involution on $W$. Let $\mathfrak{s}$ be a spin$^c$-structure on $W$ of type $T$. Since the boundary of $W$ is a union of rational homology $3$-spheres, we have $H^2(W , \partial W ; \mathbb{R}) \cong H^2(W ; \mathbb{R})$ and hence we obtain a non-degenerate intersection form on $H^2(W ; \mathbb{R})$. Let $H^+(W)$ denote a $\sigma$-invariant maximal positive definite subspace of $H^2(W ; \mathbb{R})$. Since the space of $\sigma$-invariant, maximal positive definite subspaces of $H^2(W ; \mathbb{R})$ is connected, it follows that the isomorphism class of $H^+(W)$ as a $\mathbb{Z}_2$-module does not depend on the choice of subspace. 

For $(W , \mathfrak{s})$ as above, we define $\delta(W , \mathfrak{s}) \in \mathbb{Q}$ by
\[
\delta(W , \mathfrak{s}) = \frac{c_1(\mathfrak{s})^2 - \sigma(W)}{8}.
\]

The following result is a straightforward adaptation of \cite[Theorem 4.1, Theorem 5.3]{bh}:

\begin{theorem}\label{thm:froy}
Let $W$ be a smooth, compact, oriented $4$-manifold with boundary and with $b_1(W) = 0$. Suppose that $\sigma$ is an orientation preserving smooth involution on $W$. Let $\mathfrak{s}$ be a spin$^c$-structure on $W$ of type $T \in \{E,R,S\}$. Suppose each component of $\partial W$ is a rational homology sphere and that $\sigma$ sends each component to itself. 

Suppose that $\partial W = Y_1 \cup -Y_0$. Then:
\begin{itemize}
\item[(1)]{In type $T=E$, suppose that the $\sigma$-invariant subspace of $H^2(W ; \mathbb{R})$ is negative definite. Then
\[
\delta^E_{j+b_+(W)}(Y_0 , \mathfrak{s}|_{Y_0} , \sigma|_{Y_0}) + \delta(W , \mathfrak{s}) \le 
\delta^E_j(Y_1, \mathfrak{s}|_{Y_1} , \sigma|_{Y_1})
\]
for all $j \ge 0$.
}
\item[(2)]{In type $T=R$, suppose that the $\sigma$-anti-invariant subspace of $H^2(W ; \mathbb{R})$ is negative definite. Then
\[
\delta^R_{j+b_+(W)}(Y_0 , \mathfrak{s}|_{Y_0} , \sigma|_{Y_0}) + \delta(W , \mathfrak{s}) \le 
\delta^R_j(Y_1, \mathfrak{s}|_{Y_1} , \sigma|_{Y_1})
\]
for all $j \ge 0$.
}
\item[(3)]{In type $T=S$, let $b_+(X)^\sigma, b_+(X)^{-\sigma}$ denote the dimensions of the $\sigma$-invariant/anti-invariant subspaces of $H^+(X)$. Then
\[
\delta^S_{i + b_+(W)^{\sigma} , j + b_+(W)^{-\sigma}}(Y_0 , \mathfrak{s}|_{Y_0} , \sigma|_{Y_0}) + \delta(W , \mathfrak{s}) \le 
\delta^S_{i,j}(Y_1, \mathfrak{s}|_{Y_1} , \sigma|_{Y_1})
\]
for all $i,j \ge 0$ such that either $i + b_+(W)^{\sigma} = 0$ or $j + b_+(W)^{-\sigma} \le 1$.
}
\end{itemize}

\end{theorem}

\begin{remark}
In Theorem \ref{thm:froy}, we can allow $Y_0,Y_1$ to be a (possibly empty) union of rational homology spheres. In this case the delta-invariants of $Y_i$ should be defined as per Remark \ref{rem:disjoint}. In particular, if $Y_0$ or $Y_1$ is empty, then the corresponding delta-invariant is zero.
\end{remark}

\begin{remark}\label{rem:erhci}
Theorem \ref{thm:froy} implies in particular that $\delta^E_j, \delta^R_j, \delta^S_{i,j}$ are equivariant rational homology cobordism invariants. To be precise, this means the following. Let $Y_0, Y_1$ be rational homology $3$-spheres. For $i=0,1$ let $\sigma_i$ be a smooth orientation preserving involution on $Y_i$ and let $\mathfrak{s}_i$ be a spin$^c$-structure on $Y_i$ of type $T$ (same type for $Y_0$ and $Y_1$). Let $W$ be a smooth, compact, oriented cobordism from $Y_0$ to $Y_1$ and suppose that the inclusions $Y_i \to W$ for $i=0,1$ induce isomorphisms in rational cohomology. Suppose that there is an orientation preserving smooth involution $\sigma$ on $W$ which restricts to $\sigma_i$ on $Y_i$ and a spin$^c$-structure on $W$ of type $T$ which restricts to $\mathfrak{s}_i$ on $Y_i$. Then $\delta^T_j(Y_0 , \mathfrak{s}_0 , \sigma_0) = \delta^T_j(Y_1 , \mathfrak{s}_1 , \sigma_1)$ for all $j$ if $T = E,R$ or $\delta^S_{i,j}(Y_0 , \mathfrak{s}_0 , \sigma_0) = \delta^S_{i,j}(Y_1 , \mathfrak{s}_1 , \sigma_1)$ for all $i,j$ if $T = S$.
\end{remark}

\section{Constructions}\label{sec:con}

\subsection{Plumbing graphs}\label{sec:pg}

Let $\Gamma$ be a plumbing graph. By definition this means that $\Gamma$ is a finite undirected graph with no cycles and at most one edge between any two vertices. $\Gamma$ can be disconnected in which case we will interpret the corresponding plumbed $4$-manifold to be a disjoint union. Let $x_1, \dots , x_n$ be the vertices of $\Gamma$. Associate to each vertex $x_i$ an integer weight $d_i$, which we will refer to as the {\em degree} of the vertex.

Let us recall the construction of the plumbing according to $\Gamma$ \cite[\textsection 2]{or}, \cite[\textsection 1.1.9]{sav}. Let $D^2$ denote the closed unit disc in $\mathbb{C}$. Associate to each vertex $x_i$ the $D^2$-bundle $X_i$ over $S^2$ with Euler class $d_i$. If $x_i$ has $k_i$ edges connected to it, choose $k_i$ disjoint closed discs $D_{ij}$ in the base of $X_i$, one for each $j$ such that $x_j$ is joined to $x_i$ by an edge. Since $D_{ij}$ is a disc, $X_i|_{D_{ij}}$ can be identified with the trivial disc bundle $X_i|_{D_{ij}} \cong D^2 \times D^2$. Now glue $X_i$ and $X_j$ together by identifying $X_i|_{D_{ij}}$ and $X_j|_{D_{ji}}$, exchanging the base and fibre coordinates. This construction produces a manifold with corners. Smoothing the corners yields a compact, oriented smooth $4$-manifold $X_\Gamma$ with boundary. We call $X_\Gamma$ a {\em plumbing (according to $\Gamma$)}. The boundary $Y_\Gamma = \partial X_\Gamma$ will be called the {\em boundary of the plumbing (according to $\Gamma$)}. The connected components of $X_\Gamma$ and $Y_\Gamma$ are in bijection with the components of $\Gamma$ and each component of $X_\Gamma$ is simply-connected.

The cohomology group $H^2(X_\Gamma , Y_\Gamma ; \mathbb{Z})$ has a basis $\{ e_i \}$ given by the Poincar\'e duals of the zero sections of the disc bundles $X_i \to S^2$. Let $A(\Gamma) = A(\Gamma)_{ij}$ be the matrix where $A(\Gamma)_{ii} = d_i$, $A(\Gamma)_{ij} = 1$ if $i,j$ are distinct and $x_i,x_j$ are joined by an edge and $A(\Gamma)_{ij} = 0$ otherwise. The intersection form on $H^2(X_\Gamma , Y_\Gamma ; \mathbb{Z})$ is given by $\langle e_i , e_j \rangle = A(\Gamma)_{ij}$. If $\Gamma$ is connected, then $Y_\Gamma$ is a rational homology sphere if and only if $\det(A(\Gamma)) \neq 0$. If this holds then $| H_1(Y_\Gamma ; \mathbb{Z}) | = |\det(A(\Gamma))|$. More generally, each component of $Y_\Gamma$ is a rational homology sphere if and only if $\det(A(\Gamma)) \neq 0$. All plumbing graphs considered in this paper will be assumed to satisfy this condition. We define $\sigma(\Gamma)$, $b_+(\Gamma), b_-(\Gamma)$ to equal the corresponding invariants of $X_\Gamma$.

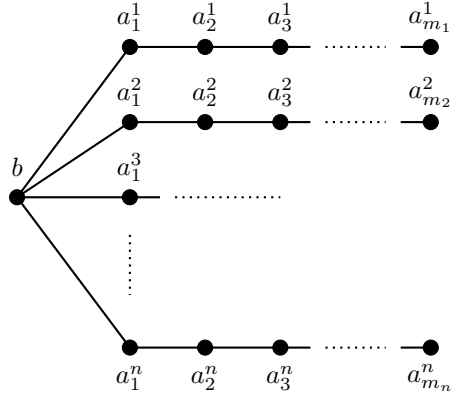
\begin{figure}

\begin{tikzpicture}
\draw[thick] (-0.5,0) -- (1,2) ;
\draw[thick] (1,2) -- (3.4,2) ;
\draw[thick] (4.6,2) -- (5,2) ;
\draw [dotted, thick] (3.6,2) -- (4.4,2);
\draw [fill] (-0.5,0) circle(0.1);
\draw [fill] (1,2) circle(0.1);
\draw [fill] (2,2) circle(0.1);
\draw [fill] (3,2) circle(0.1);
\draw [fill] (5,2) circle(0.1);
\draw[thick] (-0.5,0) -- (1,1) ;
\draw[thick] (1,1) -- (3.4,1) ;
\draw[thick] (4.6,1) -- (5,1) ;
\draw [dotted, thick] (3.6,1) -- (4.4,1);
\draw [fill] (1,1) circle(0.1);
\draw [fill] (2,1) circle(0.1);
\draw [fill] (3,1) circle(0.1);
\draw [fill] (5,1) circle(0.1);
\draw[thick] (-0.5,0) -- (1,-2) ;
\draw[thick] (1,-2) -- (3.4,-2) ;
\draw[thick] (4.6,-2) -- (5,-2) ;
\draw [dotted, thick] (3.6,-2) -- (4.4,-2);
\draw [fill] (1,-2) circle(0.1);
\draw [fill] (2,-2) circle(0.1);
\draw [fill] (3,-2) circle(0.1);
\draw [fill] (5,-2) circle(0.1);
\draw [thick] (-0.5,0) -- (1,0.0) ;
\draw [fill] (1,0) circle(0.1);
\draw [dotted, thick] (1,-0.5) -- (1,-1.3);
\draw[thick] (1,0) -- (1.4,0) ;
\draw [dotted, thick] (1.6,0) -- (3,0);
\node at (-0.5,0.4) {$b$};
\node at (1,2.4) {$a^1_1$};
\node at (2,2.4) {$a^1_2$};
\node at (3,2.4) {$a^1_3$};
\node at (5,2.4) {$a^1_{m_1}$};
\node at (1,1.4) {$a^2_1$};
\node at (2,1.4) {$a^2_2$};
\node at (3,1.4) {$a^2_3$};
\node at (5,1.4) {$a^2_{m_2}$};
\node at (1,-2.4) {$a^n_1$};
\node at (2,-2.4) {$a^n_2$};
\node at (3,-2.4) {$a^n_3$};
\node at (5,-2.4) {$a^n_{m_n}$};
\node at (1,0.4) {$a^3_1$};
\end{tikzpicture}
\caption{Star-shaped plumbing graph.}\label{fig:star}
\end{figure}

A particularly interesting class of plumbed $3$-manifold are those given by plumbing on a star-shaped graph, see Figure \ref{fig:star}. All such plumbings are Seifert fibre spaces. Given integers $b , \alpha_1, \dots , \alpha_n , \beta_1 , \dots , \beta_n$ where $\alpha_i , \beta_i$ are coprime and $\beta_i \neq 0$ for all $i$, define the Seifert manifold $Y(b ; (\alpha_1 , \beta_1) , \dots , (\alpha_n , \beta_n))$ to be surgery on the framed link given in Figure \ref{fig:seifert}. This is a rational homology sphere if and only if $b - \sum_{i=1}^{n} \beta_i/\alpha_i \neq 0$. The Seifert structure on $Y = Y(b ; (\alpha_1 , \beta_1) , \dots , (\alpha_n , \beta_n))$ completely determines the values of $b , \alpha_1, \dots , \alpha_n , \beta_1 , \dots , \beta_n$ up to a sequence of transformations of the form $b \mapsto b+k$, $\beta_i \mapsto \beta_i + k\alpha_i$ for some $i$.

For integers $a_0, \dots , a_m$, let $[a_1 , a_2 , \dots , a_m]$ be the negative continued fraction
\[
[a_1,a_2,\ldots,a_m] =a_1-\frac{1}{a_2-\raisebox{-3mm}{$\ddots$
\raisebox{-2mm}{${-\frac{1}{\displaystyle{a_m}}}$}}}.
\]
The Euclidean algorithm implies that any rational number can be written as a negative continued fraction. Now write $\alpha_i/\beta_i = [a^i_1 , a^i_2 , \dots , a^i_{m_i}]$ for some integers $\{a^i_j\}$. Then we claim that the boundary $Y_\Gamma$ of the plumbing in Figure \ref{fig:star} is the Seifert manifold $Y(b ; (\alpha_1 , \beta_1) , \dots , (\alpha_n , \beta_n))$. To see this, recall that the boundary of a plumbing graph $\Gamma$ can also be constructed as surgery on a link whose components are all unknots \cite[\textsection 1.1.9]{sav}. Applying this to the pluming graph of Figure \ref{fig:star} and repeatedly performing the reverse slam dunk operation (\cite[\textsection 5.3]{gs}), we obtain the surgery diagram in Figure \ref{fig:seifert}.

\begin{figure}[h]

\begin{tikzpicture}
  \draw[thick] ($(0, 0) + (4:3cm and 1.5cm)$(P) arc
  (4:176:3cm and 1.5cm);
  \draw[thick] ($(0, 0) + (184:3cm and 1.5cm)$(P) arc
  (184:229:3cm and 1.5cm);
  \draw[thick] ($(0, 0) + (233:3cm and 1.5cm)$(P) arc
  (233:358:3cm and 1.5cm);
  \draw[thick] ($(-3, -1.4) + (27:0.6cm and 1.4cm)$(P) arc
  (27:375:0.6cm and 1.4cm);
  \draw[thick] ($(-1.4, -2) + (28:0.6cm and 1.4cm)$(P) arc
  (28:378:0.6cm and 1.4cm);
  \draw[thick] ($(3, -1.4) + (-195:0.6cm and 1.4cm)$(P) arc
  (-195:153:0.6cm and 1.4cm);
\node at (0,2) {$b$};
\node at (-4,-1.5) {$\dfrac{\alpha_1}{\beta_1}$};
\node at (-2.2,-3.5) {$\dfrac{\alpha_2}{\beta_2}$};
\node at (0.3,-2.4) {$\bullet$};
\node at (0.6,-2.4) {$\bullet$};
\node at (0.9,-2.4) {$\bullet$};
\node at (1.2,-2.4) {$\bullet$};
\node at (4,-1.5) {$\dfrac{\alpha_n}{\beta_n}$};
\end{tikzpicture}
\caption{Surgery diagram for $Y(b ; (\alpha_1 , \beta_1 ) , \dots , (\alpha_n , \beta_n) )$.}\label{fig:seifert}
\end{figure}
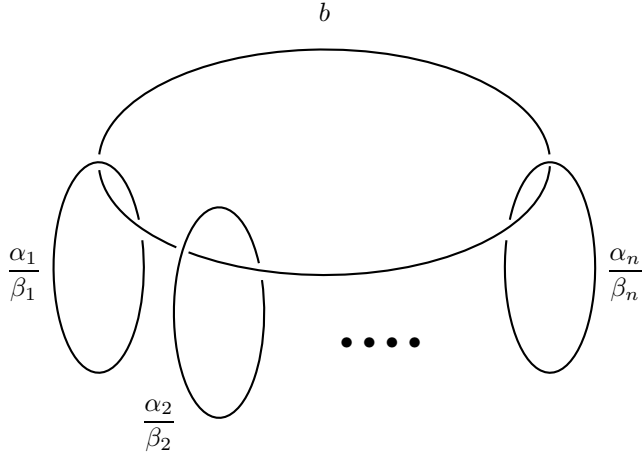

We will be particularly interested in the case of plumbing graphs $\Gamma$ where every vertex has even degree, for this is precisely the condition for the plumbing $X_\Gamma$ to be spin.

\begin{lemma}\label{lem:even}
Let $r = p/q \in \mathbb{Q}$ be a rational number where $p,q$ are coprime and $p$ or $q$ is even. Then there exists even integers $a_1, \dots , a_n$ such that $r = [a_1, \dots , a_n]$.
\end{lemma}
\begin{proof}
Let $a$ be the closest even integer to $r$. Note that $a$ is unique, for if there are exactly two even integers closest to $r$ then $r$ is an odd integer, which is impossible since then $r = r/1 = p/q$, so $p = r$ and $q=1$ are both odd. If $r$ is an even integer, then $r = a$ and we are done. Otherwise $r = a - 1/r'$, where $r' = 1/(a-r)$. Write $p = \alpha q + b$ where $0 < b < q$. Then $a$ is either $\alpha$ or $\alpha + 1$. If $a = \alpha$, then $r' = -q/b$. If $a = \alpha + 1$, then $r' = q/(q-b)$. In either case the denominator of $r'$ is strictly less than $q$. Hence iterating this process will eventually terminate. 
\end{proof}

\begin{corollary}\label{cor:even}
Every Seifert manifold $Y = Y(b ; (a_1 , b_1) , \dots , (a_n , b_n))$ with at least one $a_i$ even is the boundary of a star-shaped plumbing graph with every vertex having even degree.
\end{corollary}
\begin{proof}
If $a_i$ and $b_i$ are both odd, replace $b_i$ by $b_i + a_i$ and $b$ by $b+1$. Hence we can assume for each $i$ that at least one of $a_i,b_i$ is even. After making these replacements, if $b$ is odd, choose an $i$ such that $a_i$ is even and replace $b_i$ by $b_i + a_i$ and $b$ by $b+1$. By Lemma \ref{lem:even} we may write $a_i/b_i = [a^i_1 , \dots , a^i_{m_i}]$ where $a^i_j$ is even for all $i,j$. Then $Y$ is the boundary of the plumbing graph given in Figure \ref{fig:star}, where all the degrees are even.
\end{proof}
 
By an {\em equivariant plumbing} we mean that we perform plumbing on a graph $\Gamma$ in such a way that each individual disc bundle has a group action and these actions fit together when we perform the plumbing. In what follows, we will consider a few different types of equivariant plumbing.

\subsubsection{Complex conjugation}\label{sec:cc}

Every plumbing $X_\Gamma$ can be made into a $\mathbb{Z}_2$-equivariant plumbing as follows. Identify $S^2$ with the Riemann sphere $\mathbb{CP}^1$ and identify the disc bundle $X_i \to S^2$ of Euler class $d_i$ with the unit disc bundle in the total space of the complex line bundle $\mathcal{O}(d_i) \to \mathbb{CP}^1$ (with respect to a choice of Hermitian metric). The action on $\mathbb{CP}^1$ by complex conjugation lifts to an antiholomorphic involution $c\colon \mathcal{O}(d_i) \to \mathcal{O}(d_i)$. By restriction this defines an orientation preserving involution $c\colon X_i \to X_i$. We assume that each disc $D_{ij} \subset \mathbb{CP}^1$ is chosen to be conjugation invariant. Then we can choose the trivialisation $X_i|_{D_{ij}} \cong D^2 \times D^2$ so that $c$ acts as complex conjugation on both factors. Since this commutes with the map $D^2 \times D^2 \to D^2 \times D^2$ which swaps base and fibre coordinates, we obtain an orientation preserving involution $c_\Gamma\colon X_\Gamma \to X_\Gamma$ which is given by $c$ on each disc bundle. We call $c_\Gamma$ a {\em complex conjugation} involution of the plumbing $X_\Gamma$. Note that $c_\Gamma$ is not uniquely determined because it depends on the relative locations of the discs $D_{ij}$ whose centres must lie on the circle $S^1 \subset \mathbb{CP}^1$ fixed by $c$. Thus the construction of $c_\Gamma$ requires that for each vertex $x_i$ we choose a cyclic ordering of the edges passing through $x_i$. This ordering specifies the order in which the discs $D_{ij}$ should be placed.

Since $c_\Gamma$ acts as complex conjugation on each disc bundle $X_i \to \mathbb{CP}^1$, it sends the zero section of $X_i$ to itself orientation reversingly. Hence $c_\Gamma$ acts as $-1$ on $H^2(X_\Gamma ; \mathbb{Z})$.

Since (each component of) $X_\Gamma$ is simply-connected, a spin$^c$-structure $\mathfrak{s}$ on $X_\Gamma$ is uniquely determined by its characteristic element $c(\mathfrak{s}) \in H^2(X_\Gamma ; \mathbb{Z})$. It follows that every spin$^c$-structure on $X_\Gamma$ is of type $R$ with respect to $c_\Gamma$. Furthermore, if all the degrees $d_i$ are even, then $X_\Gamma$ is spin and the unique spin$^c$-structure on $X_\Gamma$ with $c(\mathfrak{s}) = 0$ is of types $E$ and $S$ with respect to $c_\Gamma$.

\subsubsection{$S^1$-equivariant plumbing}\label{sec:S1plumb}

Let $D_d \to S^2$ denote the unit disc bundle over $S^2$ with Euler class $d$. Consider a circle action on $D_n$ which covers a circle action on $S^2$ and acts linearly on the fibres. The simplest case is to take a trivial action on $S^2$ in which case the action is given by scalar multiplication on the fibres with some arbitrary weight $w \in \mathbb{Z}$. If the action on $S^2$ is non-trivial, it must be conjugate to a rotation with some weight $m$. More precisely, let $S^1$ act on $S^2$ by rotation about some axis. The action has two fixed points $p_+,p_- \in S^2$ and the action on the tangent spaces at $p_+, p_-$ has weights $+1$, $-1$. We will regard this standard action as having weight $m=1$. Then a weight $m$ action is given by precomposing with the map $S^1 \to S^1$, $z \mapsto z^m$ of degree $m$. The resulting action of $S^1$ on $S^2$ has weight $m$ on $T_{p_+}S^2$ and weight $-m$ on $T_{p_-}S^2$. Note that the actions of weight $m$ and $-m$ are conjugate to each other through a map that interchanges $p_+$ and $p_-$. However for the purpose of equivariant plumbing we wish to distinguish the roles of the two points $p_+, p_-$, so it is convenient to distinguish between positive and negative values of $m$.

Now consider a lift of the weight $m$ circle action on $S^2$ to the total space of $D_d$. We will regard $S^2$ as $\mathbb{P}^1$ and take the weight $m$ circle action to be $\rho [u,v] = [u , \rho^m v]$. Then $p_+ = [1,0]$, $p_- = [0,1]$. Consider the projection map $\mathbb{C}^2 \setminus \{0\} \to \mathbb{P}^1$. The total space of $\mathbb{C}^2 \setminus \{0\}$ can be thought of as the total space of $\mathcal{O}(-1)$ minus the zero section (since $\mathcal{O}(-1)$ is the tautological line bundle). There is a lift of the $S^1$-action to $\mathcal{O}(-1)$ which is defined on the complement of the zero section by $\rho(u,v) = ( u , z^{m}v )$ and which clearly extends over the zero section. Taking tensor powers of this action we obtain an $S^1$-action on $\mathcal{O}(d)$ for any $n$. Let $w^+,w^-$ denote the weights of this action over $p_+,p_-$. It follows from this construction that $w^+ = 0, w^- = -dm$, hence $w^+ - w^- = dm$. Combining such an action with scalar multiplication of weight $k$ allows us to produce a new $S^1$-action with weights $(w^+ + k , w^- + k)$ for any $k$. Hence any solution to $w^+ - w^- = dm$ defines a circle action on $\mathcal{O}(d)$ (and by restriction on $D_d$) with weights $(w^+ , w^-)$.

Now suppose we attempt to undertake plumbing equivariantly. Consider a vertex $N_0$ with degree $d_0$ and with base weight $m$. If $m=0$ then the circle action is scalar multiplication on the fibres by some weight $w$. Since each point in the base is fixed, we can have as many edges leaving the vertex as we wish. It is clear from the requirement that the plumbing is equivariant that all adjacent vertices will have a non-trivial circle action on the base. Let $N_1$ be one such vertex adjacent to $N_0$. The circle action for $N_1$ will then have exactly two fixed points which are zeros in the fibres over $p_+, p_-$. To do the plumbing equivariantly we need to attach vertices using fixed points, thus the attachment must take place at either $p_+$ or $p_-$. As a convention we assume the attachment takes place at the point $p_-$ of $N_1$. The base and fibre weights of $N_0$ are $(0,w)$ (at any point). If the vertex $N_1$ has base weight $m_1$ and fibre weights $w^+_1, w^-_1$ over $p_+, p_-$, then the base and fibre weights over $p_-$ are $( -m_1 , w^-_1)$. The plumbing swaps base and fibre, hence $(0,w) = (w^-_1 , -m_1)$. That is, $m_1 = -w$, $w^-_1 = 0$. Then since $N_1$ has degree $d_1$, we have $w^+_1 - w^-_1 = d_1 m_1 = -d_1 w$. Hence
\[
\left[ \begin{matrix} m_1 \\ w_1^+ \end{matrix} \right] = \left[ \begin{matrix} 0 & -1 \\ 1 & -d_1 \end{matrix} \right] \left[ \begin{matrix} m_0 \\ w_0^+ \end{matrix} \right]
\]
where $m_0 = 0$, $w_0^+ = w$.

Similarly, if we have two vertices $N_1,N_2$ with circle actions with possibly non-zero base weights $m_1,m_2$, with degrees $d_1,d_2$ and with fibre weights $w^{\pm}_1 , w^{\pm}_2$ and we attach $p_+$ of $N_1$ to $p_-$ of $N_2$, then because the plumbing swaps base and fibre we get $(w_1^+ , m_1 ) = (-m_2 , w_2^-)$. Together with $w_2^+ - w_2^- = d_2 m_2$, this gives
\[
\left[ \begin{matrix} m_2 \\ w_2^+ \end{matrix} \right] = \left[ \begin{matrix} 0 & -1 \\ 1 & -d_2 \end{matrix} \right] \left[ \begin{matrix} m_1 \\ w_1^+ \end{matrix} \right].
\]

Now if we start with a vertex $N_0$ with base weight $m_0 = 0$, fibre weight $w_0$ (which we can assume is equal to $\pm 1$ since we want the generic orbit to be free) then for each vertex $N_1$ that we attach, the circle action must be non-trivial in the base of $N_1$. This means there are only two points $p_+, p_-$ at which we can attach vertices to $N_1$. Then since we use $p_-$ to attach $N_1$ to $N_0$, this leaves only one remaining point $p_+$ on which to do further attachments. Each time we attach an additional vertex, the exact same situation occurs, unless the circle action in the base is trivial. There are only two fixed points $p_+, p_-$ and we use $p_-$ to attach the new vertex to the previous one leaving only $p_+$ for further attachment.

From this we see that $S^1$-equivariant plumbing can be achieved for any star-shaped graph. Moreover, the weights of the circle action are completely determined, up to an overall sign. There are other graphs which can be $S^1$-equivariantly plumbed, but there is a non-trivial condition on the degrees of the vertices for this to happen (see \cite[\textsection 2]{or} for further details).

Consider a sequence of vertices $N_0 , N_1 , \dots , N_n$, where $m_0 = 0$, $w^+_0 = 1$ and $N_{i}$ is attached to $N_{i-1}$ by joining $p_+$ of $N_{i-1}$ to $p_-$ of $N_{i}$. To simplify notation write $w_i$ for $w_i^+$. Then
\[
\left[ \begin{matrix} m_i \\ w_i \end{matrix} \right] = \left[ \begin{matrix} 0 & -1 \\ 1 & -d_i \end{matrix} \right] \left[ \begin{matrix} m_{i-1} \\ w_{i-1} \end{matrix} \right].
\]
From which one finds (assuming the $m_i$ are all non-zero) that
\[
\frac{w_n}{m_n} = [d_n , d_{n-1} , \dots , d_2 , d_1 ].
\]
If $m_1, \dots , m_k$ are non-zero and $m_{k+1} = 0$, then since $m_{k+1} = -w_k$, we see that $[ d_k , d_{k-1} , \dots , d_2 , d_1 ] = 0$. When this happens we can have more than one vertex with a trivial circle action on the base.

\subsubsection{$\mathbb{Z}_2$-equivariant plumbing}\label{sec:z2p}

Consider $\mathbb{Z}_2$-equivariant plumbings. We will construct these along the same lines as $S^1$-equivariant plumbings but with the group $\mathbb{Z}_2$ in place of $S^1$. The main difference is that for $\mathbb{Z}_2$, the weights are now to be considered as integers mod $2$. More precisely, the trivial representation is assigned weight $0$ and the sign representation is assigned weight $1$. Our discussion of equivariant plumbing in Section \ref{sec:S1plumb} carries over almost verbatim except that the weights $m, w^+, w^-$ are now valued in $\mathbb{Z}_2$ instead of $\mathbb{Z}$. Suppose we attach vertex $N_1$ to $N_2$ such that if $m_1 \neq 0$, then the attachment takes place at $p_+$ on $N_1$ and if $m_2 \neq 0$ then the attachment takes place at $p_-$ on $N_2$. Then as before we have a relation
\[
\left[ \begin{matrix} m_2 \\ w_2 \end{matrix} \right] = \left[ \begin{matrix} 0 & 1 \\ 1 & d_2 \end{matrix} \right] \left[ \begin{matrix} m_{1} \\ w_{1} \end{matrix} \right]
\]
(we have omitted minus signs since we are working over $\mathbb{Z}_2$). But now $m_i,w_i \in \mathbb{Z}_2$ and they can not both be zero (otherwise the involution is trivial), so there are only three possibilities $(m_i,w_i) \in \{ (0,1) , (1,0) , (1,1) \}$. The case $(0,1)$ corresponds to a trivial action on the base.

We make a simplifying assumption that all the degrees $d_i$ are even. Then one of two things can happen:
\begin{itemize}
\item[(1)]{Every vertex has $m=w=1$. In particular the plumbing graph (if connected) is linear and so the boundary is a lens space.}
\item[(2)]{Every vertex has either $(m,w) = (0,1)$ or $(1,0)$ in an alternating fashion, that is, each edge connects a vertex of type $(0,1)$ to a vertex of type $(1,0)$. Equivalently, we have a bipartite graph, as in Figure \ref{fig:bipartite}. The vertices of type $(0,1)$ (coloured black) can have arbitrarily many edges but the vertices of type $(1,0)$ (coloured white) can have at most two edges.}
\end{itemize}

Case (1) is not particularly interesting since the boundary must be a lens space. On the other hand Case (2) is quite interesting since there is a wide variety of plumbing graphs of this form. 

Thus if $\Gamma$ is a plumbing graph where all the vertices have even degree and weights $m_i, w_i \in \mathbb{Z}_2$ can be assigned according to Case (1) or Case (2), then we can carry out the plumbing $\mathbb{Z}_2$-equivariantly. We will refer to $\Gamma$ as a {\em $\mathbb{Z}_2$-equivariant plumbing graph} and we refer to $X_\Gamma$ as a $\mathbb{Z}_2$-equivariant plumbing. The resulting involution on $X_\Gamma$ will be denoted as $m_\Gamma$. Since $m_\Gamma$ preserves the zero section of each disc bundle of $X_\Gamma$ in an orientation-preserving manner, we see that $m_\Gamma$ acts trivially on $H^2(X ; \mathbb{Z})$. Thus, every spin$^c$-structure on $X_\Gamma$ is of type $E$ with respect to $m_\Gamma$. Furthermore, since we assumed that each $d_i$ is even, $X_\Gamma$ is spin and the unique spin$^c$-structure $\mathfrak{s}$ with $c(\mathfrak{s}) = 0$ has types $R$ and $S$ with respect to $m_\Gamma$ (the involution $m_\Gamma$ is always odd since it is constructed from gluing together odd involutions on disc bundles).

We now describe a construction that gives rise to cobordisms between boundaries of plumbings (cf. \cite[\textsection 3]{sav0}). Let $\Gamma$ be a plumbing graph and let $\Gamma'$ be a subgraph of $\Gamma$ with the property that if $e$ is an edge of $\Gamma$ that does not belong to $\Gamma'$, then at most one vertex of $e$ belongs to $\Gamma'$. We will also assume that $\det(A(\Gamma')) \neq 0$. A subgraph satisfying this property will be said to be {\em admissible}. Recall that $X_\Gamma$ is constructed by plumbing together unit disc bundles $\{ X_i \}$. For each vertex $x_i$ of $\Gamma$ which belongs to $\Gamma'$, let $X'_i$ be the closed disc bundle over $S^2$ of radius $1/2$ and Euler class $d_i$. Plumbing together the $X'_i$ gives the plumbed $4$-manifold $X_{\Gamma'}$. Identifying $X'_i$ with a subspace of $X_i$, we obtain an inclusion $X_{\Gamma'} \to X_\Gamma$. Let $X_{\Gamma',\Gamma}$ denote the closure in $X_{\Gamma}$ of the complement of $X_{\Gamma'}$. Then $X_{\Gamma',\Gamma}$ is a compact, simply-connected $4$-manifold with boundary $Y_\Gamma \cup -Y_{\Gamma'}$. So we can regard $X_{\Gamma',\Gamma}$ as a cobordism from $X_{\Gamma'}$ to $X_{\Gamma}$.

Since $\det(A(\Gamma'))$ is assumed to be non-zero, the boundary of $X_{\Gamma'}$ is a union of rational homology spheres. Mayer--Vietoris gives $H^2(X_{\Gamma} ; \mathbb{Q}) \cong H^2( X_{\Gamma'} ; \mathbb{Q}) \oplus H^2(X_{\Gamma',\Gamma} ; \mathbb{Q})$ where the direct sum is orthogonal with respect to the intersection form. Hence we can identify $H^2( X_{\Gamma' , \Gamma} ; \mathbb{Q})$ with the orthogonal complement of $H^2(X_{\Gamma'} ; \mathbb{Q})$ in $H^2(X_\Gamma ; \mathbb{Q})$.

Observe that the cobordism $X_{\Gamma' , \Gamma'}$ is equivariant with respect to the complex conjugation involution defined in Section \ref{sec:cc}. Similarly, if $X_{\Gamma}$ is an $S^1$- or $\mathbb{Z}_2$-equivariant plumbing, then the same is true of $X_{\Gamma'}$ by restriction and $X_{\Gamma' , \Gamma}$ has an $S^1$- or $\mathbb{Z}_2$-action. 

Let $a_1, \dots , a_n$ be coprime positive integers. Define the Brieskorn homology sphere $\Sigma(a_1, \dots , a_n)$ to be the unique Seifert manifold $Y(b ; (a_1,b_1) , \dots , (a_n,b_n))$ for which $b - \sum_{i=1}^{n} b_i/a_i = -1/a_1\cdots a_n$. The Brieskorn spheres other than $S^3$ and $\Sigma(2,3,5)$ have an $\widetilde{SL(2,\mathbb{R})}$ geometry with symmety group $O(2)$ which combines the circle action of the Seifert fibration with complex conjugation (viewing $\Sigma(a_1, \dots , a_n)$ as the link of a complex singularity). Let $m \in O(2)$ correspond to the unique element of order $2$ within the circle subgroup and let $c \in O(2)$ denote complex conjugation. By \cite[Theorem 2.1]{mesc} (or \cite{blp,dl} for the spherical cases $S^3$ , $\Sigma(2,3,5)$), any smooth, orientation preserving action of a finite group on $Y = \Sigma(a_1, \dots , a_n)$ is conjugate to a subgroup of this $O(2)$-action (or to a subgroup of $SO(4)$ in the case of $S^3$). In $O(2)$ there are precisely two conjugacy classes of involution, represented by $m$ and $c$. Thus any smooth, orientation preserving involution on a Brieskorn sphere other than $S^3$ is conjugate to $m$ or $c$ (in the case of $S^3$, there is again two conjugacy classes of involutions in $SO(4)$, but this time $m$ and $c$ are both conjugate to $diag(1,1,-1,-1)$ since they are both odd involutions).

Suppose that $\Gamma$ is a star-shaped plumbing graph such that $Y_\Gamma = \Sigma(a_1, \dots , a_n)$. Then we obtain two involutions $m_\Gamma, c_\Gamma$. Since $\Gamma$ is star-shaped, we obtain an $S^1$-action on the plumbing. Together with complex conjugation, this defines an action of $O(2)$ on the plumbing which then restricts to an action of $O(2)$ on $Y_\Gamma$. This action preserves a Seifert structure on $Y_\Gamma$. If $Y_\Gamma$ is not $S^3$, then the Seifert structure is unique. It follows that $m_\Gamma$ is conjugate to $m$ and $c_\Gamma$ is conjugate to $c$. In the case of $S^3$, $m_\Gamma$ and $c_\Gamma$ are odd involutions, so they must be conjugate to $m$ and $c$.

\begin{theorem}\label{thm:briedelt}
Let $Y = \Sigma(a_1, \dots , a_n)$ where $a_1, \dots , a_n$ are coprime positive integers. Then
\begin{itemize}
\item[(1)]{$\delta^E_j(Y , c) = \delta^R_j(Y , m) =  -\overline{\mu}(Y)$ for all $j \ge 1$, where $\overline{\mu}$ is the Neumann--Siebenmann invariant.}
\item[(2)]{$\delta^E_0(Y,c) = \delta^R_0(Y , m) = \delta(Y)$.}
\item[(3)]{$\delta^E_j(-Y,c) = \delta^R_j(-Y,m) = \overline{\mu}(Y)$ for all $j \ge 0$.}
\item[(4)]{$\delta^S_{j,0}(Y,m) = \delta^S_{j,1}(Y,m) = \delta^S_{0,j}(Y,c) = \delta^S_{j,1}(Y,c) = -\overline{\mu}(Y)$ for $j \ge 1$.}
\item[(5)]{$\delta^S_{j,0}(-Y,m) = \delta^S_{j,1}(-Y,m) = \overline{\mu}(Y)$ for all $j \ge 1$.}
\item[(6)]{$\delta^S_{i,j}(-Y,c) = \overline{\mu}(Y)$ for all $i,j$ with $i=0$ or $j \le 1$.}
\item[(7)]{$\delta^R_j(Y,c) \ge -\overline{\mu}(Y)$ and $\delta^R_j(-Y,c) \le \overline{\mu}(Y)$ for all $j \ge 0$.}
\end{itemize}

\end{theorem}
\begin{proof}
We will first prove (1) and (3)-(6). Consider the case that $a_i$ is even for some $i$. Corollary \ref{cor:even} implies that $Y = Y_\Gamma$ is the boundary of the plumbing $X_\Gamma$ on the graph $\Gamma$ given in Figure \ref{fig:star} where all the degrees are even. Applying Theorem \ref{thm:froy} to $X_\Gamma$ and $-X_\Gamma$, we deduce that 
\[
\delta^E_\infty(Y,c) = \delta^R_\infty(Y,m) = \delta^S_{0,\infty}(Y,c) = \delta^S_{\infty,0}(Y,m) = \delta^S_{\infty,1}(Y,m) = -\sigma(\Gamma)/8 = -\overline{\mu}(Y)
\]
and
\[
\delta^E_\infty(-Y,c) = \delta^R_\infty(-Y,m) = \delta^S_{0,\infty}(-Y,c) = \delta^S_{\infty,0}(-Y,m) = \delta^S_{\infty,1}(-Y,m) = \overline{\mu}(Y).
\]

Let $\Gamma'$ be the admissible subgraph obtained by removing the central vertex of $\Gamma$. Then $\Gamma'$ is a disjoint union of $n$ linear graphs and $Y_{\Gamma'}$ is a union of lens spaces. Thus $X_{\Gamma' , \Gamma}$ is a cobordism from a union of lens spaces to $Y$ and $b_+(X_{\Gamma' , \Gamma}) = 0$, $b_-(X_{\Gamma' , \Gamma}) = 1$. Applying Theorem \ref{thm:froy} to $X_{\Gamma' , \Gamma}$ and $-X_{\Gamma' , \Gamma}$ gives (1) and (3)-(6). To see this, consider for example the invariants $\delta^E_j(Y,c)$. Since $b_+(X_{\Gamma' , \Gamma}) = 0$, $b_-(X_{\Gamma' , \Gamma}) = 1$, Theorem \ref{thm:froy} implies that $\delta^E_j(Y,c)$ is independent of $j$ for all $j \ge 1$. Hence $\delta^E_j(Y,c) = \delta^E_\infty(Y,c) = -\overline{\mu}(Y)$ for $j \ge 1$. The other cases of (1) and (3)-(6) follow similarly.

Now suppose all the $a_i$ are odd. Since $\Sigma(a_1, \dots , a_n)$ $= \Sigma(a_1 , a_2 , \dots , a_n , 1)$ we can assume $n$ is odd by inserting a $1$ if necessary. Choose integers $b_1, \dots , b_n$ such that $\sum_{i=1}^{n} b_i/a_i = 1/(a_1 \cdots a_n)$ so that $Y = Y(0 ; (a_1 , b_1) , \dots , (a_n , b_n))$. Since all the $a_i$ are odd, this implies that $\sum_{i=1}^{n} b_i = 1 \; ({\rm mod} \; 2)$ and hence an odd number of $b_i$ are odd. If $b_{i_1}, b_{i_2}$ are both even, then we can replace them by $b_{i_1}+a_{i_1}, b_{i_2}-a_{i_2}$ to make them odd. Since $n$ is odd, by making such substitutions it is possible to choose the $b_i$ so that they are all odd. Furthermore for each $i>1$, we can replace $b_i$ by $b_i + 2k a_i$ and $b_1$ by $b_1 - 2ka_1$ for sufficiently large $k$ and hence we can assume $b_i > 0$ for $i>1$. Since $\sum_{i=1}^{n} b_i/a_i = 1/(a_1 \cdots a_n)$, it follows that $b_1 < 0$.

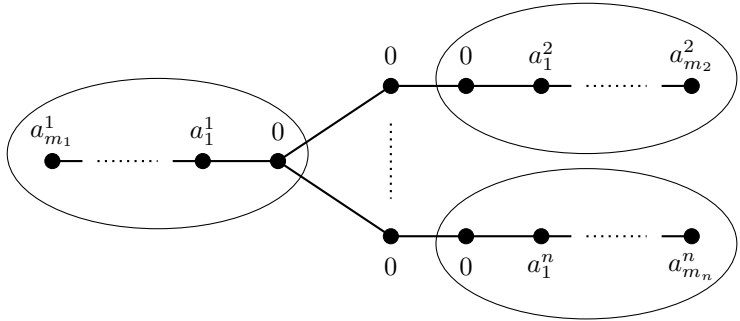
\begin{figure}[h]
\begin{tikzpicture}

\draw[-latex] (-2.1, 0.1) ellipse [x radius = 2cm, y radius = 1cm, start angle = 30, end angle = 30];
\draw[-latex] (3.6, 1.1) ellipse [x radius = 2cm, y radius = 1cm, start angle = 30, end angle = 30];
\draw[-latex] (3.6, -1.1) ellipse [x radius = 2cm, y radius = 1cm, start angle = 30, end angle = 30];  
\draw[thick] (-3.5,0) -- (-3.1,0) ;
\draw[thick] (-1.9,0) -- (-1.5,0) ;
\draw[thick] (-1.5,0) -- (-0.5,0) ;
\draw [dotted, thick] (-2.9,0) -- (-2.1,0);
\draw [fill] (-3.5,0) circle(0.1);
\draw [fill] (-1.5,0) circle(0.1);
\draw [fill] (-0.5,0) circle(0.1);
\draw[thick] (-0.5,0) -- (1,1) ;
\draw[thick] (1,1) -- (3.4,1) ;
\draw[thick] (4.6,1) -- (5,1) ;
\draw [dotted, thick] (3.6,1) -- (4.4,1);
\draw [fill] (1,1) circle(0.1);
\draw [fill] (2,1) circle(0.1);
\draw [fill] (3,1) circle(0.1);
\draw [fill] (5,1) circle(0.1);
\draw[thick] (-0.5,0) -- (1,-1) ;
\draw[thick] (1,-1) -- (3.4,-1) ;
\draw[thick] (4.6,-1) -- (5,-1) ;
\draw [dotted, thick] (3.6,-1) -- (4.4,-1);
\draw [fill] (1,-1) circle(0.1);
\draw [fill] (2,-1) circle(0.1);
\draw [fill] (3,-1) circle(0.1);
\draw [fill] (5,-1) circle(0.1);
\draw [dotted, thick] (1,0.5) -- (1,-0.5);
\node at (-3.5,0.4) {$a^1_{m_1}$};
\node at (-1.5,0.4) {$a_1^1$};
\node at (-0.5,0.4) {$0$};
\node at (1,1.4) {$0$};
\node at (2,1.4) {$0$};
\node at (3,1.4) {$a^2_1$};
\node at (5,1.4) {$a^2_{m_2}$};
\node at (1,-1.4) {$0$};
\node at (2,-1.4) {$0$};
\node at (3,-1.4) {$a^n_1$};
\node at (5,-1.4) {$a^n_{m_n}$};

\end{tikzpicture}
\caption{Plumbing graph in case where all $a_i$ are odd}\label{fig:gr1}
\end{figure}

Choose integers $a_{ij}$ with $a_i/b_i = [a_{i1} , \dots , a_{i m_i}]$. Then $Y = Y_\Gamma$ is the boundary of the plumbing $X_\Gamma$ on the graph shown in Figure \ref{fig:gr1}. Let $\Gamma'_i$ denote the subgraph of $\Gamma$ indicated by the $i$-th circle in Figure \ref{fig:gr1}. Let $\Gamma' = \cup_{i=1}^{n} \Gamma'_i$. By \cite[Proposition 7]{sav0}, we have that $X_{\Gamma' , \Gamma}$ is spin and has signature $\sum_{i=1}^{n} {\rm sign}(a_i/b_i) - 1$. Since $b_i > 0$ for $i>1$ and $b_1<0$, we get $\sigma(X_{\Gamma' , \Gamma}) = n-3$ and $b_+(X_{\Gamma' , \Gamma}) = 1$. Let $w \in H^2(X_\Gamma ; \mathbb{Z})$ denote the Wu class (the unique characteristic which when written in terms of the basis $\{e_j \}$ of $H^2(X_\Gamma ; \mathbb{Z})$ has the form $w = \sum_{i=1}^{n} u_i e_i$ where the $u_i$ are $0$ or $1$). The boundary of $X_{\Gamma'_i}$ is the lens space $L(b_i,a_i)$. Therefore $H^1( \partial X_{\Gamma'_i} ; \mathbb{Z}_2) = 0$ and we get an orthogonal decomposition 
\[
H^2(X_\Gamma ; \mathbb{Z}_2) =  \bigoplus_{i=1}^{n} H^2(X_{\Gamma'_i} ; \mathbb{Z}_2) \oplus H^2(X_{\Gamma' , \Gamma} ; \mathbb{Z}_2).
\]
Let $w_i$ denote the Wu class of $X_{\Gamma'_i}$ (which is unique since $\det( \Gamma'_i ) = \pm b_i$ is odd). Then since $X_{\Gamma' , \Gamma}$ is spin, it follows that $w = w_1 + \cdots + w_n \; ({\rm mod} \; 2)$. But the coefficients of $w$ and $w_i$ are all $0$ or $1$, hence $w = w_1 + \cdots + w_n$. Let $\mathfrak{s}$ denote the unique spin structure on $X_{\Gamma',\Gamma}$. By uniqueness $\mathfrak{s}$ is necessarily of type $S$ with respect to both $m$ and $c$. Now (1) and (3)-(6) will follow by applying Theorem \ref{thm:froy} to $(X_{\Gamma' , \Gamma} , \mathfrak{s})$. We will show this for (1). The cases (3)-(6) will follow by similar reasoning. By Theorem \ref{thm:froy}, we find 
\[
\delta^E_j(Y , c) = \delta^R_j(Y , m) = -\frac{\sigma(X_{\Gamma' , \Gamma})}{8} + \sum_{i=1}^{n} \delta( L(b_i , a_i ) , \mathfrak{s}_i )
\]
for $j \ge 1$, where $\mathfrak{s}_i$ denotes the unique spin structure on $L(b_i , a_i)$ (recall that $b_i$ is odd). But for lens spaces $L(p,q)$ with $p$ odd, one has $\delta( L(p,q) , \mathfrak{s}) = -\overline{\mu}(L(p,q))$ where $\mathfrak{s}$ denotes the spin structure \cite{sti}. Hence
\begin{align*}
\delta^E_j(Y , c) = \delta^R_j(Y , m) &= -\frac{\sigma(X_{\Gamma' , \Gamma})}{8} + \sum_{i=1}^{n} \delta( L(b_i , a_i ) , \mathfrak{s}_i ) \\
&=  -\frac{\sigma(X_{\Gamma' , \Gamma})}{8} - \sum_{i=1}^{n} \overline{\mu}(L(b_i,a_i)) \\
&=  -\frac{\sigma(X_{\Gamma' , \Gamma})}{8} - \frac{1}{8} \sum_{i=1}^{n} \left( \sigma(X_{\Gamma'_i}) - w_i^2 \right) \\
&= -\frac{1}{8} \left( \sigma(X_\Gamma) - w^2 \right) \\
&= -\overline{\mu}(Y)
\end{align*}
where in the second to last line we used that $w = w_1 + \cdots + w_n$.

Next we prove (7). For this recall that $Y = Y_\Gamma$ is the boundary of a negative definite plumbing $X_\Gamma$ (\cite[Example 1.17]{sav}). Any characteristic in $H^2(X_\Gamma ; \mathbb{Z})$ is necessarily anti-preserved by $c_\Gamma$ and so corresponds to a spin$^c$-structure of type $R$. In particular, there is a spin$^c$-structure $\mathfrak{s}$ of type $R$ for which $\overline{\mu}(Y) = (\sigma(\Gamma) - c(\mathfrak{s})^2)/8$ \cite[\textsection 7.2.3]{sav}. Applying Theorem \ref{thm:froy} gives $\delta^R_j(Y , c) \ge -\overline{\mu}(Y)$ for all $j \ge 0$. Thinking of $-Y$ as an ingoing boundary of $X_\Gamma$ Theorem \ref{thm:froy} also gives $\delta^R_j(-Y,c) \le \overline{\mu}(Y)$ for all $j \ge 0$.

It remains to prove (2). First note that $Y_\Gamma$ is the boundary of a negative definite plumbing whose plumbing graph has only one bad vertex in the terminology of \cite{os2}. Then it follows from \cite[Corollary 1.4]{os2} that $HF^+(-Y)$ is concentrated in even degrees. Consequently $d(Y)$ is even and $HF^+_{red}(Y)$ is concentrated in odd degrees. Furthermore, $HF^+(-Y)$ may be computed from the graded roots algorithm \cite{nem}. The algorithm implies that $HF^+_{red}(-Y)$ is concentrated in degrees at least $-d(Y)$. Dually it follows that $HF^+_{red}(Y)$ is concentrated in degrees at most $d(Y) - 1$. 

We have $\delta(Y) \ge -\overline{\mu}(Y)$. This follows by applying the Fr{\o}yshov inequality to any negative definite plumbing which bounds $Y$.

Now consider the spectral sequence $E_2^{p,q}$ for $HSW_E^*(Y,c)$ or $HSW_R^*(Y,m)$ given in Theorem \ref{thm:ss}. If any of the differentials in the spectral sequence are non-zero, then we must have $\delta^T_j(Y , \sigma) > \delta(Y)$ for all $j \ge 0$, where $(T,\sigma) = (E,c)$ or $(Y,m)$. But $\delta^T_1(Y,\sigma) = -\overline{\mu}(Y) \le \delta(Y)$. Hence the differentials must all be zero and it follows (since $HF^+_{red}(Y)$ is concentrated in odd degrees) that $\delta^T_0(Y , \sigma) = \delta(Y)$.
\end{proof}

For any plumbing graph $\Gamma$, let $|\Gamma|$ denote the underlying vertex set. Let $H(\Gamma)$ denote $\mathbb{Q}^{|\Gamma|}$ equipped with the bilinear form $\langle e_i , e_j \rangle = A(\Gamma)_{ij}$, where $e_1, \dots , e_{|\Gamma|}$ is the standard basis. If $\Gamma'$ is an admissible subgraph of $\Gamma$, let $H(\Gamma',\Gamma)$ denote the orthogonal complement of $H(\Gamma')$ in $H(\Gamma)$. 

\begin{definition}\label{def:jGamma}
Let $\Gamma$ be a plumbing graph. Define an invariant $j(\Gamma) \in \mathbb{Z}$ of $\Gamma$ by 
\[
j(\Gamma) = \min_{\Gamma'} \{ b_-( H(\Gamma',\Gamma)) \}
\]
where the minimum is taken over all admissible subgraphs $\Gamma' \subseteq \Gamma$ for which $\Gamma'$ is a disjoint union of linear graphs.
\end{definition}

Note that from the definition we clearly have $j(\Gamma) \le |\Gamma|$.

\begin{theorem}\label{thm:plumbdelta}
Let $\Gamma$ be a connected plumbing graph whose degrees are all even and let $Y_\Gamma$ be the boundary of the plumbing according to $\Gamma$. Let $\mathfrak{s}$ denote the restriction to $Y_\Gamma$ of the unique spin$^c$-structure on $X_\Gamma$ with $c(\mathfrak{s}) = 0$.
\begin{itemize}
\item[(1)]{$\delta_j^E( Y_\Gamma , \mathfrak{s} , c_\Gamma ) = \delta^S_{0,j}( Y_\Gamma , \mathfrak{s} , c_\Gamma) =  -\sigma(\Gamma)/8$ for all $j \ge j(\Gamma)$.}
\item[(2)]{Suppose that $\Gamma$ is a $\mathbb{Z}_2$-equivariant plumbing graph. Then $\delta_j^R(Y_\Gamma , \mathfrak{s} , m_\Gamma) = \delta^S_{j,k}( Y_\Gamma , \mathfrak{s} , m_\Gamma) = -\sigma(\Gamma)/8$ for all $j \ge j(\Gamma)$ and $k = 0,1$.}
\end{itemize}
\end{theorem}
\begin{proof}
We will prove (1). The proof for (2) is almost identical. Applying Theorem \ref{thm:froy} to $X_\Gamma$ gives $\delta^E_j(Y_\Gamma , \mathfrak{s} , c_\Gamma) = -\sigma(\Gamma)/8$ for all large enough $j$ (in fact, for $j \ge |\Gamma|)$. Hence $\delta^E_\infty(Y_\Gamma , \mathfrak{s} , c_\Gamma) = -\sigma(\Gamma)/8$.

Now choose an admissible subgraph $\Gamma' \subseteq \Gamma$ such that $\Gamma'$ is a disjoint union of linear graphs and $b_-( H(\Gamma',\Gamma)) = j(\Gamma)$. Applying Theorem \ref{thm:froy} to $X_{\Gamma' , \Gamma}$ and noting that $Y_{\Gamma'}$ is a union of $L$-spaces, we see that $\delta_j^E(Y_\Gamma , \mathfrak{s} , c_\Gamma) = -\sigma(H(\Gamma',\Gamma))/8 + \delta(Y_{\Gamma'} , \mathfrak{s}|_{Y_{\Gamma'}})$ for all $j \ge j(\Gamma)$. In particular, this shows that the value of $\delta^E_j(Y_\Gamma , \mathfrak{s}  , c_\Gamma)$ is independent of $j$ for $j \ge j(\Gamma)$. Therefore $\delta^E_j(Y_\Gamma , \mathfrak{s} , c_\Gamma ) = \delta^E_\infty(Y_\Gamma , \mathfrak{s} , c_\Gamma) = -\sigma(\Gamma)/8$ for all $j \ge j(\Gamma)$. A similar argument also shows that $\delta^S_{0,j}( Y_\Gamma , \mathfrak{s} , c_\Gamma) = -\sigma(\Gamma)$ for $j \ge j(\Gamma)$.
\end{proof}

\begin{remark}
If $\det(\Gamma)$ is odd, so that $Y_\Gamma$ is a $\mathbb{Z}_2$-homology $3$-sphere, then $\sigma(\Gamma)/8 = \overline{\mu}(Y)$ is the Neumann--Siebenmann invariant of $Y$. More generally, for any $\sigma(\Gamma)/8 = \overline{\mu}(Y , \mathfrak{s}|_Y)$ where $\mathfrak{s}$ is the unique spin$^c$-struture on $X_\Gamma$ with $c(\mathfrak{s}) = 0$ and $\overline{\mu}(Y , \mathfrak{s}|_Y)$ is the generalised Neumann--Seibenmann invariant, as in \cite[\textsection 4]{neu}.
\end{remark}

\begin{corollary}\label{cor:plumbl}
Let $Y$ be an integral homology $3$-sphere which is the boundary of the plumbing on a graph $\Gamma$ with all vertices having even degree. If $Y$ is an $L$-space, then $\delta(Y) = -\overline{\mu}(Y)$.
\end{corollary}
\begin{proof}
Since $Y = Y_\Gamma$, Theorem \ref{thm:plumbdelta} gives $\delta^E_\infty(Y , c_\Gamma) = -\overline{\mu}(Y)$. On the other hand, $Y$ is an $L$-space so $\delta^E_\infty(Y , c_\Gamma) = \delta(Y)$.
\end{proof}

\begin{remark}
A result similar to Corollary \ref{cor:plumbl} was proven in \cite{sti}, but with different assumptions on $Y$. Namely \cite{sti} requires that $Y$ is the boundary of a {\em negative definite} plumbing, but does not require the degrees of the plumbing graph to be even.
\end{remark}

\subsection{Branched double covers}\label{sec:bdc}

Let $K$ be a knot in $S^3$ and let $Y = \Sigma_2(K)$ be the double cover of $S^3$ branched over $K$. Then $Y$ is a rational homology sphere, in fact $| H_1(Y ; \mathbb{Z}) | = \det(K)$ \cite[Corollary 9.2]{lic}. Since $\det(K)$ is odd, there is a unique spin structure on $Y$. The corresponding spin$^c$-structure will be denoted $\mathfrak{s}_0$. Let $\sigma\colon Y \to Y$ be the covering involution of the branched double cover. By uniqueness, $\mathfrak{s}_0$ is preserved by $\sigma$ and is an odd spin involution. The delta-invariants of $(Y , \mathfrak{s}_0 , \sigma)$ define knot invariants of $K$ as follows.

\begin{definition}
Let $K$ be a knot in $S^3$. We define the delta-invariants $\delta^E_j(K), \delta^R_j(K)$ and $\delta^S_{i,j}(K)$ of $K$ to be the corresponding delta-invariants of $( \Sigma_2(K) , \mathfrak{s}_0 , \sigma )$.

In a similar fashion we also define invariants $\delta^E_\infty(K), j^E(K)$ etc. to be equal to the corresponding invariants of $(\Sigma_2(K) , \mathfrak{s}_0 , \sigma)$.
\end{definition}

Note that the invariant $\delta_j(K)$ defined in \cite{bh} is equal to $4 \delta^E_j(K)$.

For a knot $K$ in $S^3$, we let $\sigma(K)$ denote the signature and $g_4(K)$ the smooth slice genus.

\begin{theorem}\label{thm:deltaK}
Let $K$ be a knot in $S^3$.
\begin{itemize}
\item[(1)]{$\delta^E_j(K), \delta^R_j(K), \delta^S_{k,l}(K)$ depend only on the smooth concordance class of $K$.}
\item[(2)]{$\delta^E_j(K), \delta^R_j(K), \delta^S_{k,l}(K) \in \frac{1}{4}\mathbb{Z}$.}
\item[(3)]{$\delta^E_j(K) = \delta^R_j(K) = -\sigma(K)/8 \; ({\rm mod} \; \mathbb{Z})$. $\delta^S_{k,l}(K) = -\sigma(K)/8 \; ({\rm mod} \; 2\mathbb{Z})$.}
\item[(4)]{$\delta^E_j( K ) = \delta^S_{0,j}(K) = -\sigma(K)/8$ for $j \ge g_4(K) - \sigma(K)/2$.}
\item[(5)]{If $K$ is quasi-alternating, then $\delta^E_j(K) = \delta^R_j( K) = \delta^S_{k,l}( K ) = -\sigma(K)/8$ for all $j \ge 0$ and all $(k,l)$ with $k=0$ or $l \le 1$.}
\item[(6)]{If $g_4(K) = -\sigma(K)/2$, then $\delta^R_\infty(K) \ge -\sigma(K)/8$ and $\delta^R_0( -K ) \le \sigma(K)/8$.}
\item[(7)]{If $g_4(K) = 1 - \sigma(K)/2$, then $\delta^S_{i,j}(K) \ge -\sigma(K)/8$ for all $i,j$ with $i=0$ or $j=0$ and $\delta^S_{0,1}(-K) \le \sigma(K)/8$.} 
\end{itemize}

\end{theorem}
\begin{proof}
This is a straightforward extension of the results in \cite[\textsection 6]{bh} and so we omit the details of the proof.
\end{proof}

\begin{proposition}\label{prop:subadd}
For any knots $K_1,K_2$, we have $\delta^T_{i+j}(K_1 \# K_2) \le \delta^T_i(K_1) + \delta^T_j(K_2)$ for $T = E$ or $R$ and $\delta^S_{i+k,j+l}(K_1 \# K_2) \le \delta^S_{i,j}(K_1) + \delta^S_{k,l}(K_2)$.
\end{proposition}
\begin{proof}
This follows easily from the fact that $\Sigma_2(K_1 \# K_2)$ is diffeomorphic to the equivariant connected sum $\Sigma_2(K_1) \# \Sigma_2(K_2)$.
\end{proof}

Let $L = M( b ; (a_1,b_1) , \dots , (a_n,b_n))$ denote a Montesinos link. We use the same convention for Montesinos knots as \cite[\textsection 3.2]{owst}. $L$ is constructed by joining together $n$ rational tangles with slopes $a_1/b_1, \dots , a_n/b_n$ together with $b$ half-twists. If exactly one of the $a_i$ is even, then $L$ is a knot. The branched double cover $Y = \Sigma_2(L)$ is the Seifert fibre space $Y( b ; (a_1 , b_1) , \dots , (a_n , b_n))$ (this was shown in \cite{mont}, but see also \cite[Proposition 3.3]{owst} where the orientation is worked out carefully. Note that our definition of $Y(b ; (a_1 , b_1) , \dots , (a_n , b_n))$ corresponds to $Y(-b ; (a_1 , b_1) , \dots , (a_n , b_n))$ in \cite{owst}).

\begin{theorem}\label{thm:mont}
Let $K = M(b ; (a_1, b_1) , \dots , (a_n , b_n))$ be a Montesinos knot where exactly one $a_i$ is even. Let $e = b - \sum_{i=1}^n b_i/a_i$. If $e > 0$, then $\delta^E_j(K) = -\sigma(K)/8$ for all $j \ge 0$. If $e < 0$, then $\delta^E_j(K) = -\sigma(K)/8$ for all $j \ge 1$ and $\delta^E_0(K) = \delta( \Sigma_2(K) , \mathfrak{s}_0 )$ and for $j \ge 0$.
\end{theorem}
\begin{proof}
Recall that $\Sigma_2(K) = Y(b ; (a_1 , b_1) , \dots , (a_n , b_n))$ is the boundary of a plumbing $X_\Gamma$, where $\Gamma$ is a star-shaped graph and all the degrees are even. We claim that $\sigma$ is the restriction to $\partial X_\Gamma$ of the complex conjugation involution $c_\Gamma$ constructed in Section \ref{sec:cc}. This follows from the argument given in \cite[\textsection 7.2]{sav1}. Now the result follows from Theorem \ref{thm:plumbdelta} and the fact that $Y(b ; (a_1, b_1) , \dots , (a_n , b_n))$ is given by plumbing on a star-shaped graph $\Gamma$ which has $j(\Gamma) = 0$ if $e > 0$ and $j(\Gamma) = 1$ if $e < 0$. Note that $\sigma(K)/8 = \overline{\mu}(Y)$ \cite[\textsection 7.2.3]{sav}. 
\end{proof}

\subsection{Equivariant Dehn surgery}\label{sec:eds}

Let $L$ be a link in $S^3$ and suppose that $L$ is sent to itself under some orientation preserving, smooth involution $\sigma\colon S^3 \to S^3$. The resolution of the Smith conjecture \cite{bm} implies that the fixed point set $C$ of $\sigma$ is an unknot or is empty. Let $K$ be a component of $L$. If $\sigma(K) \neq K$, then $\sigma$ exchanges the two components $K$ and $\sigma(K)$. If $\sigma(K) = K$, then either $\sigma$ acts freely on $K$, in which case we say $K$ is {\em $2$-periodic} or $\sigma$ has exactly two fixed points on $K$, in which case we say $K$ is {\em strongly invertible}. We say that $L$ is {\em $2$-periodic} if every component of $L$ is $2$-periodic (with respect to $\sigma$) and we say that $L$ is {\em strongly invertible} if every component of $L$ is strongly invertible (with respect to $\sigma$).

Let $\mathcal{F}$ denote a framing of $L$ and let $Y$ be obtained from $S^3$ by performing Dehn surgery along $L$ with framing $\mathcal{F}$. Suppose that the framing is $\sigma$-invariant in the sense that for any component $K$ of $L$ which is not sent to itself by $\sigma$, the framings of $K$ and $\sigma(K)$ coincide. Then we can carry out Dehn surgery equivariantly with respect to $\sigma$ and the extension is unique up to conjugacy by diffeomorphisms isotopic to the identity (see \cite[\textsection 2]{sak} for details). We denote the induced involution on $Y$ by $\sigma$.

Note that $2$-periodic links can further be subdivided into two types according to whether or not $\sigma$ acts freely on $S^3$. However we will use the term $2$-periodic to refer to either of these possibilities. 

If the framing coefficients (i.e. the slopes) are all integers then there is a corresponding $4$-manifold $X$, the trace of the surgery on $L$, which is constructed by adding $2$-handles to the $4$-ball along the components of $L$, with the framing specifying the handle attachments, then smoothing out corners. The result is a compact, oriented, simply-connected smooth $4$-manifold $X$ with boundary $Y$. Futhermore, the involution $\sigma$ is easily seen to extend to a smooth, orientation preserving involution on $X$.

$H^2(X ; \mathbb{Z})$ has a natural basis corresponding to the $2$-handles. Hence for each component $K$ of $L$ there is a corresponding basis element $e_K \in H^2(X ; \mathbb{Z})$ (one also needs to choose an orientation on $L$ so that each $e_K$ is defined. Without choosing an orientation, $e_K$ is only defined up to sign). The intersection form on $H^2(X ; \mathbb{Z})$ is given by $\langle e_{K_1} , e_{K_2} \rangle = lk( K_1 , K_2)$, the linking number of $K_1, K_2$. In the case $K_1 = K_2$, $lk(K_1 , K_2)$ is the self-linking number, which coincides with the framing coefficient. Note that $H^2(X ; \mathbb{Z})$ is spin if and only if all the framing coefficients are even. The action of $\sigma$ on $H^2(X ; \mathbb{Z})$ is easy to describe: if $K$ is a component of $L$ then $\sigma( e_K ) = \pm e_{\sigma(K)}$, where the sign is $+$ or $-$ depending on whether $\sigma\colon K \to \sigma(K)$ is orientation preserving or reversing. In particular, if $L$ is $2$-periodic, then $\sigma$ acts as the identity on $H^2(X ; \mathbb{Z})$ and if $L$ is strongly invertible, then $\sigma$ acts as $-1$ on $H^2(X ; \mathbb{Z})$.

If $\mathcal{F}$ is integral and $L = K_1 \cup \cdots \cup K_n$ is the decomposition of $L$ into its components, then we define the linking matrix $A = [A_{ij}]$ of $(L , \mathcal{F})$ by $A_{ij} = lk(F_i , F_j)$. Then $Y$ is a rational homology sphere if and only if $\det(A) \neq 0$. Moreover $| H_1( Y ; \mathbb{Z}) | = |\det(A)|$. The linking matrix defines a symmetric bilinear form, which as explained above gives the intersection form on $H^2(X ; \mathbb{Z})$. Let $\sigma(A)$ denote the signature of this intersection form, which is also the signature of $X$. Similarly, we define $b_{\pm}(A)$ to be $b_{\pm}(X)$.

We will say that the framing $\mathcal{F}$ is {\em even} if all the framing coefficients are even integers. In this case the trace $X$ is spin and there is a unique spin structure $\mathfrak{s}_0$ on $X$ for which $c(\mathfrak{s}_0) = 0$. By restriction, $\mathfrak{s}_0$ defines a distinguished spin structure on $Y$. We have that $\sigma$ is an odd spin involution. This follows because the fixed point set of $\sigma$ on $Y$ is non-empty (recall that the original involution on $S^3$ had fixed point set an unknot in $S^3$).

\begin{theorem}\label{thm:surglink}
Let $(L , \sigma)$ be a $2$-periodic or strongly invertible link. Let $Y$ be the $3$-manifold obtained by Dehn surgery on $Y$ with respect to some framing $\mathcal{F}$ and denote by $\sigma$ the induced involution on $Y$. Suppose that $\mathcal{F}$ is even and let $\mathfrak{s}_0$ denote the distinguished spin structure. Let $A$ denote the linking matrix of $(L , \mathcal{F})$. Then
\begin{itemize}
\item[(1)]{If $L$ is $2$-periodic, then $\delta^R_j(Y , \mathfrak{s}_0 , \sigma) = -\sigma(A)/8$ for $j \ge b_-(A)$.}
\item[(2)]{If $L$ is strongly invertible, then $\delta^E_j(Y , \mathfrak{s}_0 , \sigma) = -\sigma(A)/8$ for $j \ge b_-(A)$.}
\end{itemize}

\end{theorem}
\begin{proof}
Let $X$ be the trace of $(L , \mathcal{F})$. Then $\sigma$ extends over $X$. Since $\mathcal{F}$ is even, $X$ is spin and $\mathfrak{s}_0$ extends to a unique spin structure $\mathfrak{s}$ on $X$ with $c(\mathfrak{s}) = 0$. Applying Theorem \ref{thm:froy} to $X$ gives the results.
\end{proof}

In the case that $L = K$ is a knot, we can improve upon Theorem \ref{thm:surglink}. In this case the framing $\mathcal{F}$ is a single rational number $p/q$ and $Y = S_{p/q}(K)$ is the Dehn surgery along $K$ with slope $p/q$. Note that it is sufficient to assume $p/q > 0$ since $S_{-p/q}(K) = -S_{p/q}(-K)$.

\begin{proposition}
Let $K$ be a $2$-periodic or strongly invertible knot. For $p > 0$, let $Y = S_{2p}(K)$ with the induced involution $\sigma$ and distinguished spin structure $\mathfrak{s}_0$. If $K$ is $2$-periodic, then:
\begin{itemize}
\item[(1)]{$\delta^E_0(Y , \mathfrak{s}_0 , \sigma) \le -1/8$, $\delta^E_\infty(-Y , \mathfrak{s}_0 , \sigma) \ge 1/8$.}
\item[(2)]{$\delta^R_j(Y , \mathfrak{s}_0 , \sigma) = -1/8$, $\delta^R_j(-Y , \mathfrak{s}_0 , \sigma) = 1/8$ for all $j \ge 0$.}
\item[(3)]{$\delta^S_{j,0}(Y , \mathfrak{s}_0 , \sigma) \ge -1/8$ for all $j \ge 0$ and $\delta^S_{1,0}(-Y , \mathfrak{s}_0 , \sigma) \le 1/8$.}
\end{itemize}

If $K$ is strongly invertible, then:
\begin{itemize}
\item[(1)]{$\delta^E_j(Y , \mathfrak{s}_0 , \sigma) = -1/8$, $\delta^E_j(-Y , \mathfrak{s}_0 , \sigma) = 1/8$ for all $j \ge 0$.}
\item[(2)]{$\delta^R_0(Y , \mathfrak{s}_0 , \sigma) \le -1/8$, $\delta^R_\infty(-Y , \mathfrak{s}_0 , \sigma) \ge 1/8$.}
\item[(3)]{$\delta^S_{i,j}(Y , \mathfrak{s}_0 , \sigma) \ge -1/8$ for all $i,j$ with $i=0$ or $j=0$ and $\delta^S_{0,1}(-Y , \mathfrak{s}_0 , \sigma) \le 1/8$.}
\end{itemize}

\end{proposition}
\begin{proof}
Let $X$ be the trace of the $2p$-surgery on $K$ and let $\mathfrak{s}$ be the unique spin$^c$-structure on $X$ with $c(\mathfrak{s}) = 0$. Applying Theorem \ref{thm:froy} to $(X , \mathfrak{s})$ gives the result.
\end{proof}

We can also consider the delta-invariants for spin$^c$-structures which are not spin. In this case we can consider $Y = S_p(K)$ where $p>0$ need not be even. The spin$^c$-structures on $Y$ form a torsor over $\mathbb{Z}_p$. The torsor can be identified with $\mathbb{Z}_p$ as follows. Let $X$ be the traces of the Dehn surgery on $K$. For any $c \in \mathbb{Z}$ such that $c = p \; ({\rm mod} \; 2)$, there is a unique spin$^c$-structure $\mathfrak{s}$ on $X$ such that $c(\mathfrak{s}) = c$ (where we identify $H^2(X ; \mathbb{Z})$ with $\mathbb{Z}$). By restriction, $\mathfrak{s}$ determines a spin$^c$-structure $\mathfrak{s}|_{Y}$ on $Y$. To $\mathfrak{s}|_{Y}$, we associate the unique $i \in \mathbb{Z}_p$ such that $c = p + 2i \; ({\rm mod} \; 2p)$. Since $\sigma$ sends $c$ to $c$ in the $2$-periodic case and to $-c$ in the strongly invertible case, we see that $\sigma$ acts trivially on the spin$^c$-structures on $Y$ in the $2$-periodic case and acts by $i \mapsto -i$ in the strongly invertible case. Denote the spin$^c$-structure on $Y$ corresponding to $i \in \mathbb{Z}_p$ by $\mathfrak{s}_i$. Then each $\mathfrak{s}_i$ is type $E$ in the $2$-periodic case and is type $R$ in the stronly invertible case.

\begin{proposition}
Let $K$ be a $2$-periodic or strongly invertible knot and $p$ a positive integer. Let $Y = S_p(K)$ and let $\sigma$ be the induced involution on $Y$. Let $i$ be an integer with $|i| \le p/2$.
\begin{itemize}
\item[(1)]{If $K$ is $2$-periodic, then
\[
\delta^E_0( Y , \mathfrak{s}_i , \sigma) \le \frac{ (p-2|i|)^2}{8p} - \frac{1}{8}, \quad \delta^E_\infty( -Y , \mathfrak{s}_i , \sigma) \ge -\frac{(p-2|i|)^2}{8p} + \frac{1}{8}.
\]
}
\item[(2)]{If $K$ is strongly invertible, then
\[
\delta^R_0( Y , \mathfrak{s}_i , \sigma) \le \frac{ (p-2|i|)^2}{8p} - \frac{1}{8}, \quad \delta^R_\infty( -Y , \mathfrak{s}_i , \sigma) \ge -\frac{(p-2|i|)^2}{8p} + \frac{1}{8}.
\]
}
\end{itemize}

\end{proposition}

\begin{proof}
Let $X$ be the trace of the $p$-surgery on $K$. For any $i$ with $|i| \le p/2$ we choose the unique spin$^c$-structure $\mathfrak{s}$ on $X$ with $c = c(\mathfrak{s})$ given by $c = 2i-p$ if $i \ge 0$, or $c = 2i+p$ if $i < 0$. Then
\[
\delta(X , \mathfrak{s}) = \frac{ (p-2|i|)^2}{8p} - \frac{1}{8}.
\]
The result now follows by applying Theorem \ref{thm:froy}.
\end{proof}

Suppose that $L$ is a strongly invertible link, and suppose that $\mathcal{F}$ is a framing where the coefficients are not all integral. Let $Y(L,\mathcal{F})$ be the $3$-manifold obtained by Dehn surgery on $(L , \mathcal{F})$. In this case we do not immediately get a $4$-manifold bounding $Y(L,\mathcal{F})$. However, we can obtain such a manifold by performing the slam dunk operation to $(L,\mathcal{F})$ \cite[\textsection 5.3]{gs}. Suppose that $(L' , \mathcal{F}')$ is obtained from $(L,\mathcal{F})$ by performing a slam dunk on a component $K$ of $L$. Hence $Y(L , \mathcal{F})$ and $Y(L' , \mathcal{F}')$ are diffeomorphic. In fact, by performing the slam dunk around a fixed point of the strong inversion as in Figure \ref{fig:slam}, one can ensure that $Y(L , \mathcal{F})$ and $Y(L' , \mathcal{F}')$ are {\em equivariantly} diffeomorphic. By repeatedly performing equivariant slam dunks, we can replace $(L , \mathcal{F})$ by a pair $(L' , \mathcal{F}')$ where $\mathcal{F}'$ is integral. Hence we can obtain a $4$-manifold $X$ bounding $Y(L , \mathcal{F})$ and such that $\sigma$ extends over $X$.

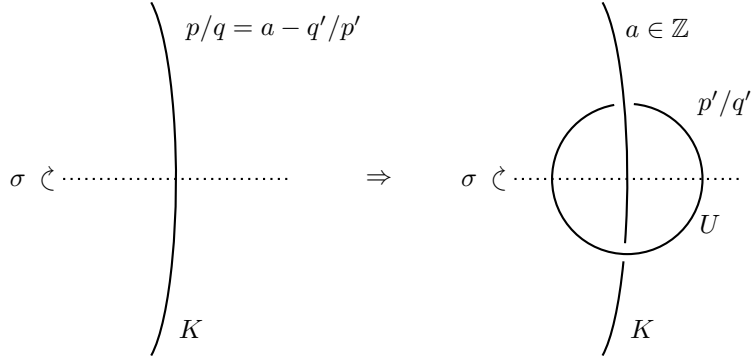
\begin{figure}
\begin{center}
\begin{tikzpicture}
\draw[thick] ($(0, 0) + (-70:0.5cm and 2.5 cm)$(P) arc (-70:70:0.5cm and 2.5cm);
\draw [dotted, thick] (-1,0) -- (2,0);
\node at (1.8,2) {$p/q = a - q'/p'$};
\node at (0.7,-2) {$K$};
\node at (6.7,-2) {$K$};
\node at (7.6,-0.6) {$U$};
\node at (-1.6,-0.02) {$\sigma$}; \node at (-1.2,0) {\rotatebox{90}{$\curvearrowright$}};
\node at (3.2,0) {$\Rightarrow$};
\node at (6-1.6,-0.02) {$\sigma$}; \node at (6-1.2,0) {\rotatebox{90}{$\curvearrowright$}};
\draw[thick] ($(6, 0) + (-70:0.5cm and 2.5 cm)$(P) arc (-70:-26:0.5cm and 2.5cm);
\draw[thick] ($(6, 0) + (-20:0.5cm and 2.5 cm)$(P) arc (-20:70:0.5cm and 2.5cm);
\draw [dotted, thick] (6-1,0) -- (6+2,0);
\node at (6+0.3+0.6,2) {$a \in \mathbb{Z}$};
\node at (7.8,1) {$p'/q'$};
\draw[thick] (7.5,0) arc (0:-260:1) ;
\draw[thick] (7.5,0) arc (0:85:1) ;
\end{tikzpicture}
\caption{Slam dunk move performed equivariantly with respect to a strong inversion.}\label{fig:slam}
\end{center}
\end{figure}

For example, let $K$ be any strongly invertible knot and let $Y = S_{1/2p}(K)$ be the Dehn surgery along $K$ with slope $1/2p$, where $p$ is any non-zero integer. Then $Y = Y(L , \mathcal{F})$ where $L = K \cup U$, where $U$ is an unknot with $lk(K,U) = 1$ and $\mathcal{F}$ is the framing which is $0$ on $K$ and is $-2n$ on $U$. The trace $X = X(L , \mathcal{F})$ of $(L , \mathcal{F})$ is spin, the intersection form on $H^2(X ; \mathbb{Z})$ is isomorphic to the hyperbolic lattice and $\sigma$ extends to $X$, acting as $-1$ on $H^2(X ; \mathbb{Z})$. The unique spin structure on $Y$ extends to a spin structure $\mathfrak{s}$ on $X$ with $c(\mathfrak{s}) = 0$. Applying Theorem \ref{thm:froy}, we find

\begin{proposition}\label{prop:12p}
Let $K$ be a strongly invertible knot and $p$ a non-zero integer. Then $\delta^E_j( \pm S_{1/2p}(K) , \mathfrak{s} , \sigma) = 0$ for all $j \ge 1$.
\end{proposition}

More generally, let $K$ be a strongly invertible knot and consider $Y = S_{p/q}(K)$ where $p$ is odd and $q$ is even and non-zero. Then by performing the slam dunk move, we can write $Y = Y(L , \mathcal{F})$, where $L = K \cup U_1 \cup U_2 \cup \cdots \cup U_n$, $U_1, \dots , U_n$ are unknots, $lk(K,U_1) = lk(U_1 , U_2 ) = \cdots = lk(U_{n-1} , U_n) = 1$, $lk(K , U_j) = 0$ for $j>1$, $lk(U_i , U_j) = 0$ for $|i-j| > 1$ and $lk(K,K) = a_0$, $lk(U_i , U_i) = a_i$ are such that all $a_i$ are even and $p/q = [a_0 , a_1, \cdots , a_n]$. Let $A$ denote the linking matrix. Note that since $p$ is odd, $Y$ has a unique spin$^c$-structure $\mathfrak{s}$ that comes from a spin structure. Then applying \ref{thm:froy} to the trace of $(L , \mathcal{F})$, we find

\begin{theorem}\label{thm:slamd}
Let $K$ be a strongly invertible knot and let $Y = S_{p/q}(K)$ where $p$ is odd and $q$ is even and non-zero. Let $p/q = [a_0 , a_1, \dots , a_n]$ where $a_0, \dots , a_n$ are even integers and let $A_{ij}$ be the matrix $A_{ii} = a_i$, $A_{ij} = 1$ for $|i-j| = 1$, $A_{ij} = 0$ for $|i - j|$. Then $\delta^E_j( S_{p/q}(K) , \mathfrak{s} , \sigma ) = -\sigma(A)/8$ for $j \ge b_-(A)$.
\end{theorem}

Note that in the case that $K$ is the unknot $Y = S_{p/q}(K) = -L(p,q)$ is a lens space. Since lens spaces are $L$-spaces, we get $-\sigma(A)/8 = \delta^E_j( S_{p/q}(K) , \mathfrak{s} , \sigma) = -\delta( L(p,q) , \mathfrak{s} )$. Hence $\sigma(A) = 8 \delta( L(p,q) , \mathfrak{s})$ and so feeding this in to Theorem \ref{thm:slamd}, we get $\delta^E_j( S_{p/q}(K) , \mathfrak{s} , \sigma) = -\delta( L(p,q) , \mathfrak{s})$ for any strongly invertible knot $K$ and where $j \ge b_-(A)$. Note also that $b_-(A) = (n+1 - \sigma(A))/2$, where $n+1$ is the number terms in the negative continued fraction $p/q = [a_0 , \dots , a_n]$, $a_i \in 2\mathbb{Z}$.

Recall that the knot Floer complex of a knot $K$ can be used to define a sequence of knot invariants $V_i(K)$, $i \in \mathbb{Z}$ \cite{niwu}. The sequence $V_i(K)$ is decreasing and is eventually zero. Hom--Wu define $\nu^+(K)$ to be the smallest $i$ such that $V_i(K) = 0$ \cite{howu}. It is shown in \cite{howu} that $\nu^+(K) \ge 0$ and equals zero if and only if $V_0(K) = 0$.

\begin{theorem}
Let $K$ be a strongly invertible knot and let $Y = S_{p/q}(K)$ where $p$ is odd, $q$ is even and $q > p > 0$. Let $\mathfrak{s}$ denote the unique spin$^c$-structure on $Y$ which comes from a spin structure. Then $\delta^E_\infty( -Y , \mathfrak{s} , \sigma ) = \delta( -Y , \mathfrak{s}) - V_0(K)$. Moreover, if $\nu^+(K) > 0$, then $j^E(-Y , \mathfrak{s} , \sigma) > 0$.
\end{theorem}
\begin{proof}
By \cite[Proposition 1.6]{niwu} we have 
\[
\delta( S_{p/q}(K) , \mathfrak{s}) = \delta( S_{p/q}(U) , \mathfrak{s}) - \max \{ V_{\lfloor \frac{i}{q} \rfloor}(K) , V_{\lceil \frac{p-i}{q} \rceil}(K) \}
\]
for some $i$, $0 \le i \le p-1$, where we use $\mathfrak{s}$ to denote the unique spin$^c$-structure coming from a spin structure for both $Y$ and $S_{p/q}(U)$. Since we have assumed $q > p$, this simplifies to $\delta( S_{p/q}(K) , \mathfrak{s}) = \delta( S_{p/q}(U) , \mathfrak{s}) - V_0(K)$. Then from Theorem \ref{thm:slamd}, we have 
\[
\delta^E_\infty(-S_{p/q}(K) , \mathfrak{s} , \sigma) = \delta^E_\infty( S_{-p/q}(-K) , \mathfrak{s} , \sigma ) = \delta( S_{-p/q}(U) , \mathfrak{s}) = -\delta( S_{p/q}(U) , \mathfrak{s}).
\]
Hence $\delta^E_\infty( -S_{p/q}(K) , \mathfrak{s} , \sigma ) = \delta( -S_{p/q}(K) , \mathfrak{s}) - V_0(K)$. Now if $\nu^+(K) > 0$, then $V_0(K) > 0$ and hence 
\[
\delta^E_0(-S_{p/q}(K) , \mathfrak{s} , \sigma) \ge \delta( -S_{p/q}(K) , \mathfrak{s}) > \delta^E_\infty( -S_{p/q}(K) , \mathfrak{s} , \sigma ),
\]
which means that $j^E( -S_{p/q}(K) , \mathfrak{s} , \sigma) > 0$.
\end{proof}

\section{Applications}\label{sec:app}

\subsection{Obstructions to extending involutions}\label{sec:obex}

Suppose we are given a rational homology $3$-sphere $Y$ with orientation preserving involution $\sigma$ and a spin$^c$-structure $\mathfrak{s}$ of type $E,R$ or $S$. Suppose that $X$ is a compact, oriented smooth $4$-manifold which bounds $Y$ and that $\mathfrak{s}$ extends to a spin$^c$-structure on $X$. We can use Theorem \ref{thm:froy} to obstruct the existence of an extension of $\sigma$ to an involution on $X$, under some assumptions on how $\sigma$ acts on $H^2(X ; \mathbb{Z})$. In a similar manner, given two triples $(Y_0 , \mathfrak{s}_0 , \sigma_0), (Y_1 , \mathfrak{s}_1 , \sigma_1)$ of rational homology spheres with involutions and spin$^c$ structures of type $T \in \{E,R,S\}$ we can obstruct the existence of certain equivariant cobordisms from $Y_0$ to $Y_1$.

To keep things relatively simple, we will mainly focus on the case that each component of the boundary of $X$ is an integral homology sphere. Consider first the case that $X$ is negative definite.

\begin{proposition}\label{prop:def}
Let $X$ be a compact, oriented, smooth $4$-manifold with boundary $Y$ empty or a union of integral homology spheres. Assume $H_1(X ; \mathbb{Z}_2) = 0$ and that $X$ is negative definite. Let $\sigma$ be an orientation preserving involution on $X$ that sends each component of $Y$ to itself. Then
\begin{itemize}
\item[(1)]{If $c \in H^2(X ; \mathbb{Z})$ is a characteristic element and $\sigma(c) = c$, then 
\[
\frac{ c^2 + b_2(X) }{8} \le \min\{ \delta_\infty^E(Y , \sigma) , -\delta^E_{0}(-Y,\sigma)\}.
\]
}
\item[(2)]{Assume that the fixed point set of $\sigma$ contains non-isolated points. If $c \in H^2(X ; \mathbb{Z})$ is a characteristic element and $\sigma(c) = -c$, then
\[
\frac{c^2 + b_2(X) }{8} \le \min\{ \delta^R_\infty(Y , \sigma) , -\delta^R_{0}(-Y,\sigma) \}.
\]
}
\item[(3)]{If $X$ is spin and $\sigma$ is odd, then
\[
\frac{b_2(X)}{8} \le \min\{ \delta^S_{0,\infty}(Y , \sigma) , \delta^S_{\infty,1}(Y,\sigma) , -\delta^S_{0,0}(-Y,\sigma) \}.
\]
}
\end{itemize}

\end{proposition}
\begin{proof}
Note that since $H_1(X ; \mathbb{Z}_2) = 0$, spin$^c$-structures on $X$ are in bijection with characteristic elements. A characteristic $c$ corresponds to a spin$^c$-structure of type $E$ if $\sigma(c) = c$ and to a spin$^c$-structure of type $R$ if $\sigma(c)=-c$. Now the result follows from applying Theorem \ref{thm:froy} twice, where we consider $X$ as having outgoing boundary $Y$ or ingoing boundary $-Y$. Note that in case (2) we assume $\sigma$ has non-isolated fixed points to ensure that the spin$^c$-structure corresponding to $c$ has type $R$.
\end{proof}

\begin{remark}
If $Y = \partial X$ has multiple components, then we can get an extension of Proposition \ref{prop:def} as follows. Partition the components of $Y$ into two so that $Y = Y_0 \cup Y_1$. Treat $Y_1$ as an outgoing boundary and $-Y_0$ as an ingoing boundary. Then in Proposition \ref{prop:def} (1), we get a bound $(c^2+b_2(X))/8 \le \delta^E_j(Y_1,\sigma)-\delta^E_j(-Y_0,\sigma)$ for any $j\ge 0$. Similar bounds apply in cases (2) and (3).
\end{remark}

\begin{example}
Consider $Y = \Sigma(2,5,11) \# -\! \! 2\Sigma(2,3,13)$ with involution $\sigma$ the connected sum of the $m$ involution on each summand. Since $\Sigma(2,5,11) = S_{-1}(T_{2,5})$, we have that $\Sigma(2,3,13) = \partial X_0$, where $X_0$ is the trace of the $-1$-surgery on $T_{2,5}$. We also have that $\Sigma(2,3,13)$ bounds a contractible smooth $4$-manifold $W$. Let $X$ be the boundary sum of $X_0$ and two copies of $-W$. Then $X$ is a negative definite smooth $4$-manifold bounding $Y$ and $b_2(X) = 1$. We claim that $\sigma$ does not extend to a smooth involution on $X$. Suppose on the contrary that $\sigma$ extends to $X$. Let $\mathfrak{s}$ be the unique spin$^c$-structure on $X$ with $c(\mathfrak{s})^2 = -1$. Since $H^2(X ; \mathbb{Z}) \cong \mathbb{Z}$, we have that either $\sigma(\mathfrak{s}) = \mathfrak{s}$ or $\sigma(\mathfrak{s}) = -\mathfrak{s}$. In the first case Proposition \ref{prop:def} (1) gives $\delta^E_\infty(Y,\sigma) \ge 0$. But 
\begin{align*}
\delta^E_\infty(Y,\sigma) &\le \delta^E_\infty( \Sigma(2,5,11) , m ) + 2\delta^E_\infty(-\Sigma(2,3,13),m) \\
&= -\lambda(\Sigma(2,5,11)) + 2\lambda(\Sigma(2,3,13)) = 3-4 = -1,
\end{align*}
a contradiction. In the second case Proposition \ref{prop:def} (2) gives $\delta^R_\infty(Y,\sigma) \ge 0$. But
\begin{align*}
\delta^R_\infty(Y,\sigma) &\le \delta^R_\infty(\Sigma(2,5,11) , m) + 2\delta^R_\infty(-\Sigma(2,3,13),m)\\
& = -\overline{\mu}(\Sigma(2,5,11)) + 2\overline{\mu}(2,3,13) = -1,
\end{align*}
a contradiction.
\end{example}

\begin{proposition}\label{prop:b+1}
Let $X$ be a compact, oriented, smooth spin $4$-manifold with boundary $Y$ empty or a union of integral homology spheres. Assume $H_1(X ; \mathbb{Z}_2) = 0$ and that $b_+(X)=1$. Let $\sigma$ be an odd involution on $X$ that sends each component of $Y$ to itself. Let $H^+(X)$ denote a $\sigma$-invariant maximal positive definite subspace of $H^2(X ; \mathbb{R})$. Then:
\begin{itemize}
\item[(1)]{If $\sigma$ acts trivially on $H^+(X)$, then
\[
-\frac{\sigma(X)}{8} \le \min\{ \delta^S_{\infty,0}(Y,\sigma) , -\delta^S_{1,0}(-Y,\sigma) \}.
\]
}
\item[(2)]{If $\sigma$ acts non-trivially on $H^+(X)$, then
\[
-\frac{\sigma(X)}{8} \le \min\{ \delta^S_{0,\infty}(Y,\sigma) , \delta^S_{\infty,1}(Y,\sigma) , -\delta^S_{0,1}(-Y,\sigma) \}.
\]
}
\end{itemize}
\end{proposition}
\begin{proof}
The condition $H_1(X ; \mathbb{Z}_2) = 0$ ensures that $X$ has a unique spin structure, which is then necessarily preserved by $S$. Now the result follows from Theorem \ref{thm:froy}.
\end{proof}

\begin{example}
Consider $Y = -\Sigma(2,3,12n-5)$, $n \ge 1$, with involution $\sigma = m$. Then $Y = \Sigma_2( -T_{3,12n-5})$ and $m$ coincides with the covering involution. Hence $\delta^S_{0,\infty}(Y,m) = \sigma(T_{3,12n-5})/8 = -(2n-1)$. Now suppose that $X$ is a compact, oriented, smooth spin $4$-manifold bounding $Y$ and that $H_1(X ; \mathbb{Z}_2)=0$ and $b_+(X) = 1$. The Fr{\o}yshov inequality for spin cobordisms implies that $\sigma(X) = -8$. Such $4$-manifolds exist, for instance we could take $X$ to be the plumbing on the graph in Figure \ref{fig:2312n-5}.

\begin{figure}
\begin{center}
\begin{tikzpicture}
\draw[thick] (-2,0) -- (6,0) ;
\draw[thick] (0,0) -- (0,-1);
\draw (-2,0)[fill] circle(0.1);
\draw (-1,0)[fill] circle(0.1);
\draw (0,0)[fill] circle(0.1);
\draw (1,0)[fill] circle(0.1);
\draw (2,0)[fill] circle(0.1);
\draw (0,-1)[fill] circle(0.1);
\draw (3,0)[fill] circle(0.1);
\draw (4,0)[fill] circle(0.1);
\draw (5,0)[fill] circle(0.1);
\draw (6,0)[fill] circle(0.1);

\node at (-2,0.5) {$-2$};
\node at (-1,0.5) {$-2$};
\node at (0,0.5) {$-2$};
\node at (1,0.5) {$-2$};
\node at (2,0.5) {$-2$};
\node at (0,-1.5) {$-2$};
\node at (3,0.5) {$-2$};
\node at (4,0.5) {$-2$};
\node at (5,0.5) {$-2$};
\node at (6,0.5) {$-2n$};

\end{tikzpicture}
\caption{Plumbing graph for $-\Sigma(2,3,12n-5)$}\label{fig:2312n-5}
\end{center}
\end{figure}
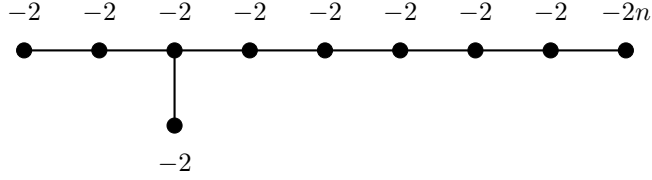

If $m$ extends to an involution $\sigma$ on $X$ and $\sigma|_{H^+(X)} = -1$, then Proposition \ref{prop:b+1} (2) gives $-\sigma(X)/8 \le \delta^S_{0,\infty}(Y,m)$. But $\sigma(X) = -8$ and $\delta^S_{0,\infty}(Y,m) = 1-2n$, which is impossible since $n \ge 1$. We conclude that any extension of $m$ to an involution on $X$ satisfies $\sigma|_{H^-(X)} = 1$. Moreover an example where $m$ extends to an involution is given by the plumbing on the graph in Figure \ref{fig:2312n-5}, where the extension is given by $m_\Gamma$.
\end{example}

Next, we consider the extension problem in the case that the involution acts homologically trivially.

\begin{proposition}\label{prop:ht}
Let $X$ be a compact, oriented, smooth spin $4$-manifold with boundary $Y$ empty or a union of integral homology spheres. Assume $H_1(X ; \mathbb{Z}_2) = 0$. Let $\sigma$ be a smooth odd involution on $X$ which acts homologically trivially on $X$ and sends each component of $Y$ to itself. Then:
\[
\sigma(X) = -8 \delta^R_\infty(Y,\sigma).
\]
Furthermore, we have
\[
b_-(X) \ge j^R(Y,\sigma), \quad b_+(X) \ge j^R(-Y,\sigma).
\]
\end{proposition}
\begin{proof}
Since $H_1(X ; \mathbb{Z}_2) = 0$, $X$ has a unique spin structure, which is necessarily of type $R$. The result follows by applying Theorem \ref{thm:froy} to $X$ and $-X$.
\end{proof}

\begin{example}
Let $Y$ be an integral homology sphere and $\sigma$ a smooth orientation preserving involution on $Y$. It is very easy to construct examples of spin $4$-manifolds $X$ bounding $Y$ and such that $\sigma$ does not extend to a homologically trivial involution on $X$. Suppose $X$ is a spin manifold bounding $Y$ and $H_1(X ; \mathbb{Z}_2) = 0$. If $\sigma(X) \neq -8 \delta^R_\infty(Y,\sigma)$, then $\sigma$ does not extend homologically trivially to $X$ by Proposition \ref{prop:ht}. If $\sigma(X) = -8\delta^R_\infty(Y,\sigma)$, then we can just replace $X$ by $X \# X'$, where $X'$ is any  closed spin $4$-manifold with $\sigma(X') \neq 0$ and $H_1(X' ; \mathbb{Z}_2) = 0$.

Suppose for example that $Y = \Sigma(a_1 , \dots , a_n)$ is a Brieskorn sphere where $a_i$ is even for some $i$ and that $\sigma = m$. Then $\delta^R_\infty(Y , m) = -\overline{\mu}(Y)$. Let $X$ be a spin manifold bounding $Y$ with $H_1(X ; \mathbb{Z}_2)$ and with $\sigma(X) \neq 8\overline{\mu}(Y)$. Then $m$ does not extend to a homologically trivial smooth involution on $X$. On the other hand, $m$ does extend to a smooth homologically trivial diffeomorphism on $X$. This is because the involution $m$ belongs to a circle action on $Y$, hence is smoothly isotopic to the identity. We can choose an extension of $m$ to $X$ which consists of an isotopy from $m$ to the identity in some collar neighbourhood of $Y$ and is the identity outside of this neighbourhood.

This non-extension result contrasts with the fact that $m$ does extend to a homologically trivial involution on the plumbing $X_{\Gamma}$ on a star-shaped graph $\Gamma$ whose boundary is $Y$.
\end{example}

Lastly, we consider the extension problem under the assumption that the involution acts as $-1$ on $H^2(X;\mathbb{Z})$. The proof is essentially the same as for Proposition \ref{prop:ht}.

\begin{proposition}\label{prop:hm}
Let $X$ be a compact, oriented, smooth spin $4$-manifold with boundary $Y$ empty or a union of integral homology spheres. Assume $H_1(X ; \mathbb{Z}_2) = 0$. Let $\sigma$ be a smooth odd involution on $X$ which acts as $-1$ on $H^2(X ; \mathbb{Z})$ and sends each component of $Y$ to itself. Then:
\[
\sigma(X) = -8 \delta^E_\infty(Y,\sigma).
\]
Furthermore, we have
\[
b_-(X) \ge j^E(Y,\sigma), \quad b_+(X) \ge j^E(-Y,\sigma).
\]
\end{proposition}

\subsection{Non-smoothable actions}\label{sec:nsa}

In this section we will use the obstruction results of Section \ref{sec:obex} to give examples of orientation preserving, locally linear involutions which are non-smoothable in the sense that they are not smooth with respect to any differentiable structure on the manifold.

Suppose that $X,X'$ are two topological $4$-manifolds (possibly with boundary) with orientation preserving locally linear involutions $\sigma, \sigma'$. Assume that $\sigma, \sigma'$ do not act freely and that every component of the fixed point set has codimension $2$. By an equivariant connected sum $(X, \sigma) \# (X', \sigma')$, we mean the following. Choose fixed points $x \in X$, $x' \in X'$ of the involutions $\sigma, \sigma'$, where $x,x'$ lie in the interiors of $X,X'$. Since every component of the fixed point sets have the same codimension, the involutions $\sigma, \sigma'$ have the same local form around any fixed points. Therefore we can perform the connected sum $X \# X'$ in such a way that the involutions extend to an involution $\sigma \# \sigma'$ on the connected sum. Note that the isomorphism class of the resulting involution $\sigma \# \sigma'$ will depend on the choice of fixed points $x,x'$. Note that since the fixed point sets of each summand are infinite, it is possible to construct an equivariant sum with any number of summands.

A particular case of this construction that we will be interested in is when $X' = S^2 \times S^2$ and $\sigma'$ is either $\sigma_+ = \phi \times \phi$, where $\phi\colon S^2 \to S^2$ is rotation by $\pi$ about an axis or $\sigma_- = r \times r$, where $r\colon S^2 \to S^2$ is the reflection about an equator. We have that $\sigma_+$ acts on $H^2( S^2 \times S^2 ; \mathbb{Z})$ as the identity and $\sigma_-$ acts on $H^2(S^2 \times S^2 ; \mathbb{Z})$ as minus the identity. 

Let $X_0$ be a compact, oriented, spin, smooth $4$-manifold with boundary $Y$, an integral homology sphere. Assume also that $H_1(X ; \mathbb{Z}_2) = 0$. Suppose $\sigma_0$ is a smooth, orientation preserving odd involution on $X_0$ and that $\sigma_0$ acts either as $+1$ or $-1$ on $H^2(X_0 ; \mathbb{Z})$. Then according to Proposition \ref{prop:ht}, we have $\sigma(X_0) = -8 \delta^T_\infty(Y,\sigma_0)$, where $T = R$ if $\sigma_0$ acts trivially on $H^2(X_0 ; \mathbb{Z})$ and $T = E$ if $\sigma_0$ acts as $-1$ on $H^2(X_0 ; \mathbb{Z})$. Assume that $\sigma_0$ does not act freely. Let $X_{\pm}(m)$ be the equivariant connected sum of $(X_0,\sigma_0)$ with $m$ copies of $(S^2 \times S^2 , \sigma_{\pm})$.

Now let $W$ be a closed, simply-connected, topological $4$-manifold whose intersection form is even, negative definite and has non-zero rank (for example, the negative definite $E_8$ lattice). By Freedman \cite{fre}, there exists a unique such $W$ for every even negative definite unimodular lattice. Let $x \in X_{\pm}(m)$ be such that $\sigma'(x) \neq x$, where $\sigma'$ denotes the involution on $X_{\pm}(m)$. Let $X(m) = X_{\pm}(m) \# 2W$ where the sign is $+$ if $\sigma_0$ acts trivially on $H^2(X_0 ; \mathbb{Z})$ and is $-$ if $\sigma_0$ acts as $-1$ on $H^2(X_0 ; \mathbb{Z})$. We perform the connected sum by attaching copies of $W$ at $x$ and $\sigma'(x)$. Then it is clear that $\sigma'$ can be extended to a locally linear, orientation preserving involution $\sigma$ on $X(m)$ which swaps the two copies of $W$.

\begin{proposition}\label{prop:nsi}
Suppose that $m > 3 b_2(W)/8$. Then $X(m)$ admits a smooth structure. However the locally linear involution $\sigma$ is not smooth with respect to any smooth structure on $X(m)$.
\end{proposition}
\begin{proof}
Observe that $X(m) = X_0 \# Z$ where $Z = m(S^2 \times S^2) \# 2W$. Then $Z$ is closed, simply-connected and has intersection form $mH \oplus 2L$, where $L$ is the intersection form of $W$ and $H$ is the intersection form of $S^2 \times S^2$. Since $m > 0$, we have that $m H \oplus 2L \cong m' H \oplus k(3H \oplus 2E_8)$, where $k = b_2(W)/8$ and $m' = m - 3k > 0$ by the assumption that $m > 3b_2(X)/8$. Thus $Z$ is homeomorphic to $m' (S^2 \times S^2) \# k K3$ and hence is smoothable. On the other hand, since $\sigma_0$ acts trivially on $H^2(X_0 ; \mathbb{Z})$ and since $\sigma_+$ acts trivially on $H^2(S^2 \times S^2 ; \mathbb{Z})$, it follows that $\sigma$ acts trivially on $H^+(X(m))$. Suppose that $X(m)$ admits a smooth structure in which $\sigma$ is smooth. Applying Theorem \ref{thm:froy} to the unique spin structure on $X(m)$ (which is necessarily of type $R$), we get $\delta^T_\infty(Y , \sigma_0) \ge -\sigma(X(m))/8 = -\sigma(X_0)/8 + b_2(W)/4$. But $\delta^T_\infty(Y , \sigma_0) = -\sigma(X_0)/8$, giving $b_2(W) \le 0$, which is a contradiction. Hence $\sigma$ is not smoothable.
\end{proof}

\begin{remark}
The non-smoothable involutions constructed in Proposition \ref{prop:nsi} have the following stability property. They remain non-smoothable upon equivariant connected sum with copies of $(S^2 \times S^2 , \sigma_{\pm})$, where the sign is $+$ if $\sigma_0$ acts trivially on $H^2(X_0 ; \mathbb{Z})$ and is $-$ if $\sigma_0$ acts as $-1$ on $H^2(X_0 ; \mathbb{Z})$. On the other hand, if we connect sum with $S^2 \times S^2$ equipped with a different involution, then it might be possible that the involution will become smoothable.
\end{remark}

\subsection{Equivariant embeddings}\label{sec:emb}

Let $Y$ be a rational homology $3$-sphere and $\sigma$ an orientation preserving smooth involution on $Y$. Consider the problem of embedding $Y$ into a closed, oriented, smooth $4$-manifold $X$ in such a way that $\sigma$ extends to an orientation preserving involution on $X$. Suppose we are given such an embedding $Y \to X$. Then we get a decomposition $X = X_- \cup_Y X_+$ where $X_-,X_+$ are compact smooth $4$-manifolds $\partial X_- = Y$, $\partial X_+ = -Y$. If $\sigma$ extends to $X$, then by restriction $\sigma$ acts on $X_+,X_-$ and we can apply Theorem \ref{thm:froy} to obtain constraints on the existence of such embeddings. We will focus on the case of embeddings into $S^4$ or connected sums of $S^2 \times S^2$.

The following is a straighforward extension of \cite[Proposition 7.15]{bh}:
\begin{proposition}
Let $Y$ be an integral homology $3$-sphere and $\sigma$ an orientation preserving smooth involution on $Y$. If $Y$ can be equivariantly embedded in $S^4$, then all delta-invariants of $Y$ vanish.
\end{proposition}

Every orientable $3$-manifold embeds in $\#n(S^2 \times S^2)$ for some sufficiently large $n$ \cite[Theorem 2.1]{agl}. Aceto--Golla--Larson define the embedding number $\varepsilon(Y)$ to be the least such $n$ such that $Y$ embeds in $\#n(S^2 \times S^2)$. We consider three equivariant versions of $\varepsilon$.

\begin{definition}
Let $Y$ be an orientable $3$-manifold and $\sigma$ an orientation preserving smooth involution on $Y$. Define the following invariants of $(Y,\sigma)$:
\begin{itemize}
\item[(1)]{$\varepsilon(Y,\sigma)$ is the least $n$ such that $Y$ embeds in $X = \#n(S^2 \times S^2)$ and $\sigma$ extends to an orientation preserving smooth involution on $X$. If no such $n$ exists, then we set $\varepsilon(Y, \sigma) = \infty$.}
\item[(2)]{$\varepsilon_+(Y,\sigma)$ is the least $n$ such that $Y$ embeds in $X = \#n(S^2 \times S^2)$ and $\sigma$ extends to a homologically trivial, orientation preserving smooth involution on $X$. If no such $n$ exists, then we set $\varepsilon_+(Y, \sigma) = \infty$.}
\item[(1)]{$\varepsilon_-(Y,\sigma)$ is the least $n$ such that $Y$ embeds in $X = \#n(S^2 \times S^2)$ and $\sigma$ extends to an orientation preserving smooth involution on $X$ which acts as $-1$ on $H^2(X ; \mathbb{Z})$. If no such $n$ exists, then we set $\varepsilon_-(Y, \sigma) = \infty$.}
\end{itemize}
\end{definition}

Clearly $\varepsilon_{\pm}(Y,\sigma) \ge \varepsilon(Y,\sigma) \ge \varepsilon(Y)$. Note that equivariant embedding number defined in \cite[\textsection 7.4]{bh} corresponds (in the case of involutions) to $\varepsilon_-$ in this paper.

The following results give some upper bounds on $\varepsilon,\varepsilon_{\pm}$:

\begin{proposition}\label{prop:eemb}
Suppose $(Y,\sigma)$ is given by equivariant Dehn surgery on a framed link $\mathcal{L}$ whose framing coefficients are all even integers. Then $\varepsilon(Y,\sigma) \le k$, where $k$ is the number of components of $k$. If $\mathcal{L}$ is $2$-periodic, then $\varepsilon_+(Y,\sigma) \le k$ and if $\mathcal{L}$ is strongly invertible, then $\varepsilon_+(Y,\sigma) \le k$.
\end{proposition}
\begin{proof}
Let $X$ denote the trace of the surgery on $\mathcal{L}$ and let $D(X) = X \cup_Y (-X)$ be the double of $X$. Then $Y$ embeds equivariantly in $X$. But $X$ is diffeomorphic to $\#k(S^2 \times S^2)$ \cite[Corollary 5.1.6]{gs}. Thus $\varepsilon(Y,\sigma) \le k$. Furthermore, if $\mathcal{L}$ is $2$-periodic then $\sigma$ acts trivially on $H^2( D(X) ; \mathbb{Z})$, hence $\varepsilon_+(Y,\sigma) \le k$. Similarly, if $\mathcal{L}$ is strongly invertible then $\sigma$ acts as $-1$ on $H^2(D(X) ; \mathbb{Z})$, hence $\varepsilon_-(Y , \sigma) \le k$.
\end{proof}

\begin{proposition}\label{prop:epplumb}
Let $\Gamma$ be a connected plumbing graph with all vertices having even degrees. Then $\varepsilon_-(Y_\Gamma , c_\Gamma) \le |\Gamma|$, where $|\Gamma|$ denotes the number of vertices in $\Gamma$. If additionally $\Gamma$ is a $\mathbb{Z}_2$-equivariant plumbing graph, then $\varepsilon_+(Y_\Gamma , m_\Gamma) \le |\Gamma|$.
\end{proposition}
\begin{proof}
$Y_\Gamma$ equivariantly embeds in the double $D(X_\Gamma)$ of the plumbing $X_\Gamma$. By \cite[Corollary 5.1.6]{gs}, we have that $D(X_\Gamma)$ is diffeomorphic to $\#|\Gamma|(S^2 \times S^2)$. Then the result follows by noting that $c_\Gamma$ acts as $-1$ on $H^2( D(X_\Gamma) ; \mathbb{Z})$ and $m_\Gamma$ acts as $+1$.
\end{proof}

We also have \cite[Proposition 7.18]{bh}:

\begin{proposition}\label{prop:ds}
Let $K$ be a knot in $S^3$ and $Y = \Sigma_2(K)$ the branched double cover with covering involution $\sigma$. Then $\varepsilon_-(Y , \sigma) \le g_{ds}(K)$, where $g_{ds}(K)$ is the double slice genus of $K$ \cite[\textsection 5]{lime}.
\end{proposition}

Now we use Theorem \ref{thm:froy} to obtain lower bounds on $\varepsilon, \varepsilon_{\pm}$:
\begin{proposition}\label{prop:embbound}
Let $Y$ be an integral homology $3$-sphere and $\sigma$ an orientation preserving smooth involution on $Y$. We have:
\begin{itemize}
\item[(1)]{If the delta-invariants $\delta^T_*$ of $(Y,\sigma)$ do not all vanish then $\varepsilon(Y,\sigma) \ge 2$.}
\item[(2)]{$\varepsilon_+(Y,\sigma) \ge \max\{ j^R(Y , \sigma) , 2j^R(Y,\sigma) - 8 \delta^R_\infty(Y,\sigma) \}$.}
\item[(3)]{$\varepsilon_-(Y,\sigma) \ge \max\{ j^E(Y , \sigma) , 2j^E(Y,\sigma) - 8 \delta^E_\infty(Y,\sigma) \}$.}
\item[(4)]{If $\delta^S_{0,1}(Y,\sigma), \delta^S_{1,0}(Y,\sigma)$ are both non-zero, then $\varepsilon(Y,\sigma) \ge 4$.}
\item[(5)]{If the Rokhlin invariant of $Y$ is non-zero and if $\delta^S_{0,\infty}(Y,\sigma), \delta^S_{\infty,0}(Y,\sigma), \delta^S_{\infty,1}(Y,\sigma)$ do not all have the same sign, then $\varepsilon(Y,\sigma) \ge 10$.}
\end{itemize}
\end{proposition}
\begin{proof}
Suppose that $Y$ embeds equivariantly in $X = \#n(S^2 \times S^2)$. Then $Y = X_- \cup_Y X_+$. Let $L^+,L^-$ denote the intersection forms on $X_+,X_-$. Since $Y$ is an integral homology sphere, $L^+,L^-$ are even unimodular lattices and $L^+ \oplus L^- \cong H^2(X ; \mathbb{Z}) \cong nH$ where $H$ is the hyperbolic lattice.

If $n < 2$, then $L^+$ or $L^-$ is zero. Applying Theorem \ref{thm:froy} to $X_+$ if $L^+$ is zero or $X_-$ if $L^-$ is zero, we see that all the delta-invariants of $(Y,\sigma)$ must vanish. This proves (1).

Suppose now that $\sigma$ acts trivially on $H^2(X ; \mathbb{Z})$, hence also on $H^2(X_{\pm} ; \mathbb{Z})$. Let $L^-$ have signature $(a,b)$. Then $L^+$ has signature $(n-a,n-b)$. Applying Theorem \ref{thm:froy} to $X_-$ gives $\delta^R_j(Y,\sigma) = (b-a)/8$ for $j \ge b$. Hence $\delta^R_\infty(Y,\sigma) = (b-a)/8$ and $b \ge j^R(Y,\sigma)$. Applying Theorem \ref{thm:froy} to $X_+$ gives $\delta^R_j(Y,\sigma) = (b-a)/8$ for $j \ge n-a$, hence $n-a \ge j^R(Y,\sigma)$. Hence $n = a+b \ge b \ge j^R(Y,\sigma)$ and
\[
n = 2b + (a-b) \ge 2j^R(Y,\sigma)-8\delta^R_j(Y,\sigma).
\]
Hence $n \ge \max\{ j^R(Y,\sigma) , 2j^R(Y,\sigma) - 8\delta^R_j(Y,\sigma) \}$ proving (2). In the case that $\sigma$ acts as $-1$ on $H^2(X ; \mathbb{Z})$ an identical argument proves (3).

Suppose $\varepsilon(Y,\sigma) < 4$. Then one of $L^+$ or $L^-$ is either zero or $H$. Reversing orientation on $X$ if necessary, we can assume $L^+$ is either zero or $H$. If $L^+$ is zero then $\delta^S_{0,1}(Y,\sigma) = \delta^S_{1,0}(Y,\sigma) = 0$ by (1). Now suppose $L^+ = H$. Consider the action of $\sigma$ on $H$. There are four possible involutions on $H$. In each of the four cases one finds (by applying Theorem \ref{thm:froy} to $X_+$) that either $\delta^S_{0,1}(Y,\sigma)=0$ or $\delta^S_{1,0}(Y,\sigma) = 0$. This proves (4).

Lastly, suppose the Rokhlin invariant of $Y$ is non-zero and $n \le 9$. Then either $L^+$ or $L^-$ must be $\pm E_8$, where $E_8$ denotes the negative definite $E_8$ lattice. Reversing orientation on $X$ if necessary, we can assume $L^+$ or $L^-$ is $E_8$. If $L^+ = E_8$, applying Theorem \ref{thm:froy} to $X_+$ gives $\delta^S_{i,j}(Y,\sigma) \ge 1$ for all $i,j$. Similarly if $L^- = E_8$, applying Theorem \ref{thm:froy} to $X_-$ gives $\delta^S_{i,j}(Y,\sigma) \le -1$ for all $i,j$. In either case we see that $\delta^S_{0,\infty}(Y,\sigma), \delta^S_{\infty,0}(Y,\sigma), \delta^S_{\infty,1}(Y,\sigma)$ all have the same sign. This proves (5).
\end{proof}

In \cite{bar}, we introduced a concordance invariant $\theta^{(2)}(K)$ of knots which is defined by 
\[
\theta^{(2)}(K) = \max\{ 0 , j^E( \Sigma_2(K) , \sigma) - \sigma(K)/2 \},
\]
where $\sigma$ is the covering involution on $\Sigma_2(K)$. But note that $\sigma(K) = 8\delta^E_\infty( -\Sigma_2(K) , \sigma)$, so it follows from Proposition \ref{prop:embbound} (3) and Proposition \ref{prop:ds} that
\begin{equation}\label{equ:varbound}
2\theta^{(2)}(K) \le \varepsilon_-( \Sigma_2(K) , \sigma) \le g_{ds}(K).
\end{equation}

\begin{proposition}\label{prop:tpq}
Let $p,q$ be odd, positive, coprime integers. Then $\varepsilon_-( \Sigma_2( T_{p,q}) , \sigma ) = (p-1)(q-1)$.
\end{proposition}
\begin{proof}
For any knot $K$, we have $2g_4(K) \le g_{ds}(K) \le 2g_3(K)$, where $g_3$ denotes the $3$-genus. For a torus knot $T_{p,q}$, the $3$-genus and $4$-genus agree and hence $g_{ds}(T_{p,q}) = 2g_4(T_{p,q}) = (p-1)(q-1)$. On the other hand, \cite[Theorem 1.1]{bh2} gives $2\theta^{(2)}(T_{p,q}) = (p-1)(q-1)$. So Equation \ref{equ:varbound} gives $\varepsilon_-( \Sigma_2(T_{p,q}) , \sigma) = (p-1)(q-1)$.
\end{proof}

\begin{proposition}
Let $\Gamma$ be a negative definite plumbing graph with all vertices having even degree and $\det(A(\Gamma)) = 1$. Then $\varepsilon_-(Y_\Gamma , c_\Gamma) = |\Gamma|$, where $|\Gamma|$ is the number of vertices of $\Gamma$. If in addition $\Gamma$ is a $\mathbb{Z}_2$-equivariant plumbing graph, then $\varepsilon_+(Y_\Gamma , m_\Gamma) = |\Gamma|$.
\end{proposition}
\begin{proof}
Proposition \ref{prop:epplumb} gives the bound $\varepsilon_-(Y_\Gamma , c_\Gamma) \le |\Gamma|$ and Proposition \ref{prop:embbound} gives the bound $\varepsilon_-(Y_\Gamma , c_\Gamma) \ge 2j^E(-Y,\sigma) - 8\delta^E_\infty(-Y,\sigma) \ge -8\delta^E_\infty(-Y,\sigma) = -8\overline{\mu}(Y_\Gamma)$. But $\Gamma$ is negative definite with all degrees even, so $\overline{\mu}(Y_\Gamma) = -|\Gamma|/8$, hence $\varepsilon_-(Y_\Gamma , c_\Gamma) = |\Gamma|$. The case of $\varepsilon_+(Y_\Gamma , m_\Gamma)$ is similar.
\end{proof}

\begin{example}\label{ex:epsbri}
We consider the embedding numbers of some Brieskorn spheres equipped either of the involutions $\sigma$ or $c$.

$Y = \Sigma(2,3,5)$. Since $\mu(Y)=1$, all embedding numbers are at least $8$. But $Y$ is the boundary of the $E_8$-plumbing over which both $\sigma$ and $c$ extend. So we get $\varepsilon(Y) = \varepsilon(Y,c) = \varepsilon(Y,m) = \varepsilon_+(Y,m) = \varepsilon_-(Y,c) = 8$. Since $Y = \Sigma_2(T_{3,5})$ with $m$ being the covering involution, Proposition \ref{prop:tpq} also gives $\varepsilon_-(Y,m) = 8$. We do not know the value of $\varepsilon_-(Y,c)$.

$Y = \Sigma(2,3,7)$. Then $\varepsilon(Y) = 10$ \cite[Proposition 3.5]{agl}. On the other hand, $Y$ is the boundary of a plumbing graph shown in Figure \ref{fig:237}. Hence $\varepsilon_+(Y,m) \le 10$ and $\varepsilon_-(Y,c) \le 10$. So this gives $\varepsilon(Y,m) = \varepsilon(Y,c) = \varepsilon_+(Y,m) = \varepsilon_-(Y,c) = 10$. Proposition \ref{prop:tpq} gives $\varepsilon_-(Y,m) = 12$. We do not know the value of $\varepsilon_+(Y,c)$.

\begin{figure}[h]
\begin{center}
\begin{tikzpicture}
\draw[thick] (-2,0) -- (6,0) ;
\draw[thick] (0,0) -- (0,-1);
\draw (-2,0)[fill] circle(0.1);
\draw (-1,0)[fill] circle(0.1);
\draw (0,0)[fill] circle(0.1);
\draw (1,0)[fill] circle(0.1);
\draw (2,0)[fill] circle(0.1);
\draw (0,-1)[fill] circle(0.1);
\draw (3,0)[fill] circle(0.1);
\draw (4,0)[fill] circle(0.1);
\draw (5,0)[fill] circle(0.1);
\draw (6,0)[fill] circle(0.1);

\node at (-2,0.5) {$2$};
\node at (-1,0.5) {$2$};
\node at (0,0.5) {$2$};
\node at (1,0.5) {$2$};
\node at (2,0.5) {$2$};
\node at (0,-1.5) {$2$};
\node at (3,0.5) {$2$};
\node at (4,0.5) {$2$};
\node at (5,0.5) {$2$};
\node at (6,0.5) {$2$};
\end{tikzpicture}
\caption{Plumbing graph for $\Sigma(2,3,7)$}\label{fig:237}
\end{center}
\end{figure}
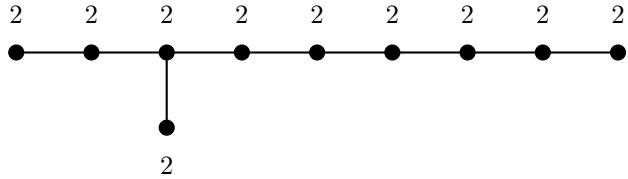

$Y = \Sigma(2,3,13)$. It is known that $Y$ embeds in $S^4$ \cite[Theorem 2.13]{bb}, so $\varepsilon(Y) = 0$. On the other hand, $\delta^E_\infty(Y,m) = -\sigma(T_{3,13})/8 = 2$, so Proposition \ref{prop:embbound} (1) gives $\varepsilon(Y,m) \ge 2$. Furthermore, Proposition \ref{prop:tpq} gives $\varepsilon_-(Y , m) = 24$. Next, since $\Sigma(2,3,13) = Y(0 ; (2,-1) , (3,2) , (13,-2))$ we can write $Y = Y_\Gamma$ where $\Gamma$ is given by Figure \ref{fig:2313}. This graph has $6$ vertices, so $\varepsilon_+(Y,m) \le 6$, $\varepsilon_-(Y,c) \le 6$. Furthermore, we observe that $Y = S_{-1/2}(T_{2,3})$. The strong inversion on $T_{2,3}$ corresponds to the involution $c$ on $Y$ (it can not correspond to $m$ since $\delta^E_\infty(Y , m) \neq 0$, which would contradict Proposition \ref{prop:12p}). Using the slam dunk move we can write $Y$ as surgery on a two component link with even surgery coefficients. Therefore $\varepsilon_-(Y , c) \le 2$. Notice that $\varepsilon(Y) = 0$, $\varepsilon(Y,m) \in [2,6]$, $\varepsilon_-(Y,m) = 24$, so $\varepsilon(Y), \varepsilon(Y,m), \varepsilon_-(Y,m)$ take distinct values.

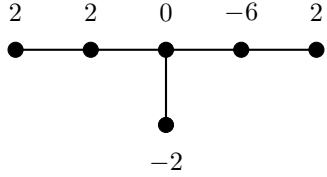
\begin{figure}[h]
\begin{center}
\begin{tikzpicture}
\draw[thick] (-2,0) -- (2,0) ;
\draw[thick] (0,0) -- (0,-1);
\draw (-2,0)[fill] circle(0.1);
\draw (-1,0)[fill] circle(0.1);
\draw (0,0)[fill] circle(0.1);
\draw (1,0)[fill] circle(0.1);
\draw (2,0)[fill] circle(0.1);
\draw (0,-1)[fill] circle(0.1);

\node at (-2,0.5) {$2$};
\node at (-1,0.5) {$2$};
\node at (0,0.5) {$0$};
\node at (1,0.5) {$-6$};
\node at (2,0.5) {$2$};
\node at (0,-1.5) {$-2$};

\end{tikzpicture}
\caption{Plumbing graph for $\Sigma(2,3,13)$}\label{fig:2313}
\end{center}
\end{figure}

\end{example}

\begin{example}
Let $Y$ be the equivariant connected sum $\Sigma(2,3,5) \# - \! \!\Sigma(2,3,13)$ with involution $\sigma$ the connected sum of the $m$ involutions on the two summands. Then since $\varepsilon(\Sigma(2,3,5)) = 8$, $\varepsilon(\Sigma(2,3,13))=0$, if follows that $\varepsilon(Y) \le 8$. On the other hand $\varepsilon(Y) \ge 8$ since $\mu(Y) = 1$. So $\varepsilon(Y) = 8$. In Example \ref{ex:epsbri} we saw that $\varepsilon(\Sigma(2,3,5) , m) = 8$ and $\varepsilon(\Sigma(2,3,13) , m) \le 6$, so $\varepsilon(Y,\sigma) \le 14$. We claim that $\varepsilon(Y,\sigma) \ge 10$. Suppose this were not the case. Then as in the proof of Proposition \ref{prop:embbound} (5), we get $X = X_- \cup_Y X_+$ where $L^+$ or $L^-$ is $E_8$. If $L^+ = E_8$, then Theorem \ref{thm:froy} gives $\delta^E_0(Y,\sigma) \le -1$. But $\delta^E_0(Y,\sigma) \ge \delta(Y) = \delta(\Sigma(2,3,5)) - \delta(\Sigma(2,3,13)) = 1$, a contradiction. If $L^- = E_8$, then Theorem \ref{thm:froy} gives $\delta^E_\infty(Y,\sigma) \ge 1$. But $\delta^E_\infty(Y,\sigma) \le \delta^E_\infty(-\Sigma(2,3,13),m) + \delta^E_\infty(\Sigma(2,3,5),m) = -2+1 = -1$, a contradiction. So we have shown that $\varepsilon(Y) = 8$ and $10 \le \varepsilon(Y,\sigma) \le 14$.
\end{example}

\begin{example}
We will show that the differences $\varepsilon_+(Y,\sigma) - \varepsilon(Y)$, $\varepsilon_-(Y,\sigma) - \varepsilon(Y)$ can be arbitrarily large. We do not know whether this is also true for $\varepsilon(Y,\sigma) - \varepsilon(Y)$.

Take $Y = \Sigma(2,3,12n-1)$ with involution $\sigma = m$. Then $Y = -S_{1/2n}(T_{2,3})$ and as in Example \ref{ex:epsbri}, this gives $\epsilon(Y) \le 2$. On the other hand, Proposition \ref{prop:tpq} gives $\varepsilon_-(Y , \sigma) = 24n-4$. 

Finding examples where $\varepsilon_+(Y,\sigma) - \varepsilon(Y)$ is large requires a little more effort. Let $n_1, \dots , n_k$ be positive integers and take $Y = \#_{i=1}^{k} \Sigma(2,3,12n_i-1)$ to be an equivariant connected sum of $\Sigma(2,3,12n_1-1), \dots , \Sigma(2,3,12n_k-1)$, where we take the $m$ involution on each summand. We have $l(\Sigma(2,3,12n_i-1)) = 1$ and hence $l(Y) \ge k$. Also $\delta^R_\infty(Y,\sigma) \le \sum_{i=1}^l \delta^R_\infty( \Sigma(2,3,12n_i-1),m) = -\sum_{i=1}^k \overline{\mu}(\Sigma(2,3,12n_i-1)) = 0$. Hence Proposition \ref{prop:jbound}, $j^R(Y,\sigma) \ge k$. On the other hand, $j^R(Y,\sigma) \le \sum_{i=1}^{k} j^R(\Sigma(2,3,12n_i-1),m) = k$, so $j^R(Y,\sigma) = k$. Hence $\varepsilon_+(Y,\sigma) \ge 2k$ by Proposition \ref{prop:embbound} (2).
\end{example}

\subsection{Non-orientable surfaces bounding knots}\label{sec:nos}

Let $K$ be a knot in $S^3$ and let $S$ be a connected, properly embedded, non-orientable surface in $D^4$ which bounds $K$. In this section we consider a constraint on the topology of $S$ obtained from the type $R$ delta-invariants of $K$. A similar application was considered in \cite{kmt2}. The main difference is that our invariant can be calculated for a much larger class of knots, in particular for all quasi-alternating knots.

Let $e(S)$ denote the relative Euler class of $S$ with respect to the zero framing on $K$. Since $e(S)$ is valued in the orientation local system, will identify $e(S)$ with an integer. Note that $e(S)$ is always even becauses its value mod $2$ is the mod $2$ self-intersection number of $S$, which is zero since $H^2(D^4 , S^3 ; \mathbb{Z}_2) = 0$. A natural question to ask is for a given $K$, what possible values of $( e(S) , b_1(S) )$ can be attained? This problem is studied for torus knots in \cite{all}. Here $b_1(S)$ is the first Betti number of $S$. The minimum possible value of $b_1(S)$ that such an $S$ can attain is known as {\em the non-orientable $4$-genus of $S$} \cite{gl} and denoted by $\gamma_4(K)$.

\begin{proposition}\label{prop:xy}
Let $x = \sigma(K) - e(S)/2$ and $y = b_1(S)$. Then $x,y \in \mathbb{Z}$, $y \ge 0$, $|x| \le y$ and $x = y \; ({\rm mod} \; 2)$.
\end{proposition}
\begin{proof}
Since $e(S)$ is even, $x$ is an integer. We have assumed that $S$ is non-orientable, so $y = b_1(S) > 1$ (any non-orientable surface with one boundary component has positive first Betti number).

Let $X$ denote the double cover of $D^4$ branched over $S$. Then $X$ is a smooth, compact, oriented $4$-manifold with boundary $Y = \partial(X) = \Sigma_2(K)$. By a straightforward computation (eg, \cite[Lemma 4.5]{kmt2}) we have:
\begin{equation}\label{equ:bpm}
b_+(X) = \frac{x+y}{2}, \quad b_-(X) = \frac{y-x}{2}.
\end{equation}
The result follows, since $b_{\pm}(X)$ are non-negative integers.
\end{proof}

Suppose the surface $S$ has $x = \sigma(K) - e(S)/2$, $y = b_1(S)$. By connect summing $S$ with an embedded copy of $\mathbb{RP}^2$ in $S^4$, we can obtain a new surface $S'$ with values $(x \pm 1, y+1)$. This operation can not increase the value of $y-|x|$. Of particular interest is the boundary case where $|x| = y$. In this case double cover $X$ of $D^4$ branched over $S$ is positive (if $x=y$) or negative definite (if $x=-y$).

\begin{proposition}\label{prop:noxy}
Suppose that $x=-y$, or equivalently $\sigma(K) - e(S)/2 = -b_1(S)$. Then there exists a spin$^c$-structure $\mathfrak{s}$ on $\Sigma_2(K)$ for which $\delta^R_\infty( \Sigma_2(K) , \mathfrak{s} , \sigma) \ge 0$ and $\delta^R_0(-\Sigma_2(K) , \mathfrak{s} , \sigma) \le 0$, where $\sigma$ is the covering involution on $\Sigma_2(K)$.  Furthermore, if $\delta^R_\infty( \Sigma_2(K) , \mathfrak{s} , \sigma) = 0$ or $\delta^R_0(-\Sigma_2(K) , \mathfrak{s} , \sigma) = 0$, then $\mathfrak{s}$ is the unique spin structure on $\Sigma_2(K)$.

If $K$ is quasi-alternating (or more generally, if $\Sigma_2(K)$ is an $L$-space), then there exists a spin$^c$-structure $\mathfrak{s}$ on $\Sigma_2(K)$ for which $\delta( \Sigma_2(K) , \mathfrak{s} ) \ge 0$, with equality only if $\mathfrak{s}$ is the spin structure.
\end{proposition}
\begin{proof}
Let $X$ be the double cover of $D^4$ branched over $S$. Since $x=-y$, $X$ is negative definite. The covering involution $\sigma$ on $Y = \Sigma_2(K)$ extends over $X$ as the covering involution of the double cover $\pi\colon X \to D^4$. We have that $1 + \sigma^* = \pi^* \pi_* = 0$, since $H^2(D^4 ; \mathbb{Z}) = 0$. Hence $\sigma$ acts as $-1$ on $H^2(X ; \mathbb{Z})$. Furthermore, $H^2(X ; \mathbb{Z})$ has no $2$-torsion \cite[Lemma 4.5]{kmt2}, so we can identify spin$^c$-structures on $X$ with their characteristics. We also have $b_1(X) = 0$. It follows that every spin$^c$-structure on $X$ has type $R$. From the long exact sequence in cohomology of the pair $(X,Y)$, we have
\[
0 \to H^2(X , Y ; \mathbb{Z}) \to H^2(X ; \mathbb{Z}) \to H^2(Y ; \mathbb{Z}).
\]
Let $L = H^2(X , Y ; \mathbb{Z})/{tors}$ denote the intersection lattice of $X$. Then it follows that the discriminant group $\bar{L} = L^*/L$ is isomorphic to a subquotient of $H^2(Y ; \mathbb{Z})$. Moreover, $| H^2(Y ; \mathbb{Z}) | = \det(K)$, hence $| \bar{L} |$ divides $\det(K)$. In particular, $|\bar{L}|$ is odd.

Choose a spin$^c$-structure $\mathfrak{s}$ on $X$ whose characteristic $c \in H^2(X ; \mathbb{Z})$ attains the maximum possible value of $c^2$ (such a spin$^c$-structure exists, since $X$ is negative definite). By \cite[Theorem 1]{owst2}, we have $c^2 \ge 1-n-1/\delta$, where $n$ is the rank of $L$ and $\delta = | \bar{L} |$. Hence $(c^2 - \sigma(X))/8 \ge (1/8)(1 - 1/\delta) \ge 0$. Now the result follows by applying Theorem \ref{thm:froy} to $(X , \mathfrak{s} , \sigma)$. Furthermore the inequalities $\delta^R_\infty( \Sigma_2(K) , \mathfrak{s} , \sigma) \ge 0$ and $\delta^R_0(-\Sigma_2(K) , \mathfrak{s} , \sigma) \le 0$ are strict unless $\delta = 1$. But in this case $L$ is unimodular and it follows that $\mathfrak{s}|_Y$ is the unique spin structure on $Y$.

The last statement also follows, for if $Y$ is an $L$-space, then $\delta^R_\infty( Y , \mathfrak{s} , \sigma) = \delta(Y , \mathfrak{s})$ by Proposition \ref{prop:deltaprop} (4).
\end{proof}

\begin{example}
Let $K = T_{p,q}$ be a torus knot, where $p$ and $q$ are odd. Then $\det(K) = 1$, so $Y = \Sigma_2(K) = \Sigma(2,p,q)$ has a unique spin$^c$-structure. Suppose that a non-orientable surface $S$ bounding $K$ has $x = -y$, where $x = \sigma(K) - e(S)/2$, $y = b_1(S)$. Then by Proposition \ref{prop:noxy}, $\delta^R_\infty(Y , \sigma) \ge 0$. In this case $\delta^R_\infty(Y , \sigma) = -\overline{\mu}( \Sigma(2,p,q) )$. So we get $\overline{\mu}(\Sigma(2,p,q)) \le 0$. Similarly, suppose there is a surface $S$ with $x = y$. Then applying Proposition \ref{prop:noxy} to $-K$ we get $\overline{\mu}(\Sigma(2,p,q)) \ge 0$. So if $\overline{\mu}(\Sigma(2,p,q)) > 0$, then $x > -y$ and if $\overline{\mu}(\Sigma(2,p,q)) < 0$, then $x < y$.
\end{example}

\begin{example}
Let $K = M( e ; (a_1,b_1) , \dots , (a_n , b_n))$ be a Montesinos knot where $a_1, \dots , a_n$ are coprime, $a_i$ is even for some $i$ and $e - \sum_{i=1}^{n} b_i/a_i = 1/(a_1\cdots a_n)$. Then $\Sigma_2(K) = -Y$, where $Y$ is the Brieskorn sphere $\Sigma(a_1, \dots , a_n)$. Then by Theorem \ref{thm:briedelt}, we have $\delta^R_\infty( \Sigma_2(K) , \sigma ) \le -\sigma(K)/8 = \overline{\mu}(Y)$. So if $\overline{\mu}(Y) < 0$, then $\delta^R_\infty( \Sigma_2(K) , \sigma) < 0$. Applying Proposition \ref{prop:noxy} we see that if $\overline{\mu}(Y) < 0$, then there does not exist a non-orientable surface $S$ bounding $K$ with $\sigma(K) - e(S)/2 = -b_1(S)$.
\end{example}


\bibliographystyle{amsplain}

\end{document}